\documentclass[10pt]{amsart}
\usepackage{amsfonts,amssymb}
\RequirePackage{ifpdf}
\usepackage{amsmath}
\usepackage{color}
\usepackage{pstcol,pst-node}
\usepackage{graphics}
\usepackage{epic}
\usepackage{curves}

\def\appendix#1{
\addtocounter{section}{1} \setcounter{equation}{0}
\renewcommand{\thesection}{\Alph{section}}
\section*{Appendix \thesection\protect\indent\quad
#1}
}
\renewcommand{\theequation}{\thesection.\arabic{equation}}

\catcode`\@=11
\def\marginnote#1{}

\newcount\hour
\newcount\minute
\newtoks\amorpm
\hour=\time\divide\hour by60 \minute=\time{\multiply\hour by60
\global\advance\minute by-\hour}
\edef\standardtime{{\ifnum\hour<12 \global\amorpm={am}%
        \else\global\amorpm={pm}\advance\hour by-12 \fi
        \ifnum\hour=0 \hour=12 \fi
        \number\hour:\ifnum\minute<10 0\fi\number\minute\the\amorpm}}
\edef\militarytime{\number\hour:\ifnum\minute<100\fi\number\minute}

\let\hhat=\widehat
\let\ttilde=\widetilde

\def\draftlabel#1{{\@bsphack\if@filesw {\let\thepage\relax
      \xdef\@gtempa{\write\@auxout{\string
          \newlabel{#1}{{\@currentlabel}{\thepage}}}}}\@gtempa \if@nobreak
    \ifvmode\nobreak\fi\fi\fi\@esphack} \gdef\@eqnlabel{#1}}
    \def\@eqnlabel{}
\def\@vacuum{}
\def\draftmarginnote#1{\marginpar{\raggedright\scriptsize\tt#1}}

\def\draft{
  \oddsidemargin -.5truein
  \def\@oddfoot{\footnotesize \sl preliminary draft \hfil
    \rm\thepage\hfil\sl\today\quad\militarytime}
  \let\@evenfoot\@oddfoot \overfullrule 3pt
    \let\label=\draftlabel
    \let\marginnote=\draftmarginnote
  \def\@eqnnum{(\theequation)\rlap{\kern\marginparsep\tt\@eqnlabel}%
    \global\let\@eqnlabel\@vacuum}

  }

\newcommand{\tr}{\,{\rm Tr}\,}
\newcommand{\ID}{1\!\!1}

\def\be{\begin{equation}}
\def\ee{\end{equation}}
\def\bea{\begin{eqnarray}}
\def\eea{\end{eqnarray}}
\def\<{\langle}
\def\>{\rangle}
\def\ddz{\frac{\rm d}{{\rm d}z}}

\def\nn{\nonumber}

\def\Tr{{\rm Tr}}
\def\one#1{#1^{\raise5pt\hbox{$\scriptstyle\!\!\!\!1$}}\,{}}
\def\two#1{#1^{\raise5pt\hbox{$\scriptstyle\!\!\!\!2$}}\,{}}
\def\onetwo#1{#1^{\raise5pt\hbox{$\scriptstyle\!\!\!\!\!{12}$}}\,{}}
\def\otim{\mathop{\otimes}}
\def\thi{\vartheta}

\newtheorem{theorem}{Theorem}[section]
\newtheorem{lm}[theorem]{Lemma}
\newtheorem{prop}[theorem]{Proposition}
\theoremstyle{definition}
\newtheorem{df}[theorem]{Definition}
\newtheorem{example}[theorem]{Example}
\newtheorem{remark}[theorem]{Remark}

\newtheorem{cor}[theorem]{Corollary}
\theoremstyle{remark}

\begin{document}
\title[Isomonodromic deformations and  Yangians in Teichm\"uller theory]
{Isomonodromic deformations and twisted Yangians arising in Teichm\"uller theory}
\author{Leonid Chekhov and Marta Mazzocco}

\maketitle

\begin{abstract}
In this paper we build a link between the Teichmuller theory of
hyperbolic Riemann surfaces and isomonodromic deformations of linear
systems whose monodromy group is the Fuchsian group associated to
the given hyperbolic Riemann surface by the Poincar\'e
uniformization. In the case of a one--sheeted hyperboloid with $n$
orbifold points we show that the Poisson algebra ${\mathfrak D}_n$ of
geodesic length functions is the semiclassical limit of the twisted
$q$--Yangian $Y'_q(\mathfrak{o}_n)$ for the orthogonal Lie algebra
$\mathfrak{o}_n$ defined by Molev, Ragoucy and Sorba. We give a
representation of the braid group action on ${\mathfrak D}_n$ in
terms of an adjoint matrix action. We characterize two types of
finite--dimensional Poissonian reductions and give an explicit
expression for the generating function of their central elements.
Finally, we interpret the algebra ${\mathfrak D}_n$  as the Poisson
algebra of monodromy data of a Frobenius manifold in the vicinity of
a non-semisimple point.
\end{abstract}

\tableofcontents

\section{Introduction}

In recent years Teichm\"uller theory has attracted interest from the
mathematical physics comminity due to the manifestation of the
Teichm\"uller space as the Hilbert space for three-dimensional
quantum gravity~\cite{VV}. The Teichm\" uller space possesses its
canonical (Weil--Petersson) Poisson structure, whose symmetry group
is the mapping class group of orientation-preserving homeomorphisms
modulo isotopy.  The algebra of observables is the collection of
length functions of geodesic representatives of homotopy
classes of essential closed curves together with its natural mapping
class group action.

Algebras of geodesic length functions appearing in studies of
Teichm\"uller spaces of hyperbolic Riemann surfaces are closely
related to those appearing in isomonodromy problems. For example,
the Nelson--Regge algebra~\cite{NR},~\cite{NRZ} appearing as the
algebra of geodesic--length--functions on a genus $g$ Riemann
surface with $1$ or $2$ holes~\cite{ChF2},~\cite{ChP}, and its isomorphic algebras
$A_n$ of geodesic functions on a disk with $n$
orbifold points \cite{Ch2}, coincide with the Poisson algebras of
monodromies in Fuchsian systems arising in Frobenius manifold theory
\cite{Ugaglia} and algebras of groupoid of upper triangular matrices
\cite{Bondal}.

This coincidence between algebras of geodesic length functions
appearing in Teichm\"uller theory and   algebras of monodromy data
of isomonodromic systems remained a mystery so far. In this paper we
characterize a natural isomonodromic connection on the punctured
$\mathbb P^1$ (the Chern--Simons connection) whose monodromy group
is given by the Fuchsian group of a disk with $n$
orbifold points. This shows that the $A_n$ algebras coincide with the
Poisson algebras of the monodromy data of a $2\times2$ Fuchsian
system with $n+1$ poles.

We then generalize this correspondence and introduce a new type of
Poisson algebras, whose geometrical origin are algebras of geodesic
functions on a one-sheeted hyperboloid (or topologically an annulus)
with $n$ orbifold points. On the analytical side, we can obtain the
corresponding Fuchsian system starting from an $A_{n+m}$-system and
clashing $m$ regular singularities to produce a new hole. The
corresponding Poisson algebra of monodromy data is now a quadratic
algebra independent on the number $m$ of the clashed poles that can
be interpreted as an  abstract algebra for infinitely many
generators $G_{i,j}^{(k)}$, $i,j=1,\dots,n$, $k\in\mathbb Z_{\geq
0}$ (see theorem \ref{th:main}). We call this algebra the
${\mathfrak D}_n$ algebra. At {\it level $0$,}\/ i.e. for the
generators $G_{i,j}^{(0)}$, $i,j=1,\dots,n$, this algebra restricts
to the Nelson Regge algebra.

We show that the ${\mathfrak D}_n$ algebra is the semiclassical
limit of the twisted $q$-Yangian $Y'_q(\mathfrak{o}_n)$ for the
orthogonal Lie algebra $\mathfrak{o}_n$ \cite{MRS}, or, in other
words, the defining relations of ${\mathfrak D}_n$ algebra are the
semiclassical limit of the well-known reflection equation.

Beside the Poisson structure, another common property of the
algebras of geodesic--length--functions on a Teichm\"uller space are
the braid-group relations, which generate the mapping class group.
The braid group invariants are simultaneously the central elements
of the Poisson algebra, therefore constructing a convenient
representation of the braid group is always helpful in finding the
central elements of the Poisson algebra. Such a representation in
terms of the adjoint matrix action for the $A_n$ algebras was
constructed in \cite{dubrovin} and was used in \cite{Bondal} for
constructing the Poisson invariants of the corresponding algebra. On
the analytic side, the action of the braid group corresponds to the
analytic continuation of the solutions to the isomonodromic problem
\cite{DM}.

In this paper we give a representation of the braid group action on
${\mathfrak D}_n$ in terms of an adjoint matrix action (see
proposition \ref{lem-braid}). Due to topological considerations,
this action is represented by $n$ generators, the generators
$\beta_{i,i+1}$, $i=1,\dots,n-1$, interchanging the $i$-th orbifold
point with the $(i+1)$-th one, and a new generator $\beta_{n,1}$
interchanging the first and the last orbifold points from the other
side of the new hole. This new generator acts in a non trivial way
mixing different  levels.

We characterize two types of finite-dimensional Poissonian
reductions, the so-called {\it level $p$ reductions}\/ and the {\it
$D_n$ reduction,}\/ and give an explicit expression for the
generating function of their central elements.\footnote{We draw the
attention of the reader to the fact that throughout this paper we
deal with two distinct objects: the $\mathfrak D_n$ algebra and the
$D_n$ algebra.}  Let us briefly describe these two reductions from a
geometric point of view. The level $p$ reduction corresponds to
collapsing the newly created hole to an orbifold point of order $p$.
The $D_n$ reduction is the reduction to a finitely generated cubic
algebra produced in \cite{Ch1}, where  the corresponding braid-group
action was also constructed. However, a procedure for finding
central elements (or the braid-group invariants of this algebra) was
lacking in \cite{Ch1}. We fill this gap in this paper.

The quantum braid-group action representation for the $D_n$
algebra was found in \cite{Ch1},~\cite{Ch2}. Since the reduction of
the ${\mathfrak D}_n$ algebra to the $D_n$ algebra can be presented
in the matrix form, it is clear that the very same representation of the
quantum braid group must be simultaneously a representation for the
quantum braid group (or  quantum mapping class group) for the
${\mathfrak D}_n$ algebra as well as for all its $p$-level (quantum)
reductions. Using this insight we show that the subgroup generated
by $\beta_{i,i+1}$ for $i=1,\dots,n-1$ is quantized to the one
acting on the twisted quantized enveloping algebra $U'_q(\mathfrak o_n)$
studied in \cite{MR}, while the action of  $\beta_{n,1}$ (quantized
or not) is new.

The fact that the Nelson--Regge algebra  coincides with the Poisson
algebras of monodromies of Fuchsian systems arising in Frobenius
manifold theory poses the natural question of characterizing the
special class of Frobenius manifolds coming from
Teichm\"uller theory. This is a highly non-trivial problem that we
postpone to subsequent work \cite{ChMa}.  In this paper, we show that
in two special limiting cases of the $A_n$ algebra, that we call
$A_3^\ast$ and $A_4^\ast$ respectively, the Teichm\"uller space
carries the same Frobenius manifold structure as the respective quantum
cohomology rings $H^\ast(\mathbb C\mathbb P^2)$ and $H^\ast(\mathbb
C\mathbb P^3)$.

Finally we interpret our $\mathfrak D_n$ algebra as the Poisson
algebra of the Stokes data of a Frobenius manifold in the vicinity
of a non semi-simple point.

The paper is organized as follows. In Sec.~\ref{se.orbi}, we briefly
recall the combinatorial description of Teichm\"uller spaces of
Riemann surfaces with holes and with  orbifold points and describe
the Goldman bracket \cite{Gold} of geodesic functions. The special
case of the Nelson--Regge algebras are considered in
Subsec.~\ref{sub:An} and the one of the  $D_n$ algebras in
Subsec.~\ref{sub:dn}. In Sec.~\ref{se:schl}, we consider the
isomonodromic deformations of a Fuchsian system with $n+1$ poles and
introduce its monodromy data. The Poisson brackets on the set of
these monodromy data are the Korotkin--Samtleben brackets \cite{KS}
described in Sec.~\ref{se:KS}. In Sec.~\ref{se:clash}, we introduce
the procedure of pole clashing and use it to generate a new hole.
The new (infinite-dimensional) Poisson algebras ${\mathfrak D}_n$ of
the monodromy data for the one-sheeted hyperboloid with $n$ orbifold
points are introduced in Sec.~\ref{se:newDn} where we also prove
that ${\mathfrak D}_n$ algebra is the semiclassical limit of the
twisted $q$-Yangian $Y'_q(\mathfrak{o}_n)$ for the orthogonal Lie
algebra $\mathfrak{o}_n$ (Subsec. \ref{subse:yangian}). In Subsec.
\ref{subse:braid} we construct the braid-group representation in
matrix form and in Subsec. \ref{subse:qbraid} we quantize it. We
study various reductions of these algebras in Sec.~\ref{se:central},
where we introduce the $p$-level reductions, the
algebras ${\mathfrak D}^{(p)}_n$,
which enjoy the same braid-group representation in the
adjoint matrix form as the general algebra, enabling us to evaluate
their central elements. We then turn to the case of the
$D_n$-algebra and show that it can be obtained by a special
reduction (based on the skein relations) from the ``ambient''
${\mathfrak D}_n$ algebra and with the same braid-group
representation as above. This fact enables us to construct $n$
central elements of the $D_n$ algebra. We prove the algebraic
independence of these central elements and that  the algebra $D_n$
admits in general no more than $n$ algebraically independent central
elements.

The link with the Frobenius manifold theory and the quantum
cohomology of projective spaces is carried out in Sec.~\ref{se:FM}.

Finally, in Appendix~A, we provide the description of monodromy data
for a general $n\times n$ Fuchsian system. In Appendix~B, we present
the proof of the Jacobi identities for the brackets of the
${\mathfrak D}_n$ algebra brackets, and in~Appendix~C, we present
the proof of the algebraic independence for the $\left[np/2\right]$
central elements of the algebra ${\mathfrak D}^{(p)}_n$.

\vskip 2mm \noindent{\bf Acknowledgements.} The authors are grateful
to Boris Dubrovin, who put them into contact and gave them many
helpful suggestions and to A. Molev for his insights on the theory
of Yangians. We would like to thank also Alexei Bondal, Volodya
Fock, Davide Guzzetti, Nigel Hitchin, Dima Korotkin, Maxim Nazarov,
Bob Penner, Vasilisa Shramchenko, Alexander Strohmaier and Sasha
Veselov for many enlighting conversations. This research  was
supported by the EPSRC ARF EP/D071895/1 and RA EP/F03265X/1, by the
RFBR grants 08-01-00501 and 09-01-92433-CE, by the Grant for Support
for the Scientific Schools 195.2008.1, by the Program Mathematical
Methods for Nonlinear Dynamics, by the Marie Curie training network
ENIGMA, by the ESF network MISGAM and by the Manchester Institute
for Mathematical Sciences.

\section{Orbifold Riemann surfaces}\label{se.orbi}

The graph description of the Teichm\"uller theory of surfaces with
orbifold points was
proposed in \cite{Ch1},~\cite{Ch1a}.\footnote{In \cite{Ch1}, it was developed for the
bordered Riemann surfaces, the interpretation in terms of the orbifold Riemann surfaces was given
in \cite{Ch1a}, but all the algebraic formulas in \cite{Ch1} are identical for the both geometrical interpretations.}
This theory is formulated in terms of
hyperbolic geometry by introducing new parameters (the number of
orbifold points on a Riemann surface with holes). Let us denote by
$\Sigma_{g,s,n}$ a Riemann surface of genus $g$ with $s$ holes and
$n$ orbifold points of order two. By the Poincar\'e uniformization
theorem
$$
\Sigma_{g,s,n}\sim {\mathbb H}\slash \Delta_{g,s,n},
$$
where
$$
 \Delta_{g,s,n}=\langle\gamma_1\dots,\gamma_{2g+s+n-1}\rangle,\qquad
 \gamma_1\dots,\gamma_{2g+s+n-1}\in{\mathbb P}SL(2,{\mathbb R})
$$
is a Fuchsian group,  the fundamental group of the
surface~$\Sigma_{g,s,n}$. In particular for orbifold Riemann
surfaces, the Fuchsian group $\Delta_{g,s,n}$ is {\it almost
hyperbolic,}\/ i.e. all its elements are either hyperbolic
(when all the holes have nonzero perimeters;
parabolic elements are allowed when a hole degenerates into a puncture) or have
trace equal to zero.

Let us remind the Thurston shear-coordinate description of the
Teichm\"uller spaces of Riemann surfaces with holes and, possibly, orbifold points (see \cite{Ch1a}).
The main idea \cite{Fock1} is to decompose each hyperbolic matrix
$\gamma\in \Delta_{g,s,n}$ as a product of the form
\begin{equation}\label{eq:decomp}
\gamma= (-1)^K R^{k_{i_p}} X_{Z_{i_p}} \dots R^{k_{i_1}} X_{Z_{i_1}},\qquad i_j\in I,\quad
k_{i_j}=1,2,\quad K:=\sum_{j=1}^p k_{i_j}
\end{equation}
where $I$ is a set of integer indices and the matrices $R,\, L$ and $X_{Z_i}$ are defined as follows:
\begin{eqnarray}\nn\label{eq:generators}
&&
R:=\left(\begin{array}{cc}1&1\\-1&0\\
\end{array}\right), \qquad
L=-R^2:=\left(\begin{array}{cc}0&1\\-1&-1\\
\end{array}\right), \\
&&
X_{Z_i}:=\left(\begin{array}{cc}0&-\exp\left({\frac{Z_i}{2}}\right)\\
\exp\left(-{\frac{Z_i}{2}}\right)&0\\
\nn\end{array}\right),
\end{eqnarray}
and to decompose each traceless element as
\begin{equation}
\label{eq:decomp1}
\gamma_0=\gamma^{-1} F \gamma,
\end{equation}
where $\gamma$ is decomposed as in (\ref{eq:decomp}) and
$$
F=\left(\begin{array}{cc}0&1\\-1&0\\
\end{array}\right).
$$
The decomposition of each element in the Fuchsian group
$\Delta_{g,s,n}$ is obtained by looking at the closed geodesic
corresponding to it in the fat--graph associated to
$\Sigma_{g,s,n}$.

Let us briefly recall how to associate a fat-graph to a Riemann
surface with holes but without orbifold points~\cite{Fock1}~\cite{Fock2}.
In this case, one considers
a spine $\Gamma_{g,s}$ corresponding to the Riemann surface
$\Sigma_{g,s}$ with $g$ handles and $s$ boundary components (holes).
The spine, or fat--graph $\Gamma_{g,s}$ is a connected graph that
can be drawn without self-intersections on $\Sigma_{g,s}$, it has all
vertices of valence three, it has a prescribed cyclic ordering of
labeled edges entering each vertex, and it is a maximal graph in the
sense that its complement on the Riemann surface
is a set of disjoint polygons (faces), each polygon
containing exactly one hole (and becoming simply connected after gluing
this hole). Since a graph must have at least one face, only Riemann
surfaces with at least one hole, $s>0$, can be described in this way.
The hyperbolicity condition also implies $2g-2+s>0$.

In the case where no orbifold points are present, the Fuchsian group
$\Delta_{g,s}$ is strictly hyperbolic if all the holes have nonzero
perimeters, and only the elements that correspond to holes degenerated
into punctures are parabolic ones. The decomposition (\ref{eq:decomp}) can
be obtained by establishing a one-to-one correspondence between
elements of the Fuchsian group and closed paths in the spine
starting and terminating at the same directed edge. Each  time the
path $A$ corresponding to the element $\gamma_A$ (or, equivalently,
to its invariant closed geodesic)  passes through the
$\alpha$th edge, an edge--matrix  $X_{Z_\alpha}$ with the real
coordinate $Z_\alpha$ appears in the decomposition of $\gamma$. At
the end of the edge, the path can either turn right or left, and a
matrix $R$ or $L$ respectively appears in the decomposition
\cite{Fock1}.

The introduction of orbifold points is achieved by considering {\em
new types of graphs} with pending vertices \cite{Ch1}. Then, if a
geodesic line comes to a pending vertex, it undergoes an {\em
inversion}, which corresponds to inserting the {\em inversion
matrix} $F$, into the corresponding string of $2\times2$-matrices.
The edge terminating at a pending vertex is called {\em pending}
edge.

All possible paths in the spine (graph) that are closed and may
experience an arbitrary number of inversions at pending vertices of
the graph must be taken into account.

The  algebras of geodesic length functions were constructed in \cite{Ch1}
by  postulating the Poisson relations on the level of the shear
coordinates $X_\alpha$ of the Teichm\"uller space:
\begin{equation}
\label{eq:Poisson}
\bigl\{f({\mathbf X}),g({\mathbf X})\bigr\}=\sum_{{\hbox{\small 3-valent} \atop \hbox{\small vertices $\alpha=1$} }}^{4g+2s+n-4}
\,\sum_{i=1}^{3\!\!\mod 3}
\left(\frac{\partial f}{\partial X_{\alpha_i}} \frac{\partial g}{\partial X_{\alpha_{i+1}}}
- \frac{\partial g}{\partial X_{\alpha_i}} \frac{\partial f}{\partial X_{\alpha_{i+1}}}\right),
\end{equation}
where the sum ranges all the three-valent vertices of a graph and
$\alpha_i$ are the labels of the cyclically (counterclockwise)
ordered ($\alpha_4\equiv \alpha_1 $) edges incident to the vertex
with the label $\alpha$. This bracket gives rise to the {\it Goldman
bracket} on the space of geodesic length functions \cite{Gold}.

We recall an important relation valid in $\mathbb P SL(2)$:
\begin{equation}\label{eq:skein}
 \Tr \gamma_A \Tr  \gamma_B=\Tr ( \gamma_A  \gamma_B ) + \Tr ( \gamma_A  \gamma_B^{-1}) .
\end{equation}
This relation corresponds to resolving the crossing between the two corresponding geodesics $A$ and $B$ as in Fig. \ref{fi:skein-cl}
and it referred to as {\it skein relation.}

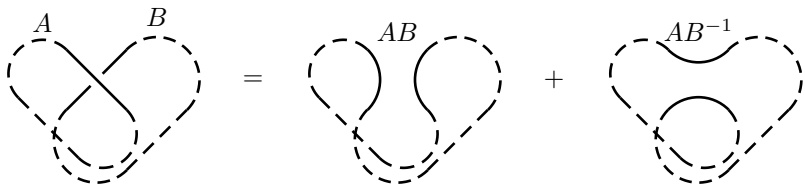
\begin{figure}[h]
{\psset{unit=0.4}
\begin{pspicture}(-12,-4)(12,4)
\newcommand{\LOOPS}{%
\psarc[linestyle=dashed, linewidth=1pt](2,0){1.42}{-45}{135}
\pcline[linestyle=dashed, linewidth=1pt](3,-1)(1,-3)
\psarc[linestyle=dashed, linewidth=1pt](0,-2){1.42}{135}{315}
\psarc[linestyle=dashed, linewidth=1pt](0.2,-1.8){1.13}{-135}{45}
\pcline[linestyle=dashed, linewidth=1pt](-0.6,-2.6)(-2.6,-0.6)
\psarc[linestyle=dashed, linewidth=1pt](-1.8,0.2){1.13}{45}{225}
}
\rput(-10,0){\LOOPS}
\rput(-10,0){
\pcline[linewidth=1pt](-1,1)(1,-1)
\pcline[linewidth=1pt](-1,-1)(-0.2,-0.2)
\pcline[linewidth=1pt](1,1)(0.2,0.2)
\rput(-1.8,1.6){\makebox(0,0)[cb]{$A$}}
\rput(2,1.8){\makebox(0,0)[cb]{$B$}}
}
\rput(-4.8,0){\makebox(0,0)[cc]{$=$}}
\rput(0,0){\LOOPS}
\rput(0,0){
\psarc[linewidth=1pt](-2,0){1.42}{-45}{45}
\psarc[linewidth=1pt](2,0){1.42}{135}{225}
\rput(0,1.3){\makebox(0,0)[cb]{$AB$}}
}
\rput(5.2,0){\makebox(0,0)[cc]{$+$}}
\rput(10,0){\LOOPS}
\rput(10,0){
\psarc[linewidth=1pt](0,2){1.42}{225}{315}
\psarc[linewidth=1pt](0,-2){1.42}{45}{135}
\rput(0,1.3){\makebox(0,0)[cb]{$A B^{-1}$}}
}
\end{pspicture}
}
\caption{\small The classical skein relation.}
\label{fi:skein-cl}
\end{figure}

\subsection{$A_n$ algebra}\label{sub:An}

The simplest case of orbifold Riemann surface is a Poincar\'e disk with $n\geq3$
orbifold points in the interior; we denote it by  $\Sigma_{0,1,n}$. In this case, the
fat-graph $\Gamma_{0,1,n}$ is a tree-like graph depicted in Fig.~\ref{fi:An} for
$n=3,4,\dots$. We enumerate the $n$ dot-vertices counterclockwise,
$i,j=1,\dots, n$, and consider the algebra of all geodesic
functions.\footnote{Note that in the cluster algebra terminology (see
\cite{FST}) these algebras were denoted by $A_{n-2}$.}

\subsubsection{Poisson relations for $A_n$ algebra}

Let $G_{i,j}$ with $i<j$
denote the geodesic function corresponding to the geodesic line that
encircles exactly two pending vertices with the indices $i$ and
$j$. Examples for $n=3$ and $n=4$ are in the figure \ref{fi:An}.
It turns out that these geodesic functions suffice for closing the
Poisson algebra:
\begin{eqnarray}\label{eq:NR}
&&
\left\{G_{i,k},G_{j,l}\right\}=0,\quad\hbox{for}\,\,i<k<j<l,\quad\hbox{and for}\,\, i<j<l<k, \nn\\
&&
\left\{G_{i,k},G_{j,l}\right\}=2\left(G_{i,j}G_{k,l}-G_{i,l}G_{k,j}\right),\quad\hbox{for}\, \, i<j<k<l,\nn\\
&&
\left\{G_{i,k},G_{k,l}\right\}=G_{i,k}G_{k,l}-2G_{i,l},\quad\hbox{for}\,\,  i<k<l,\\
&&
\left\{G_{i,k},G_{j,k}\right\}=-\left(G_{i,k}G_{j,k}-2G_{i,j}\right),\quad\hbox{for}\, \, i< j<k,\nn\\
&&
\left\{G_{i,k},G_{i,l}\right\}=-\left(G_{i,k}G_{i,l}-2G_{k,l}\right),\quad\hbox{for}\, \, i<k<l.\nn
\end{eqnarray}
Note that the left-hand side is doubled in this case as compared to Nelson--Regge algebras
recalled in~\cite{ChP}.

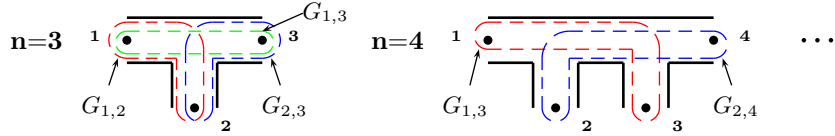
\begin{figure}[h]
\begin{center}
{\psset{unit=0.6}
\begin{pspicture}(-4,-3)(8,1)
\rput(-6,0){\makebox(0,0){$\mathbf{n{=}3}$}}
\pcline[linewidth=1pt](-4,0.5)(-1,0.5)
\pcline[linewidth=1pt](-4,-0.5)(-3,-0.5)
\pcline[linewidth=1pt](-2,-0.5)(-1,-0.5)
\pcline[linewidth=1pt](-3,-0.5)(-3,-1.5)
\pcline[linewidth=1pt](-2,-0.5)(-2,-1.5)
\pscircle*(-4,0){0.1}
\pscircle*(-1,0){0.1}
\pscircle*(-2.5,-1.5){0.1}
\psarc[linecolor=red, linestyle=dashed, linewidth=0.5pt](-4,0){.4}{90}{270}
\psarc[linecolor=red, linestyle=dashed, linewidth=0.5pt](-2.6,-1.5){.3}{180}{0}
\pcline[linecolor=red, linestyle=dashed, linewidth=0.5pt](-2.9,-1.5)(-2.9,-0.4)
\pcline[linecolor=red, linestyle=dashed, linewidth=0.5pt](-2.9,-0.4)(-4,-0.4)
\pcline[linecolor=red, linestyle=dashed, linewidth=0.5pt](-2.3,-1.5)(-2.3,-0.2)
\pcline[linecolor=red, linestyle=dashed, linewidth=0.5pt](-2.9,0.4)(-4,0.4)
\psarc[linecolor=red, linestyle=dashed, linewidth=0.5pt](-2.9,-0.2){.6}{0}{90}
\psarc[linecolor=blue, linestyle=dashed, linewidth=0.5pt](-1,0){.4}{-90}{90}
\psarc[linecolor=blue, linestyle=dashed, linewidth=0.5pt](-2.4,-1.5){.3}{180}{0}
\pcline[linecolor=blue, linestyle=dashed, linewidth=0.5pt](-2.1,-1.5)(-2.1,-0.4)
\pcline[linecolor=blue, linestyle=dashed, linewidth=0.5pt](-2.1,-0.4)(-1,-0.4)
\pcline[linecolor=blue, linestyle=dashed, linewidth=0.5pt](-2.7,-1.5)(-2.7,-0.2)
\pcline[linecolor=blue, linestyle=dashed, linewidth=0.5pt](-2.1,0.4)(-1,0.4)
\psarc[linecolor=blue, linestyle=dashed, linewidth=0.5pt](-2.1,-.2){.6}{90}{180}
\psarc[linecolor=green, linestyle=dashed, linewidth=0.5pt](-4,-0.05){.25}{90}{270}
\psarc[linecolor=green, linestyle=dashed, linewidth=0.5pt](-1,-.05){.25}{-90}{90}
\pcline[linecolor=green, linestyle=dashed, linewidth=0.5pt](-4,.2)(-1,.2)
\pcline[linecolor=green, linestyle=dashed, linewidth=0.5pt](-4,-0.3)(-1,-0.3)
\rput(-4.7,0.1){\makebox(0,0){\tiny$\mathbf1$}}
\rput(-1.8,-1.9){\makebox(0,0){\tiny$\mathbf2$}}
\rput(-0.3,0.1){\makebox(0,0){\tiny$\mathbf3$}}
\pcline[linewidth=0.5pt]{->}(-4.4,-1)(-4.2,-0.4)
\rput(-4.5,-1.5){\makebox(0,0){\small$G_{1,2}$}}
\pcline[linewidth=0.5pt]{->}(-0.6,-1)(-0.8,-0.4)
\rput(-.5,-1.5){\makebox(0,0){\small$G_{2,3}$}}
\pcline[linewidth=0.5pt]{->}(-0.1,0.6)(-1,0.2)
\rput(0.4,.6){\makebox(0,0){\small$G_{1,3}$}}
\rput(2,0){\makebox(0,0){$\mathbf{n{=}4}$}}
\pcline[linewidth=1pt](4,0.5)(9,0.5)
\pcline[linewidth=1pt](4,-0.5)(5,-0.5)
\pcline[linewidth=1pt](6,-0.5)(7,-0.5)
\pcline[linewidth=1pt](8,-0.5)(9,-0.5)
\pcline[linewidth=1pt](5,-0.5)(5,-1.5)
\pcline[linewidth=1pt](6,-0.5)(6,-1.5)
\pcline[linewidth=1pt](7,-0.5)(7,-1.5)
\pcline[linewidth=1pt](8,-0.5)(8,-1.5)
\pscircle*(4,0){0.1}
\pscircle*(9,0){0.1}
\pscircle*(5.5,-1.5){0.1}
\pscircle*(7.5,-1.5){0.1}
\psarc[linecolor=red, linestyle=dashed, linewidth=0.5pt](4,0.1){.3}{90}{270}
\psarc[linecolor=red, linestyle=dashed, linewidth=0.5pt](7.5,-1.5){.3}{180}{0}
\pcline[linecolor=red, linestyle=dashed, linewidth=0.5pt](7.2,-1.5)(7.2,-0.2)
\pcline[linecolor=red, linestyle=dashed, linewidth=0.5pt](7.8,-1.5)(7.8,-0.2)
\pcline[linecolor=red, linestyle=dashed, linewidth=0.5pt](4,0.4)(7.2,0.4)
\pcline[linecolor=red, linestyle=dashed, linewidth=0.5pt](4,-0.2)(7.2,-0.2)
\psarc[linecolor=red, linestyle=dashed, linewidth=0.5pt](7.2,-0.2){.6}{0}{90}
\psarc[linecolor=blue, linestyle=dashed, linewidth=0.5pt](9,-0.1){.3}{-90}{90}
\psarc[linecolor=blue, linestyle=dashed, linewidth=0.5pt](5.5,-1.5){.3}{180}{0}
\pcline[linecolor=blue, linestyle=dashed, linewidth=0.5pt](5.2,-1.5)(5.2,-0.4)
\pcline[linecolor=blue, linestyle=dashed, linewidth=0.5pt](5.8,-1.5)(5.8,-0.4)
\pcline[linecolor=blue, linestyle=dashed, linewidth=0.5pt](5.8,-0.4)(9,-0.4)
\pcline[linecolor=blue, linestyle=dashed, linewidth=0.5pt](5.8,0.2)(9,0.2)
\psarc[linecolor=blue, linestyle=dashed, linewidth=0.5pt](5.8,-.4){.6}{90}{180}
\rput(3.3,0.1){\makebox(0,0){\tiny$\mathbf1$}}
\rput(6.2,-1.9){\makebox(0,0){\tiny$\mathbf2$}}
\rput(8.2,-1.9){\makebox(0,0){\tiny$\mathbf3$}}
\rput(9.7,0.1){\makebox(0,0){\tiny$\mathbf4$}}
\pcline[linewidth=0.5pt]{->}(3.6,-1)(3.8,-0.2)
\rput(3.5,-1.5){\makebox(0,0){\small$G_{1,3}$}}
\pcline[linewidth=0.5pt]{->}(9.4,-1)(9.2,-0.4)
\rput(9.5,-1.5){\makebox(0,0){\small$G_{2,4}$}}
\pscircle*(11,0){0.05}
\pscircle*(11.3,0){0.05}
\pscircle*(11.6,0){0.05}
\end{pspicture}
}
\caption{\small Generating graphs for $A_n$ algebras for $n=3,4,\dots$. We indicate character
geodesics whose geodesic functions $G_{ij}$ enter bases of the corresponding algebras.}
\label{fi:An}\end{center}
\end{figure}


In this paper we consider a basis $\gamma_1,\dots,\gamma_n$ in the Fuchsian group $\Delta_{0,1,n}$ such that
$$
-\Tr(\gamma_i\gamma_j)=G_{i,j}.
$$
(The sign convention is such that when we interpret $G_{i,j}$ as being the geodesic functions
related to lengths $\ell_{i,j}$ of closed geodesics, we have $G_{i,j}=2\cosh(\ell_{i,j}/2)\ge2$.)  In this case, for convenience
we let $Z_i$ denote the coordinates of pending edges and $Y_j$ all
other coordinates. In the case where we do not distinguish
between pending and internal edges, we preserve the notation $X_\alpha$ for all the
coordinates. This basis is given by the following (we write it in $SL(2,\mathbb R)$):
\begin{eqnarray}
&{}&
\gamma_1=F,\nn\\
&{}&
\gamma_2 =- X_{Z_1} L X_{Z_2} F X_{Z_2} R X_{Z_1}\nn\\
&&
\gamma_3= -X_{Z_1} R X_{Y_1} L X_{Z_3} F X_{Z_3} R X_{Y_1} L X_{Z_1}   \nn\\
&{}&
\dots\nn \\
&{}&
\gamma_i =- X_{Z_1} R X_{Y_1} R X_{Y_2}  \dots R X_{Y_{i-2}} L X_{Z_i} F X_{Z_i}  R X_{Y_{i-2}} L  \dots
X_{Y_1} L X_{Z_1},\label{eq:basisn}\\
&{}&
\dots \nn \\
&{}&
\gamma_{n-1} =- X_{Z_1} R X_{Y_1} R X_{Y_2}  \dots R X_{Y_{n-3}} L X_{Z_{n-1}} F X_{Z_{n-1}}  R X_{Y_{n-3}} L  \dots
X_{Y_1} L X_{Z_1},\nn \\
&{}&
\gamma_n =-  X_{Z_1} R X_{Y_1} R X_{Y_2}  \dots R X_{Y_{n-2}} R X_{Z_n} F X_{Z_n}  R X_{Y_{n-2}} L  \dots X_{Y_1} L X_{Z_1},\nn
\end{eqnarray}
Observe that $\Tr\gamma_i=0$, $i=1,\dots,n$. It is not hard to check that the matrix
$$
\gamma_\infty:=(\gamma_1 \gamma_{2} \dots \gamma_n)^{-1}
$$
has eigenvalues $(-1)^{n-1} e^{\pm P/2}$, where $P$ is the length of the perimeter around the hole:
\begin{equation}\label{eq:per}
P = 2\sum_{i=1}^n Z_i+2\sum_{j=1}^{n-3} Y_j.
\end{equation}

\subsubsection{Mapping-class group action on $A_n$}

Observe that there is a degree of arbitrariness in the choice of the
fat graph $\Gamma_{g,s,n}$ associated to a Riemann surface
$\Sigma_{g,s,n}$. This arbitrariness is described by the Whitehead
moves \cite{Penn1}
and their generalization to the case of pending edges \cite{Ch1}.
Using these moves, or {\it flip morphisms,}\/ one can establish a
morphism between any two algebras corresponding to surfaces of the
same genus, same number of boundary components, and same
number of orbifold points. If, after a series of
morphisms,  a graph of the same combinatorial type as the
initial one (disregarding marking of edges) is obtained, then this morphism is associated to a {\em
mapping class group} operation, therefore passing
from the groupoid of morphisms to the mapping class group.

For the $A_n$ algebra, the action of the mapping class group
corresponds to the following action of the braid group \cite{Ch1}:
let us construct the upper-triangular matrix
\begin{gather}
\label{A-matrix}
{\mathcal A}=\left(\begin{array}{ccccc}
                     1 & G_{1,2} & G_{1,3} & \dots & G_{1,n} \\
                     0 & 1 & G_{2,3} & \dots & G_{2,n} \\
                     0 & 0 & 1 & \ddots & \vdots \\
                     \vdots & \vdots & \ddots & \ddots & G_{n-1,n} \\
                     0 & 0 & \dots & 0 & 1 \\
                   \end{array}
\right)
\end{gather}
associating the entries $G_{i,j}$ with the geodesic length functions. The action of the braid group element $\beta_{i,i+1}$ is defined as:
\begin{equation}\label{eq:braidch}
\beta_{i,i+1}{\mathcal A}={\ttilde{\mathcal A}},\qquad \hbox{where} \ \
\left\{
\begin{array}{ll}
  {\ttilde G}_{i+1,j}=G_{i,j}, & j>i+1,\\
  {\ttilde G}_{j,i+1}=G_{j,i}, & j<i, \\
  {\ttilde G}_{i,j}=G_{i,j}G_{i,i+1}-G_{i+1,j}, & j>i+1, \\
  {\ttilde G}_{j,i}=G_{j,i}G_{i,i+1}-G_{j,i+1}, & j<i, \\
  {\ttilde G}_{i,i+1}=G_{i,i+1}. &  \\
\end{array}%
\right.
\end{equation}
A very convenient way to present this transformation is by introducing the special matrices
$B_{i,i+1}$ of the block-diagonal form (see \cite{dubrovin})
\begin{equation}
\label{Bii+1}
B_{i,i+1}=\begin{array}{c}
            \vdots \\
            i \\
            i+1 \\
            \vdots \\
          \end{array}
          \left(
          \begin{array}{cccccccc}
            1 &  &  &  &  &  &  &  \\
             & \ddots &  &  &  &  &  &  \\
             &  & 1 &  &  &  &  &  \\
             &  &  &   G_{i,i+1} & -1 &  & &  \\
            &  &  & 1 & 0 &  &  &  \\
             &  &  &  &  & 1 &  &  \\
             &  &  &  &  &  & \ddots &  \\
             &  &  &  &  &  &  & 1 \\
          \end{array}
          \right).
\end{equation}
Then, the action of the braid group generator $\beta_{i,i+1}$ on ${\mathcal A}$ acquires merely
a matrix product form:
\begin{gather}
\label{BAB}
\beta_{i,i+1}{\mathcal A}=B_{i,i+1}{\mathcal A}B^{T}_{i,i+1}
\end{gather}
where $B^{T}_{i,i+1}$ denotes the matrix transposed to $B_{i,i+1}$.
In this setting it is easy to prove the braid group relations:
$$
\beta_{i-1,i}\beta_{i,i+1}\beta_{i-1,i}=\beta_{i,i+1}\beta_{i-1,i}\beta_{i,i+1},\qquad 2\le i\le n-1,
$$
and also the extra relation \cite{dubrovin}:
\begin{equation}\label{eq:secondb}
\bigl(\beta_{n-1,n}\beta_{n-2,n-1}\dots\beta_{2,3} \beta_{1,2}\bigr)^n=\hbox{\rm Id}.
\end{equation}

\begin{remark}
Observe that upper triangular matrices of the form (\ref{A-matrix})
can be interpreted as Stokes matrices of certain linear system of
ordinary differential equations appearing in the theory of Frobenius
Manifolds \cite{dubrovin}. A study of the special class of Frobenius
manifolds arising in Teichm\"uller theory is in progress
\cite{ChMa}. We show some preliminary results in Sec.~\ref{se:FM}.
\end{remark}

\subsection{$D_n$ algebra}\label{sub:dn}

The simplest case of orbifold Riemann surface with two holes is an
annulus with $n\geq 2$ marked points,
$\Sigma_{0,2,n}$. The fat graph $\Gamma_{0,2,n}$ with $n=4$ is shown in
Fig.~\ref{fi:Dn}.

\subsubsection{Poisson relations for the $D_n$ algebra}

Again, the  algebras of geodesic functions were constructed in
\cite{Ch1} by  postulating the Poisson relations on the level of the
shear coordinates of the Teichm\"uller space. To close this Poisson
algebra more geodesic functions are needed than in the case of $\mathcal A_n$,
they are: ${\hhat G}_{i,i}$, the
geodesic containing the $i$-th pending vertex and the hole, and for
each $i,j=1,\dots, n$ two geodesics containing the pending vertices
$i$ and $j$: ${\hhat G}_{i,j}$ and ${\hhat G}_{j,i} $.  Here, the order of subscripts
indicates the direction of encompassing the hole (the second
boundary component of the annulus), see Fig.~\ref{fi:Dn}. Obviously, ${\hhat G}_{i,j}$ with $1\le i<j\le n$
constitute one (among $n$ possible) $A_n$-subalgebras of the $D_n$ algebra.
The total number of generators of the $D_n$ algebra is therefore $n^2$.
In this paper, we indicate the
geodesic functions from this set by the hat symbol to distinguish them
from the level $k$ geodesic functions which will be introduced in Section \ref{se:newDn}.
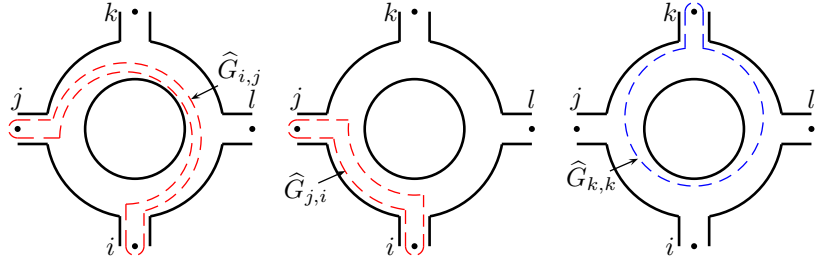
\begin{figure}[tb]
{\psset{unit=0.4}
\begin{pspicture}(-6,-4)(6,4)
\psarc[linewidth=1pt](0,0){2.95}{-80}{-10}
\psarc[linewidth=1pt](0,0){2.95}{10}{80}
\psarc[linewidth=1pt](0,0){2.95}{100}{170}
\psarc[linewidth=1pt](0,0){2.95}{190}{260}
\pscircle[linewidth=1pt](0,0){1.7}
\pcline[linewidth=1pt](-0.5,-2.9)(-0.5,-3.9)
\pcline[linewidth=1pt](0.5,-2.9)(0.5,-3.9)
\pcline[linewidth=1pt](-0.5,2.9)(-0.5,3.9)
\pcline[linewidth=1pt](0.5,2.9)(0.5,3.9)
\pcline[linewidth=1pt](-2.9,-0.5)(-3.9,-0.5)
\pcline[linewidth=1pt](-2.9,0.5)(-3.9,0.5)
\pcline[linewidth=1pt](2.9,-0.5)(3.9,-0.5)
\pcline[linewidth=1pt](2.9,0.5)(3.9,0.5)
\psarc[linecolor=red, linestyle=dashed, linewidth=0.5pt](0,-3.9){.3}{180}{360}
\pscircle*[linewidth=0.5pt](0,-3.9){.1}
\pcline[linecolor=red, linestyle=dashed, linewidth=0.5pt](-0.3,-3.9)(-0.3,-2.5)
\pcline[linecolor=red, linestyle=dashed, linewidth=0.5pt](0.3,-3.9)(0.3,-2.8)
\psarc[linecolor=red, linestyle=dashed, linewidth=0.5pt](-0.3,-0.3){2.2}{-90}{180}
\psarc[linecolor=red, linestyle=dashed, linewidth=0.5pt](-0.3,-0.3){2.5}{-75}{165}
%
\pscircle*[linewidth=0.5pt](0,3.9){.1}
\psarc[linecolor=red, linestyle=dashed, linewidth=0.5pt](-3.9,0){.3}{90}{270}
\pscircle*[linewidth=0.5pt](-3.9,0){.1}
\pcline[linecolor=red, linestyle=dashed, linewidth=0.5pt](-3.9,0.3)(-2.8,0.3)
\pcline[linecolor=red, linestyle=dashed, linewidth=0.5pt](-3.9,-0.3)(-2.5,-0.3)
\pscircle*[linewidth=0.5pt](3.9,0){.1}
\rput(3.5,2){\makebox(0,0){${\hhat G}_{i,j}$}}
\pcline[linewidth=0.5pt]{->}(2.76,1.44)(1.86,0.99)
\rput(-3.9,1){\makebox(0,0){$j$}}
\rput(-0.8,3.9){\makebox(0,0){$k$}}
\rput(3.9,1){\makebox(0,0){$l$}}
\rput(-0.8,-3.9){\makebox(0,0){$i$}}
\end{pspicture}
\begin{pspicture}(-3,-4)(6,0)
\psarc[linewidth=1pt](0,0){2.95}{-80}{-10}
\psarc[linewidth=1pt](0,0){2.95}{10}{80}
\psarc[linewidth=1pt](0,0){2.95}{100}{170}
\psarc[linewidth=1pt](0,0){2.95}{190}{260}
\pscircle[linewidth=1pt](0,0){1.7}
\pcline[linewidth=1pt](-0.5,-2.9)(-0.5,-3.9)
\pcline[linewidth=1pt](0.5,-2.9)(0.5,-3.9)
\pcline[linewidth=1pt](-0.5,2.9)(-0.5,3.9)
\pcline[linewidth=1pt](0.5,2.9)(0.5,3.9)
\pcline[linewidth=1pt](-2.9,-0.5)(-3.9,-0.5)
\pcline[linewidth=1pt](-2.9,0.5)(-3.9,0.5)
\pcline[linewidth=1pt](2.9,-0.5)(3.9,-0.5)
\pcline[linewidth=1pt](2.9,0.5)(3.9,0.5)
\psarc[linecolor=red, linestyle=dashed, linewidth=0.5pt](0,-3.9){.3}{180}{360}
\pscircle*[linewidth=0.5pt](0,-3.9){.1}
\pcline[linecolor=red, linestyle=dashed, linewidth=0.5pt](-0.3,-3.9)(-0.3,-2.7)
\pcline[linecolor=red, linestyle=dashed, linewidth=0.5pt](0.3,-3.9)(0.3,-2.2)
\psarc[linecolor=red, linestyle=dashed, linewidth=0.5pt](0,0){2.2}{175}{275}
\psarc[linecolor=red, linestyle=dashed, linewidth=0.5pt](0,0){2.7}{187}{263}
\pscircle*[linewidth=0.5pt](0,3.9){.1}
\psarc[linecolor=red, linestyle=dashed, linewidth=0.5pt](-3.9,0){.3}{90}{270}
\pscircle*[linewidth=0.5pt](-3.9,0){.1}
\pcline[linecolor=red, linestyle=dashed, linewidth=0.5pt](-3.9,0.3)(-2.2,0.3)
\pcline[linecolor=red, linestyle=dashed, linewidth=0.5pt](-3.9,-0.3)(-2.7,-0.3)
\pscircle*[linewidth=0.5pt](3.9,0){.1}
\rput(-3.6,-2){\makebox(0,0){${\hhat G}_{j,i}$}}
\pcline[linewidth=0.5pt]{->}(-3.2,-1.8)(-2.3,-1.4)
\rput(-3.9,1){\makebox(0,0){$j$}}
\rput(-0.8,3.9){\makebox(0,0){$k$}}
\rput(3.9,1){\makebox(0,0){$l$}}
\rput(-0.8,-3.9){\makebox(0,0){$i$}}
\end{pspicture}
\begin{pspicture}(-3,-4)(6,0)
\psarc[linewidth=1pt](0,0){2.95}{-80}{-10}
\psarc[linewidth=1pt](0,0){2.95}{10}{80}
\psarc[linewidth=1pt](0,0){2.95}{100}{170}
\psarc[linewidth=1pt](0,0){2.95}{190}{260}
\pscircle[linewidth=1pt](0,0){1.7}
\pcline[linewidth=1pt](-0.5,-2.9)(-0.5,-3.9)
\pcline[linewidth=1pt](0.5,-2.9)(0.5,-3.9)
\pcline[linewidth=1pt](-0.5,2.9)(-0.5,3.9)
\pcline[linewidth=1pt](0.5,2.9)(0.5,3.9)
\pcline[linewidth=1pt](-2.9,-0.5)(-3.9,-0.5)
\pcline[linewidth=1pt](-2.9,0.5)(-3.9,0.5)
\pcline[linewidth=1pt](2.9,-0.5)(3.9,-0.5)
\pcline[linewidth=1pt](2.9,0.5)(3.9,0.5)
%
\pscircle*[linewidth=0.5pt](0,-3.9){.1}
%
\psarc[linecolor=blue, linestyle=dashed, linewidth=0.5pt](0,3.9){.3}{0}{180}
\pscircle*[linewidth=0.5pt](0,3.9){.1}
\pcline[linecolor=blue, linestyle=dashed, linewidth=0.5pt](-0.3,3.9)(-0.3,2.7)
\pcline[linecolor=blue, linestyle=dashed, linewidth=0.5pt](0.3,3.9)(0.3,2.7)
\psarc[linecolor=blue, linestyle=dashed, linewidth=0.5pt](0,0.4){2.3}{97}{83}
\pscircle*[linewidth=0.5pt](-3.9,0){.1}
\pscircle*[linewidth=0.5pt](3.9,0){.1}
\rput(-3.4,-1.6){\makebox(0,0){${\hhat G}_{k,k}$}}
\pcline[linewidth=0.5pt]{->}(-2.8,-1.4)(-1.86,-0.99)
\rput(-3.9,1){\makebox(0,0){$j$}}
\rput(-0.8,3.9){\makebox(0,0){$k$}}
\rput(3.9,1){\makebox(0,0){$l$}}
\rput(-0.8,-3.9){\makebox(0,0){$i$}}
\end{pspicture}
}
\caption{\small Typical geodesics corresponding to the geodesic functions
constituting a set of generators of the $D_n$ algebra in the approach of \cite{Ch1}. }
\label{fi:Dn}
\end{figure}

The relevant Poisson brackets are cubic, and can be found in
\cite{Ch1},~\cite{Ch2}. One of the aims of this paper is to describe
the $D_n$-algebras as reductions of  the $\mathfrak D_n$ algebras
which will be constructed in section \ref{se:newDn} below. We shall
realize this reduction first in the geometric case, i.e., using the
skein relations, then in the analytical case and finally in  the
abstract algebraic case.

Let us briefly describe the mapping class group action in terms of
the braid group. Note that both the Poisson brackets and the action of the braid group
do not depend on the perimeter of the hole, so these $D_n$ algebras can be considered as
abstract Poisson algebras, i.e. as algebras for $n^2$ formal objects $\hhat G_{i,j}$, $i,j=1,\dots,n$.

\subsubsection{Braid group relations for $D_n$-algebras}\label{se:brdn}

The action of the braid group on
the generators  ${\hhat G}_{i,j}$, $i,j=1,\dots,n$ of the $D_n$ algebra can be
presented in the explicit form as follows:
\be
\beta_{i,i+1}{\hhat G}_{k,l}={\ttilde{\hhat G}}_{k,l},: \
\left\{
\begin{array}{ll}
  {\ttilde{\hhat G}}_{i+1,k}={\hhat G}_{i,k} & k\ne i,i+1,\\
  {\ttilde {\hhat G}}_{i,k}={\hhat G}_{i,k}{\hhat G}_{i,i+1}-{\hhat G}_{i+1,k} &  k\ne i,i+1,\\
  {\ttilde {\hhat G}}_{k,i+1}={\hhat G}_{k,i} &  k\ne i,i+1, \\
  {\ttilde {\hhat G}}_{k,i}={\hhat G}_{k,i}{\hhat G}_{i,i+1}-{\hhat G}_{k,i+1} &  k\ne i,i+1, \\
  {\ttilde {\hhat G}}_{i,i+1}={\hhat G}_{i,i+1} &  \\
  {\ttilde {\hhat G}}_{i+1,i+1}={\hhat G}_{i,i} &  \\
  {\ttilde {\hhat G}}_{i,i}={\hhat G}_{i,i}{\hhat G}_{i,i+1}-{\hhat G}_{i+1,i+1} &  \\
\multispan{2} \hbox{${\ttilde {\hhat G}}_{i+1,i}={\hhat G}_{i+1,i}+{\hhat G}_{i,i+1}{\hhat G}_{i,i}^2
  -2{\hhat G}_{i,i}{\hhat G}_{i+1,i+1}$}  \\
\end{array}%
\right.
\label{Dn-braid-cl}
\ee
for $1\le i\le n-1$ and
\be
\beta_{n,1}{\hhat G}_{k,l}={\ttilde{\hhat G}}_{k,l}: \
\left\{
\begin{array}{ll}
  {\ttilde {\hhat G}}_{1,k}={\hhat G}_{n,k} & k\ne n,1,\\
  {\ttilde {\hhat G}}_{n,k}={\hhat G}_{n,k}{\hhat G}_{n,1}-{\hhat G}_{1,k} &  k\ne n,1,\\
  {\ttilde {\hhat G}}_{k,1}={\hhat G}_{k,n} &  k\ne n,1, \\
  {\ttilde {\hhat G}}_{k,n}={\hhat G}_{k,n}{\hhat G}_{n,1}-{\hhat G}_{k,1} &  k\ne n,1, \\
  {\ttilde {\hhat G}}_{n,1}={\hhat G}_{n,1} &  \\
  {\ttilde {\hhat G}}_{1,1}={\hhat G}_{n,n} &  \\
  {\ttilde {\hhat G}}_{n,n}={\hhat G}_{n,n}{\hhat G}_{n,1}-{\hhat G}_{1,1} &  \\
\multispan{2} \hbox{${\ttilde {\hhat G}}_{1,n}={\hhat G}_{1,n}+{\hhat G}_{n,1}{\hhat G}_{n,n}^2-2{\hhat G}_{n,n}{\hhat G}_{1,1}$}  \\
\end{array}%
\right. .
\label{Dn-braid-cl-n1}
\ee

\begin{lm} \label{lem-braid-Dn-old}
For any $n\ge2$, we have the braid group relation for the transformations (\ref{Dn-braid-cl}),
(\ref{Dn-braid-cl-n1}):
\be
\label{RRR-Dn-old}
\beta_{i-1,i}\beta_{i,i+1}\beta_{i-1,i}=\beta_{i,i+1}\beta_{i-1,i}\beta_{i,i+1}, i=1,\dots,n\ \mod n,
\ee
where  for $i=n$ the element $\beta_{n,n+1}$ stands for $\beta_{n,1}$.
\end{lm}

Note that the second braid-group relation (\ref{eq:secondb}) is lost
in the case of $D_n$-algebras. This is due to the topological restriction imposed by the extra hole.

Presenting the braid-group action in the matrix-action (covariant)
form (\ref{BAB}) is a nontrivial problem. In fact special combinations of ${\hhat G}_{i,j}$
admit similar transformation laws under the subgroup
$\langle\beta_{1,2},\dots,\beta_{n-1,n}\rangle$ of braid-group
transformations generated by relations (\ref{Dn-braid-cl}) alone. In
fact the following result was proved in \cite{Ch1}:

\begin{lm} \label{lem-braid-Dn-matrix}
Consider the $n\times n$ skewsymmetric matrix $\widehat{\mathcal R}$ of entries:
\be\label{R.cl}
({\widehat{\mathcal R}})_{i,j}:=\left\{ \begin{array}{cc}
                        -{\hhat G}_{j,i}-{\hhat G}_{i,j}+{\hhat G}_{i,i}{\hhat G}_{j,j} &\quad j<i \\
                        {\hhat G}_{j,i}+{\hhat G}_{i,j}-{\hhat G}_{i,i}{\hhat G}_{j,j} &\quad  j>i \\
                               0 &\quad  j=i \\
                             \end{array}
                             \right.,
\ee
 the symmetric matrix $\widehat{\mathcal S}$ of entries:
\be
\label{S.cl}
({\widehat{\mathcal S}})_{i,j}:={\hhat G}_{i,i}{\hhat G}_{j,j}\quad \hbox{for all}\quad 1\le i,j\le n;
\ee
and the upper triangular matrix $\hhat{\mathcal A}$ of entries
\be
\label{hatA-matrix}
{\hhat{\mathcal A}}_{i,j}
=\left\{\begin{array}{cc}
                     {\hhat G}_{i,j} &\quad i<j \\
                        0 &\quad  i>j \\
                               1&\quad  i=j \\
                   \end{array}
\right.
\ee
Then any linear combination $w_1\widehat{\mathcal A}+w_2\widehat{\mathcal A}^T+\rho
{\widehat{\mathcal R}}+\sigma {\widehat{\mathcal S}}$ with complex $w_1$, $w_2$, $\rho$,
and $\sigma$ transforms by formula (\ref{BAB}) under the subgroup
$\langle\beta_{1,2},\dots,\beta_{n-1,n}\rangle$ of braid-group
transformations generated by relations (\ref{Dn-braid-cl}) alone.
\end{lm}

Below we construct the matrix representation of the total braid group
action and find the central elements of the $D_n$ algebra (see sub--section \ref{subse:DDn}).

\section{Monodromy preserving deformations}\label{se:schl}

In this section we interpret the matrices $\gamma_1,\dots,\gamma_n$
as monodromy matrices of a Fuchsian system with $2\times 2$ residue
matrices ${A}_j$ independent on $\lambda$:
\begin{equation}
{{\rm d}\over{\rm d}\lambda} \Phi=\left(
\sum_{k=1}^{n}{{A}_k\over \lambda-u_k} \right)\Phi,
\label{N1in}
\end{equation}
where ${\bf u}=(u_1,\dots,u_{n})$ are the pairwise distinct pending vertices in the fat graph.
The residue matrices ${A}_j$ satisfy the following conditions:
$$
{\rm eigen}\left({A}_j\right)=\pm\frac{1}{4}
\quad\hbox{and}\quad
-\sum_{k=1}^{n}{A}_k={A}_\infty,
$$
where, given
$$
\mu:=\left\{\begin{array}{cc}
\frac{P}{4\pi i},&\quad\hbox{for }\, n \hbox{ odd}\\
\frac{P}{4\pi i}+\frac{1}{2},&\quad\hbox{for }\, n \hbox{ even}\\
\end{array}\right.
$$
\begin{equation}\label{eq:Ainf}
{A}_\infty:=\left(
\begin{array}{cc}
\mu& 0\\ 0 & -\mu\\
\end{array}\right),\,\,\hbox{for }\,\mu\neq 0\hbox{ and }
{A}_\infty:=\left(
\begin{array}{cc}0& 1\\ 0 & 0\\\end{array}\right),\,\,\hbox{for }\,\mu= 0. \end{equation}

The description of the monodromy data of the system (\ref{N1in}) is
recalled in Appendix~A. It is convenient to fix the base point of
the fundamental group at $\infty$ so that one actually considers the
monodromy matrices
\begin{equation}\label{eq:mon-d}
M_i = C_\infty^{-1} \gamma_i C_\infty, \qquad M_\infty = C_\infty^{-1} \gamma_\infty C_\infty,
\end{equation}
where $C_\infty$ is the matrix of the eigenvalues of $\gamma_\infty$  so that
\begin{eqnarray}\label{eq:Minf}
& M_\infty =\left(
\begin{array}{cc}(-1)^{n} e^{P/2}&0\\ 0& (-1)^{n}e^{-P/2}\\ \end{array}
\right),&\hbox{for}\,P\neq 0,\nn\\
\\
& M_\infty:=\left(
\begin{array}{cc}(-1)^{n}& 1\\ 0 & (-1)^{n}\\\end{array}\right),&\hbox{for}\, P=0\nn
\end{eqnarray}

Given a point in the Teichm\"uller space, specified by
$\gamma_1,\dots,\gamma_n$, or equivalently $M_1,\dots,M_n$ there
exists a Fuchsian system having monodromy matrices $M_1,\dots,M_n$.
More precisely the following general theorems hold true:

\begin{theorem}\cite{Dek}\label{th:RH}
Given $n$ arbitrary $2\times2$ matrices $M_1,\dots,M_n$ and an arbitrary number $\mu$ such that
$$
M_\infty:=(M_1 M_2 \dots M_{n-1} M_n)^{-1}
$$
is given by
\begin{equation}\label{eq:Minf1}
M_\infty =\left\{\begin{array}{l}
\left(
\begin{array}{cc} e^{2i\,\pi\mu}&0\\ 0& e^{-2i\,\pi\mu}\\ \end{array}
\right),\quad\hbox{for}\,\mu\not\in \mathbb Z,\, \frac{1}{2}+\mathbb Z,\\
\\
\left(\begin{array}{cc} 1& 1\\ 0 &1\\ \end{array}\right),\quad\hbox{for}\, \mu\in\mathbb Z,\\
\\
\left(\begin{array}{cc} -1& 1\\ 0 &-1\\ \end{array}\right),\quad\hbox{for}\, \mu\in\frac{1}{2}+\mathbb Z,\\
\end{array}\right.
\end{equation}
and fixed a point ${\bf u}^0=(u_1^0,\dots,u_n^0)\in X_n$,
$X_n:=\mathbb C^n\setminus\{diagonals\}$, for any neighbourhood
$U\subset X_n$ of ${\bf u}^0$ there exists ${\bf u}\in U$ and a
Fuchsian system
$$
{{\rm d}\over{\rm d}\lambda} \Phi=\left(
\sum_{k=1}^{n}{{A}_k\over \lambda-u_k} \right)\Phi,
$$
with the given monodromy matrices  $M_1,\dots,M_n$ and with $A_\infty$ given by (\ref{eq:Ainf}).
\end{theorem}

Indeed there is a whole family of Fuchsian systems with the same
monodromy matrices, they are given by the solutions  of the
Schlesinger equations (\ref{schleq}). In fact the following theorem is true {\it in any dimension:}:

\begin{theorem}\cite{Mal},~\cite{Miwa}\label{th:iso}
Let $M_1,\dots,M_n$ be the monodromy matrices of the Fuchsian system
\begin{equation}\label{eq:fuchs0}
{{\rm d}\over{\rm d}\lambda} \Phi^0=\left(
\sum_{k=1}^{n}{{A}_k^0\over \lambda-u_k^0} \right)\Phi^0,
\end{equation}
with ${\bf u}^0=(u_1^0,\dots,u_n^0)\in X_n$. Then there exists a
neighbourhood $U\subset X_n$ of ${\bf u}^0$ such that for any ${\bf u}\in U$
there exists a unique $n$-uple $A_1({\bf u}),\dots,A_n({\bf u})$
of analytic valued matrix functions such that
$$
A_i({\bf u}^0)=A_i^0,\qquad i=1,\dots,n,
$$
and the monodromy matrices of the system
$$
{{\rm d}\over{\rm d}\lambda} \Phi=\left(
\sum_{k=1}^{n}{{A}_k({\bf u})\over \lambda-u_k} \right)\Phi,
$$
with respect to the same basis of loops, coincide with
$M_1,\dots,M_n$. The matrices $A_1({\bf u}),\dots,A_n({\bf u})$  are
solutions of the Schlesinger equations:
\begin{equation}
{\partial\over\partial u_j} {A}_i=
{[ {A}_i, {A}_j]\over u_i-u_j},\quad
{\partial\over\partial u_i} {A}_i=
-\sum_{j\neq i}{[ {A}_i, {A}_j]\over u_i-u_j}.
\label{schleq}\end{equation}
The solution $\Phi^0(\lambda)$ of (\ref{eq:fuchs0}) can be uniquely continued, for $\lambda\neq u_i$ to an analytic function
$$
\Phi(\lambda,{\bf u}),\qquad {\bf u}\in U,
$$
such that $\Phi(\lambda,{\bf u}^0)=\Phi^0(\lambda)$. This continuation is the local solution of the Cauchy problem with
the initial data $\Phi^0$ for the following system:
$$
\frac{\partial}{\partial u_i}\Phi = -\frac{A_i}{\lambda-u_i}\Phi.
$$
Moreover the functions $A_1({\bf u}),\dots,A_n({\bf u})$  and $\Phi(\lambda,{\bf u})$ can be
continued analytically to global meromorphic functions on the universal
coverings of $X_n$ and ${\mathbb P}^1\setminus\{u_1,\dots,u_n\}\otimes X_n$
respectively.
\end{theorem}

The above theorems establish the {\it Riemann--Hilbert correspondence:}
$$
\mathcal M\slash \mathcal D\leftrightarrow \mathcal A\slash \mathcal D
$$
where $\mathcal D=\{D\in GL(2,\mathbb C) \hbox{ diagonal matrix}\}$ and
$$
\mathcal M := \{(M_1,\dots,M_n)\in SL(2,\mathbb C): \Tr M_i=0,\,(M_1M_2\dots M_n)^{-1}= M_\infty
 \hbox{ given in (\ref{eq:Minf})} \}
$$
and
$$
\mathcal A := \{(A_1,\dots,A_n)\in\mathfrak{sl}(2,\mathbb C): {\rm eigen}(A_i)=\pm\frac{1}{4},
\,\sum_{k=1}^{n}{A}_k=-{A}_\infty,\, A_\infty \hbox{ as in (\ref{eq:Ainf})} \}.
$$
The choice of the generators $\gamma_1,\dots,\gamma_n$ in the
Fuchsian group $\Delta_n$ allows us to extend the Riemann--Hilbert
correspondence to the "suitably" complexified Teichm\"uller space,
where "suitably" means the complexification of the Teichm\"uller
space that corresponds to the handle--body case. We postpone the
study of the extension of the Riemann--Hilbert correspondence to the
handle--body Teichm\"uller space to subsequent publications.

\subsection{Analytic continuation of the Schlesinger equations solutions and braid group action}

The procedure of the analytic  continuation of the solutions to the
Schlesinger equations  in terms of the action of the braid group
${\mathcal B}_n=\langle \beta_1,\dots,\beta_{n-1}\rangle$ on the
monodromy matrices $M_1,\dots,M_n$ was obtained in \cite{DM}. Let us
recall here the main ideas of this derivation.

According to Theorem \ref{th:iso}, any solution of the Schlesinger equations
can be continued analytically from a point ${\bf u}_0$ to any other point ${\bf u} \in X_n$
provided that the end-points are not the poles of the solution. The result
of the analytic continuation depends only on the homotopy class of the
path in $X_n$, i.e. one obtains a natural action of the pure braid group ${\mathcal P}_n$
$$
{\mathcal P}_n=\pi_1\left(X_n,{\bf u}_0\right)
$$
on the space of solutions of the Schlesinger equations. By using the
fact that thanks to Theorem \ref{th:RH} the solutions of the
Schlesinger equations are locally uniquely determined by the
monodromy matrices $M_1,\dots,M_n$, one can describe the procedure
of analytic  continuation by an action of the pure braid group on
the monodromy matrices. For technical simplicity,  we deal with the
action of the full braid group:
$$
{\mathcal B}_n=\pi_1\left(X_n\backslash S_n,{\bf u}_0\right),
$$
where $S_n$ is the symmetric group. This action is given by
\begin{eqnarray}\label{eq:braidDM}
&&
\beta_{i,i+1}(M_j)=M_j,\quad\hbox{for } j=1,\dots, i-1, i+2,\dots,n\nn\\
&&
\beta_{i,i+1}(M_i)=M_i M_{i+1}M_{i}^{-1},
\qquad \beta_{i,i+1}(M_{i+1})=M_i .
\end{eqnarray}
By using the skein relation, it is a straightforward computation to
show that on $G_{i,j}:=-\Tr(M_i M_j)$ the braid group action
coincides with the action (\ref{eq:braidch}), so that the action of
the mapping class group on the Teichm\"uller space of a disk with
$n$ marked points corresponds to the procedure of analytic
continuation of the corresponding solution to the Schlesinger
equations.

\section{Korotkin--Samtleben bracket}\label{se:KS}

In this section we  remind the Hamiltonian formulation of the
Schlesinger equations, the definition of the Korotkin--Samtleben
bracket and we show how to obtain the $A_n$ Poisson algebra from it.

\subsection{Hamiltonian formulation of the Schlesinger equations}

The Hamiltonian description of the Schlesinger equations in any dimension $m$ was derived
\cite{Hit2}
from the general construction of a Poisson bracket on the
space of flat connections in a principal $G$-bundle over a surface with
boundary
using Atiyah--Bott symplectic structure (see \cite{audin}). Explicitly this
approach yields the following well known formalism representing
the Schlesinger equations in
Hamiltonian form with $n$ time variables $u_1,\dots,u_{n}$ and
$n$ commuting time--dependent Hamiltonian flows on the dual space to the
direct sum of $n$ copies
of the Lie algebra ${\mathfrak{sl}}(m)$
\begin{equation}\label{lie}
{\mathfrak g}:=\oplus_{n} {\mathfrak{sl}}(m)\ni
\left( {A}_1, {A}_2, \dots, {A}_{n} \right).
\end{equation}

\begin{theorem} \cite{MJ1} \label{thm1.1}
The dependence of the solutions ${A}_k$, $k=1,\dots,n$, of the
Schlesinger equations upon the
variables $u_1,\dots,u_{n}$ is determined by  Hamiltonian systems
on (\ref{lie})
with time-dependent quadratic Hamiltonians
\begin{eqnarray}
&&
H_k= \sum_{l\neq k}
{{\rm Tr}\left({A}_k{A}_l\right)\over u_k-u_l},
\label{ham0}\\
&&
{\partial\over\partial u_k}{A}_l = \{{A}_l,H_k\}.\label{ham000}
\end{eqnarray}
\end{theorem}

Because of isomonodromicity the Hamiltonian equations (\ref{ham000}) can be restricted onto the symplectic leaves
$$
{\mathcal O}_1 \times \dots \times {\mathcal O}_n \in {\mathfrak g}
$$
obtained by fixation of the conjugacy classes ${\mathcal O}_1$,\dots,
${\mathcal O}_n$ of the matrices ${A}_1$, \dots, ${A}_n$.
The matrix ${A}_\infty$  is a common integral of  the Schlesinger
equations. Applying the procedure of symplectic reduction \cite{mars-wein}
one obtains the reduced symplectic space
\begin{eqnarray}\label{level}
&&
\left\{ A_1\in {\mathcal O}_1 , \dots, A_n\in {\mathcal O}_n,\,
A_\infty = \hbox{given diagonal matrix}\right\}
\nn\\
&&
\qquad\hbox{modulo simultaneous diagonal conjugations}.
\end{eqnarray}
The dimension of this reduced symplectic leaf in the generic situation
is equal to $2g$ where
$$
g={m(m-1)(n-1)\over2}-(m-1).
$$
In the $2\times 2$ case, i.e. for $m=2$ the dimension of the symplectic leaves is $2(n-2)$, which coincides with the
dimension of the Teichm\"uller space.

\subsection{Korotkin--Samtleben bracket}

The standard Lie--Poisson bracket on  ${\mathfrak g}^*$ can be represented in
$r$-matrix formalism:
$$
\left\{ {A}({\lambda_1})\otim_,{A}(\lambda_2)\right\}
= \left[ \one{A}({\lambda_1})+\two{A} (\lambda_2),{r}({\lambda_1} - \lambda_2)\right] ,
$$
where $r({z})=\frac{\Omega}{\lambda}$ is a {\it classical
$r$-matrix}, i.e a solution of the classican Yang--Baxter equation.
In the case of ${\mathfrak g}:=\oplus_{n} {\mathfrak{sl}}(m)$,
$\Omega$ is the {\it exchange matrix}\/ $\Omega=\sum_{i,j}\one{E_{ij}}\otim\two{E_{ji}}$ (we identify ${\mathfrak{sl}}(m)$ with its
dual by using the Killing form $(A,B)={\rm Tr}\, AB, \quad A, B \in
{\mathfrak{sl}}(m)$).

\noindent The standard Lie--Poisson bracket on
$\mathfrak{sl}(m,{\mathbb C})$ is mapped by the Riemann--Hilbert
correspondence to the {\it Korotkin--Samtleben bracket:}

\begin{eqnarray}\nn\label{eq:KSP}
&&
\left\{ M_i\otim_,M_i\right\}
=\frac{1}{2}\left(\two{M_i}\Omega\one{M_i} - \one{M_i} \Omega \two{M_i} \right)\\
&&
\\
&&
\left\{ M_i\otim_,M_j\right\}
=\frac{1}{2}\left(  \one{M_i} \Omega \two{M_j} +\two{M_j}\Omega\one{M_i}
-\Omega\one{M_i}\two{M_j}-\two{M_j}\one{M_i}\Omega\right) ,\quad\hbox{for}\,\,i<j.\nn
\end{eqnarray}
This bracket does not satisfy the Jacobi identity - however it
restricts to a Poisson bracket on the adjoint invariant objects.

\begin{lm}\label{lm:techKS}
The $A_n$ Poisson algebra (\ref{eq:NR}) is the Korotkin--Samtleben
bracket restricted to the adjoint invariant objects\footnote{The
Poisson  algebra (\ref{eq:NR}) was obtained in \cite{Ugaglia} as the
restriction of the Korotkin--Samtleben bracket to the traces of
products of $n\times n$ monodromy matrices.}
$$
G_{i,j}:=-\Tr (\gamma_i \gamma_j)=-\Tr (M_i M_j).
$$
\end{lm}

\begin{proof} We show how to prove relation:
\begin{equation}\label{eq:proof}
\left\{G_{i,k},G_{j,l}\right\}=2\left(G_{i,j}G_{k,l}-G_{i,l}G_{k,j}\right),\quad\hbox{for}\, \, i<j<k<l.
\end{equation}
By definition of $G_{i,j}$ we have:
\begin{eqnarray}
\left\{G_{i,k},G_{j,l}\right\}&=&\{\Tr(M_i M_k),\Tr(M_j M_l)\}=
\onetwo\Tr\left(\left\{\one M_i\otim_,\two M_j\right\} \one M_k\two M_l+\right.\nn\\
&&+\two M_j\left\{\one M_i\otim_,\two M_l\right\} \one M_k+
 \one M_i\left\{\one M_k\otim_,\two M_j\right\}\two M_l+\nn\\
&&+\left.\one M_i\two M_j\left\{\one M_k\otim_,\two M_l\right\}
\right).\nn
\end{eqnarray}
Applying the Korotkin--Samtleben bracket (\ref{eq:KSP}), one gets:
\begin{eqnarray}
\left\{G_{i,k},G_{j,l}\right\}&=&\frac{1}{2} \onetwo\Tr
\left[\left( \one{M_i} \Omega \two{M_j} +\two{M_j}\Omega\one{M_i}-\Omega\one{M_i}\two{M_j}-
\two{M_j}\one{M_i}\Omega\right) \one M_k\two M_l+\right.\nn\\
&&+\two M_j\left( \one{M_i} \Omega \two{M_l} +\two{M_l}\Omega\one{M_i}-\Omega\one{M_i}\two{M_l}-
\two{M_l}\one{M_i}\Omega\right) \one M_k+\label{eq:pbproof}\\
&&-\one M_i\left( \one{M_k} \Omega \two{M_j} +\two{M_j}\Omega\one{M_k}-\Omega\one{M_k}\two{M_j}-
\two{M_j}\one{M_k}\Omega\right) \two M_l+\nn\\
&&\left.+\one M_i\two M_j\left( \one{M_k} \Omega \two{M_l} +\two{M_l}\Omega\one{M_k}-\Omega\one{M_k}\two{M_l}-
\two{M_l}\one{M_k}\Omega\right)\right] .
\nn
\end{eqnarray}
This is a rather long computation. To simplify it we introduce a
graphic representation (this will be useful also in the proof of
Theorem \ref{th:main}) for the restriction of the
Korotkin--Samtleben bracket (\ref{eq:KSP}) to traces of products of
matrices. We represent the term of type
$$
\one M_i \two M_j \Omega \one M_k \two M_l
$$
as in Fig. \ref{pic:KS}.

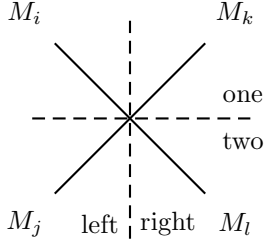
\begin{figure}[h]
\begin{pspicture}(-1.5,-1.5)(1.5,1.5)
 \psline(-1,-1)(1,1)
 \psline[linestyle=dashed](-1.3,0)(1.6,0)
 \psline[linestyle=dashed](0,-1.6)(0,1.3)
 \psline(-1,1)(1,-1)
 \rput(1.5,0.3){one}
 \rput(1.5,-0.3){two}
 \rput(-0.4,-1.4){left}
 \rput(0.5,-1.4){right}
 \rput(-1.4,1.4){$M_i$}
 \rput(-1.4,-1.4){$M_j$}
 \rput(1.4,1.4){$M_k$}
 \rput(1.4,-1.4){$M_l$}
\end{pspicture}
\caption{Graphic representation of the  Korotkin--Samtleben bracket.
The horizontal dashed line separates the spaces one (on top) and two (bottom).
The vertical dashed line divides left from right:
matrices at the left (resp. right) of $\Omega$ are on the left
(resp. right) half plane. The two diagonal lines represent the
freedom of transferring matrices trough the exchange matrix.}
\label{pic:KS}
\end{figure}

The trace is obtained by mapping all matrices to the right (or to
the left) through the diagonal lines and taking the trace of the
product of the contribution in space one with the contribution in
space two (see examples in Fig. \ref{pic:KSTR} and \ref{pic:KS2}).

\begin{figure}[h]
\begin{pspicture}(-4.5,-1.5)(-2.5,1.5)
 \psline(-1,-1)(1,1)
 \psline[linestyle=dashed](-1.3,0)(1.3,0)
 \psline[linestyle=dashed](0,-1.3)(0,1.3)
 \psline(-1,1)(1,-1)
 \rput(1.4,1.4){$M_j M_k$}
 \rput(1.4,-1.4){$M_i M_l$}
 \end{pspicture}
 \rput(0,1.5){$\rightarrow$}
 \begin{pspicture}(2.5,-1.5)(4.5,1.5)
 \psline(-1,-1)(1,1)
 \psline[linestyle=dashed](-1.3,0)(1.3,0)
 \psline[linestyle=dashed](0,-1.3)(0,1.3)
 \psline(-1,1)(1,-1)
 \rput(-1.4,1.4){$M_i$}
 \rput(-1.4,-1.4){$M_j$}
 \rput(1.4,1.4){$M_k$}
 \rput(1.4,-1.4){$M_l$}
\end{pspicture}
\caption{Graphic representation of the relation  $\onetwo\Tr\left(\one M_i \two
M_j \Omega \one M_k \two M_l\right)=
 \Tr\left(M_i M_l M_j M_k\right)$. }
\label{pic:KSTR}
\end{figure}

\begin{figure}[h]
\begin{pspicture}(.5,-1.5)(2.5,1.5)
 \psline(-1,-1)(1,1)
 \psline[linestyle=dashed](-1.3,0)(1.3,0)
 \psline[linestyle=dashed](0,-1.3)(0,1.3)
 \psline(-1,1)(1,-1)
 \rput(-1.4,1.4){$M_i^nM_k$}
 \rput(-1.4,-1.4){$M_j$}
 \rput(1.4,1.4){$\ID$}
 \rput(1.4,-1.4){$M_l$}
\end{pspicture}
 \rput(0,1.5){$\rightarrow$}
\begin{pspicture}(-2.5,-1.5)(-.5,1.5)
 \psline(-1,-1)(1,1)
 \psline[linestyle=dashed](-1.3,0)(1.3,0)
 \psline[linestyle=dashed](0,-1.3)(0,1.3)
 \psline(-1,1)(1,-1)
 \rput(1.4,1.4){$M_j$}
 \rput(1.4,-1.4){$M_i^n M_k M_l$}
  \end{pspicture}
\caption{Graphic representation of the relation $\onetwo\Tr\left(\one M_i^n\two
M_j \one M_l\Omega\two M_k\right) =\Tr\left(M_i^nM_kM_l
M_j\right)$.}
\label{pic:KS2}
\end{figure}

Using this graphic representation we immediately obtain that the
first two lines on the right hand side of (\ref{eq:pbproof}) cancel
each other and:
$$
\left\{G_{i,k},G_{j,l}\right\}= \Tr\left(M_i M_j M_l M_k +
M_j M_i M_k M_l-M_l M_i M_k M_j- M_i M_l M_j M_k
\right).
$$
Since $M_i^{-1}=-M_i$ and $\Tr M_i=0$ for all $i=1,\dots,n$ by the
skein relation (\ref{eq:skein}) we obtain the final result. In fact
\begin{eqnarray}
&&
 \Tr\left(M_i M_j M_l M_k\right) = \Tr\left( M_i M_j \right) \Tr\left( M_l M_k \right)-
  \Tr\left( M_iM_j  M_k M_l\right)
  \nn\\
&&  \Tr\left(M_j M_i M_k M_l\right) = \Tr\left( M_i M_j \right) \Tr\left( M_l M_k \right)-
  \Tr\left( M_i M_j  M_k M_l\right)\nn\\
  &&
  \Tr\left(M_l M_i M_k M_j\right) = \Tr\left( M_i M_l\right) \Tr\left( M_jM_k \right)-
  \Tr\left( M_l M_i  M_j M_k\right)\nn\\
  &&
  \Tr\left(M_i M_   l M_j M_k\right) = \Tr\left( M_i M_l\right) \Tr\left( M_j M_k \right)-
  \Tr\left( M_l M_i  M_j M_k\right).\nn
  \end{eqnarray}
  The other relations can be obtained in a similar way.
\end{proof}

\section{Clashing of poles}\label{se:clash}

In this section we consider a Fuchsian system of the form
(\ref{N1in}), with monodromy matrices $M_1,\dots,M_n$ as was described
in Sec.~\ref{se:schl}, and study its relation with a
new Fuchsian system
\begin{equation}
{{\rm d}\over{\rm d}\lambda}\ttilde\Phi=\left(
\sum_{k=1}^{\ttilde n}{\ttilde{A}_k\over \lambda-u_k} \right)\ttilde \Phi,
\label{eq:clash}
\end{equation}
with $\ttilde n=n-m+1$  for some positive integer $m<n$, monodromy matrices
$\ttilde M_1,\dots,\ttilde M_{\ttilde n}$, where
$$
\ttilde M_i=M_i, \quad\hbox{for } i=1,\dots, \ttilde n-1,\quad\hbox{and }
\ttilde M_{\ttilde n}=M_{n-m+1}\dots M_{n-1} M_n.
$$
The main idea is that system (\ref{eq:clash}) is obtained from
system (\ref{N1in}) by clashing $m$ poles \cite{JMS}. To this aim we
set
$$
\ttilde{\bf u}:=(u_1,\dots,u_{\ttilde n-1}),
$$
$$
u_j:=t v_j, \quad j=\ttilde n,\dots,\ttilde n+m-1=n
$$
and
\begin{eqnarray}\label{eq:AtoB}
&&
A_i(\ttilde{\bf u},t):= A_i(\ttilde{\bf u}, t v_{\ttilde n},\dots, t v_{n}), \quad\hbox{for }
i=1,\dots,\ttilde n-1,\nn\\
&&
\\
&&
B_j(\ttilde{\bf u},t):= A_{\ttilde n-1+j}(\ttilde{\bf u}, t v_{\ttilde n},\dots, t v_{n}),
\quad\hbox{for } j=1,\dots,m.\nn
\end{eqnarray}
The Schlesinger equations in the variable $t$ become:
\begin{equation}\label{eq:sch-red}
\frac{\partial A_i}{\partial t} = \sum_{j=1}^m  \frac{v_j}{t v_j-u_i}[B_j,A_i],\qquad
\frac{\partial B_j}{\partial t} =\frac{1}{t}\sum_{k\neq j} [B_k,B_j]
-\sum_{i=1}^{\ttilde n-1}  \frac{v_j}{t v_j-u_i}[B_j,A_i].
\end{equation}

The following theorem gives the conditions under which system
(\ref{eq:clash}) can be obtained as the limit for $t\to 0$ of system
(\ref{N1in}). We state it in full generality, namely for any dimension $m$ of the Fuchsian systems involved.

\begin{theorem}\cite{JMS}\label{th:JMS}
Let $A_1^0,\dots,A_{\ttilde n-1}^0,B_1^0,\dots,B_m^0$ be some constant  in $t$ matrices such that
$$
\Lambda:=\sum_{j=1}^m B_j^0
$$
has eigenvalues $\lambda_1,\dots,\lambda_n$ such that the following technical assumption is satisfied:
\be\label{eq:techn}
\thi:=\max_{i,j=1,\dots,m}\left\{|\Re(\lambda_i)-\Re(\lambda_j) |\right\})\in[0,1[
\ee
and let $K$ be a constant such that
$$
| A_i^0|<K \hbox{ for } i=1,\dots,\ttilde n-1,\qquad
| B_j^0|<K \hbox{ for } j=1,\dots,m,\qquad
$$
then the following three statements are true:

For any $\varphi$, there exists an $\varepsilon >0$ such that the Schlesinger equations
(\ref{eq:sch-red}) admits a unique solution in the sector
$\{t\in\mathbb C, \, |t|<\varepsilon,\, \arg(t)<\varphi\}$ such that
the following estimates on the asymptotic behavior hold true:
\begin{equation}\label{eq:estimates1}
|A_i(t)- A_i^0|\leq K |t |^{1-\sigma} \quad
\left| t^{-\Lambda}(A_i(t)-A_i^0) t^\Lambda\right| \leq K |t|^{1-\sigma},
\end{equation}
\begin{equation}\label{eq:estimates2}
\left|t^{-\Lambda}B_j(t)t^\Lambda-B_j^0 \right|\leq K |t|^{1-\sigma}
\end{equation}
for some $\sigma$ such that $1>\sigma>\thi$.

Let $\Phi(\lambda, t)$ be the corresponding solution of the system
(\ref{N1in}) normalized at infinity. Then the limit
$\ttilde\Phi(\lambda):= \lim_{t\to 0} \Phi(\lambda,t)$ exists and it
satisfies the system (\ref{eq:clash}) with
\begin{equation}
\label{eq:tildesystem}
u_{\ttilde n}=0,\quad
\ttilde A_i = A_i^0\hbox{ for } i=1,\dots,\ttilde n-1,\quad\hbox{and }\,
\ttilde A_{\ttilde n}=\Lambda.
\end{equation}

The corresponding monodromy matrices of the system (\ref{eq:clash}) under the conditions (\ref{eq:tildesystem}) are
$\ttilde M_1,\dots,\ttilde M_{\ttilde n}$, where
$$
\ttilde M_i=M_i, \quad\hbox{for } i=1,\dots, \ttilde n-1,\quad\hbox{and }\,
\ttilde M_{\ttilde n}=M_{n-m+1}\dots M_{n-1} M_n.
$$
\end{theorem}

In our case, i.e. when system (\ref{N1in}) comes from the $A_n$
algebra, i.e. the monodromy matrices are given by (\ref{eq:mon-d})
and (\ref{eq:basisn}), we are always able to clash odd numbers of
poles. In fact for an odd number $m=n-\ttilde n+1$,
$$
\Tr(M_{n-m+1}\dots M_{n-1} M_n)=
\Tr (\gamma_{n-m+1}\dots \gamma_{n-1}\dots \gamma_{n})=2 \cos(\pi \thi)
$$
is always a real number bigger than $2$. This implies that $\thi$ is purely imaginary, hence
$\Re(\thi)=0$ and the hypotheses of theorem \ref{th:JMS} are satisfied.

\section{$\mathfrak{D}_n$ algebras}\label{se:newDn}

In this section, we interpret system (\ref{eq:clash}) as a system on
an annulus with $n$ marked points, the hole at the center of the
annulus being the result of clashing $m$ poles. For convenience we
change our notation: we start with a system with $n+m$ poles, clash
$m$ of them and we call the final number of orbifold points $n$.  We
denote the monodromy matrix around the hole as
$$
M_h := M_{n+1} \dots M_{n+m-1}  M_{n+m}.
$$
Now the paths that join the  points $u_i$ and $u_j$ winding $k$
times (possibly with self-intersections) around the new hole
perimeter will contribute to the Poisson algebra. The traces of the
corresponding elements in the Fuchsian group are:
\be
\label{eq:Gijk}
G_{i,j}^{(k)}:=-\Tr\left(M_i M_h^k M_j M_h^{-k}\right).
\ee
Note that
\be\label{eq:Gijk1}
G_{i,j}^{(k)} = G_{j,i}^{(-k)},
\ee
so that in particular the {\it level $0$ elements}\/ $G_{ij}^{(0)}$ are symmetric and
\be\label{eq:Gijk2}
G_{i,i}^{(0)}= 2.
\ee

The Poisson algebra for the elements $G_{ij}^{(k)}$ is described in the following:

\begin{theorem}\label{th:main}
The geodesic length functions $G_{ij}^{(k)}$ satisfy the following Poisson relations,
for $0\le k$
\bea\label{eq:newDn}
\{G_{j,i}^{(0)},G_{p,l}^{(k)}\}&=&\bigl(\epsilon(j-l)-\epsilon(i-l)\bigr)(G_{l,i}^{(0)}G_{p,j}^{(k)}
-G_{l,j}^{(0)}G_{p,i}^{(k)}) +\\
&&+ \bigl(\epsilon(j-p)-\epsilon(i-p)\bigr)(G_{p,i}^{(0)}G_{j,l}^{(k)}
-G_{p,j}^{(0)}G_{i,l}^{(k)})\nn
\eea
and for $0<m\leq k$:
\bea\label{eq:newDn1}
&{}&\left\{G_{j,i}^{(m)},G_{p,l}^{(k)}\right\}=
\nonumber
\\
&{}&
\epsilon(i-l)(G^{(k)}_{p,i}G^{(m)}_{j,l}-G^{(0)}_{i,l}G^{(k-m)}_{p,j})
+\epsilon(i-p)(G^{(m)}_{j,p}G^{(k)}_{i,l}-G^{(0)}_{i,p}G^{(k+m)}_{j,l})
\nonumber
\\
&{}&
+\epsilon(j-l)(G^{(k)}_{p,j}G^{(m)}_{l,i}-G^{(0)}_{j,l}G^{(k+m)}_{p,i})
+\epsilon(j-p)(G^{(m)}_{p,i}G^{(k)}_{j,l}-G^{(0)}_{j,p}G^{(k-m)}_{i,l})
\nonumber
\\
&{}&+\Bigl[
G^{(k+m)}_{p,i}G^{(0)}_{j,l}+2G^{(k+m-1)}_{p,i}G^{(1)}_{j,l}+\cdots
+2G^{(k+1)}_{p,i}G^{(m-1)}_{j,l}+G^{(k)}_{p,i}G^{(m)}_{j,l}\Bigr.
\\
&{}&
-G^{(m)}_{p,i}G^{(k)}_{j,l}-2G^{(m-1)}_{p,i}G^{(k+1)}_{j,l}-\cdots
-2G^{(1)}_{p,i}G^{(k+m-1)}_{j,l}-G^{(0)}_{p,i}G^{(k+m)}_{j,l}
\nonumber
\\
&{}&
+G^{(k-m)}_{i,l}G^{(0)}_{j,p}+2G^{(k-m+1)}_{i,l}G^{(1)}_{j,p}+\cdots
+2G^{(k-1)}_{i,l}G^{(m-1)}_{j,p}+G^{(k)}_{i,l}G^{(m)}_{j,p}
\nonumber
\\
&{}&\Bigl.
-G^{(0)}_{l,i}G^{(k-m)}_{p,j}-2G^{(1)}_{l,i}G^{(k-m+1)}_{p,j}-\cdots
-2G^{(m-1)}_{l,i}G^{(k-1)}_{p,j}-G^{(m)}_{l,i}G^{(k)}_{p,j}\Bigr].
\nonumber
\eea
where $\epsilon$ denotes the sign function ($\epsilon(x)=\{-1,x<0;\ 0,x=0;\ 1,x>0\}$).
We introduce the generating function
\be
{\mathcal G}_{i,j}(\lambda):= {\mathcal A}^{(0)}_{i,j}+\sum_{k=1}^\infty G_{i,j}^{(k)}\lambda^{-k},
\label{G-def}
\ee
where ${\mathcal A}^{(0)}$ is an upper-triangular matrix with the entries
\be\label{eq:MatA}
{\mathcal A}^{(0)}_{i,j}=\left\{\begin{array}{l}
G^{(0)}_{i,j}\qquad\hbox{for}\quad i<j\\
1 \qquad\hbox{for}\quad i=j\\
0 \qquad\hbox{for}\quad i>j.\\
\end{array}
\right.
\ee
The Poisson bracket then becomes
\begin{eqnarray}
\left\{{\mathcal G}_{j,i}(\lambda),{\mathcal G}_{p,l}(\mu)\right\}&=&
\left(\epsilon(j-p)-\frac{\lambda+\mu}{\lambda-\mu}\right){\mathcal G}_{p,i}(\lambda) {\mathcal G}_{j,l}(\mu)+
\nn\\
&&
+\left(\epsilon(i-l)+\frac{\lambda+\mu}{\lambda-\mu}\right) {\mathcal G}_{p,i}(\mu) {\mathcal G}_{j,l}(\lambda)+
\nn\\
&&
+\left(\epsilon(i-p)-\frac{1+\lambda\mu}{1-\lambda\mu}\right)
{\mathcal G}_{j,p}(\lambda){\mathcal G}_{i,l}(\mu)+\nn\\
&&+
\left(\epsilon(j-l)+\frac{1+\lambda\mu}{1-\lambda\mu}  \right){\mathcal G}_{l,i}(\lambda){\mathcal G}_{p,j}(\mu)
\label{eq:Yangian}
\end{eqnarray}
This is an abstract infinite-dimensional Poisson algebra.
\end{theorem}

Note that al level zero the relation (\ref{eq:Gijk1}) produces the Nelson--Regge
algebra (\ref{eq:NR})  for $G_{ij}^{(0)}$.

Before proving this theorem it is important to stress that the
number $m$ of clashed poles does not appear in the formulae. Indeed,
if we consider the elements $G^{(k)}_{i,j}$ as
infinitely many independent elements, the brackets (\ref{eq:newDn}) and
(\ref{eq:newDn1}) define a Poisson algebra satisfying the Jacobi
identity.

\begin{df}
We call the Poisson relations  (\ref{eq:newDn}) and (\ref{eq:newDn1})
the ${\mathfrak D}_n$-{\em algebra} of the elements $G^{(k)}_{i,j}$.
\end{df}

In Subsec.~\ref{subse:yangian} we prove that $\mathfrak D_n$
is the semiclassical limit of the twisted $q$--Yangian
$Y'_q(\mathfrak{o}_n)$ for the orthogonal Lie algebra
$\mathfrak{o}_n$ introduced in \cite{MRS}. This gives an alternative
proof to the Jacobi identity.

Geometrically, we interpret the  ${\mathfrak D}_n$ algebra as the
algebra of geodesic length functions on a annulus with $n$ marked
points. To understand this it is enough to look at the geometric
construction of a hole by clashing two poles.\footnote{From a
geometric view point we have no constraint on the number of clashed
poles. The constraint that $m$ needs to be odd arises when we want
to carry out the clashing in the analytic framework.}

\subsection{Example: Clashing two poles}

Observe first that the product of two traceless elements  is a
hyperbolic element. This is consistent with the fact that by clashing
two points we get a hole. We begin with the algebra $A_{n+2}$ and
interpret the element $M_h=M_{n+1}M_{n+2}$ as an element
corresponding to going around a new hole. We consider the Fuchsian
group ${\Delta}_{0,2,n}$ generated by $M_h$ and $M_i$,
$i=1,\dots,n$. Obviously, the group thus constructed is a subgroup
of $\Delta_{0,1,n+2}$, the Fuchsian group in the $A_{n+2}$ case.

When clashing two poles $u_{n+1}$ and $u_{n+2}$ we create a hole
thus obtaining a new hole perimeter, the loop containing $u_{n+1}$
and $u_{n+2}$. As a consequence, we can consider the subgraph in the
right-hand side of Fig.~\ref{fi:An-Dn} instead of the tree-like
subgraph in the left-hand side.
\begin{figure}[h]
{\psset{unit=0.7}
\begin{pspicture}(-7,-3.5)(7,1)
\newcommand{\FORK}{%
\pcline[linewidth=1pt](-1,1)(1,1)
\pcline[linewidth=1pt](-1,0)(-0.5,0)
\pcline[linewidth=1pt](1,0)(0.5,0)
\pcline[linewidth=1pt](-0.5,0)(-0.5,-1)
\pcline[linewidth=1pt](0.5,0)(0.5,-1)
\pcline[linewidth=1pt](-0.5,-1)(-1.5,-2)
\pcline[linewidth=1pt](0.5,-1)(1.5,-2)
\pcline[linewidth=1pt](0,-1.7)(-0.9,-2.6)
\pcline[linewidth=1pt](0,-1.7)(0.9,-2.6)
\pcline[linewidth=1pt](-0.3,-1.4)(-1.2,-2.3)
\pcline[linewidth=1pt](0.3,-1.4)(1.2,-2.3)
\psarc[linewidth=1pt](0,-1.7){.42}{45}{135}
\pscircle*(-1.2,-2.3){.1}
\pscircle*(1.2,-2.3){.1}
\pcline[linecolor=red, linestyle=dashed, linewidth=1pt](-1,.2)(-0.3,.2)
\pcline[linecolor=red, linestyle=dashed, linewidth=1pt](1,.2)(0.3,.2)
\pcline[linecolor=red, linestyle=dashed, linewidth=1pt](-0.3,.2)(-0.3,-1.1)
\pcline[linecolor=red, linestyle=dashed, linewidth=1pt](0.3,.2)(0.3,-1.1)
\pcline[linecolor=red, linestyle=dashed, linewidth=1pt](-0.3,-1.1)(-1.35,-2.15)
\pcline[linecolor=red, linestyle=dashed, linewidth=1pt](0.3,-1.1)(1.35,-2.15)
\pcline[linecolor=red, linestyle=dashed, linewidth=1pt](-0.15,-1.55)(-1.05,-2.45)
\pcline[linecolor=red, linestyle=dashed, linewidth=1pt](0.15,-1.55)(1.05,-2.45)
\psarc[linecolor=red, linestyle=dashed, linewidth=1pt](0,-1.7){.2}{45}{135}
\psarc[linecolor=red, linestyle=dashed, linewidth=1pt](-1.2,-2.3){.22}{135}{315}
\psarc[linecolor=red, linestyle=dashed, linewidth=1pt](1.2,-2.3){.22}{-135}{45}
\rput(-0.6,-0.5){\makebox(0,0)[rc]{$Y$}}
\rput(-1.4,-1.8){\makebox(0,0)[rb]{$Z_{n+1}$}}
\rput(1.4,-1.8){\makebox(0,0)[lb]{$Z_{n+2}$}}
}
\newcommand{\SPOON}{%
\pcline[linewidth=1pt](-1,1)(1,1)
\pcline[linewidth=1pt](-1,0)(-0.5,0)
\pcline[linewidth=1pt](1,0)(0.5,0)
\pcline[linewidth=1pt](-0.5,0)(-0.5,-1.1)
\pcline[linewidth=1pt](0.5,0)(0.5,-1.1)
\psarc[linewidth=1pt](0,-2.5){1.5}{-250}{70}
\pscircle[linewidth=1pt](0,-2.5){.5}
\pcline[linecolor=red, linestyle=dashed, linewidth=1pt](-1,.2)(-0.3,.2)
\pcline[linecolor=red, linestyle=dashed, linewidth=1pt](1,.2)(0.3,.2)
\pcline[linecolor=red, linestyle=dashed, linewidth=1pt](-0.3,.2)(-0.3,-1.3)
\pcline[linecolor=red, linestyle=dashed, linewidth=1pt](0.3,.2)(0.3,-1.3)
\psarc[linecolor=red, linestyle=dashed, linewidth=1pt](0,-2.5){1.24}{-256}{76}
\rput(-0.6,-0.5){\makebox(0,0)[rc]{$Y_h$}}
\rput(-1.2,-1.3){\makebox(0,0)[rb]{$Z_{h}$}}
}
\rput(-3,0){\FORK}
\rput(3,0){\SPOON}
\rput(0,-.5){\makebox(0,0){$\equiv$}}
\end{pspicture} }
\caption{\small }
\label{fi:An-Dn}
\end{figure}

The product of matrices
corresponding to paths that go around $u_{n+1}$ and $u_{n+2}$  in
the left-hand side is then preserved: we have the matrix equality
$$
X_YRX_{Z_{n+2}}FX_{Z_{n+2}}RX_{Z_{n+1}}FX_{Z_{n+1}}RX_Y=X_{Y_h}RX_{Z_h}RX_{Y_h}.
$$
This equality holds provided the new coordinates $Z_h$ and $Y_h$ are defined
as follows:
\bea
&&
Y+Z_{n+1}+Z_{n+2}=Y_h+\frac{Z_h}{2},
\label{rel1}
\\
&&
G_{n+1,n+2}=e^{Z_{n+1}+Z_{n+2}}+e^{-Z_{n+1}-Z_{n+2}}+e^{Z_{n+1}-Z_{n+2}}=e^{Z_h/2}+e^{-Z_h/2}.
\label{rel2}
\eea
Obviously, $G_{n+1,n+2}$ commutes with all elements of the (sub)group ${\Delta}_{0,2,n}$.

To show that  the Fuchsian group
${\Delta}_{0,2,n}$ generated by $M_h$ and $M_i$, $i=1,\dots,n$ can be seen as a subgroup of
$\Delta_{0,1,n+2}$, the Fuchsian group in the $A_{n+2}$ case, we show that  the elements
${G}^{(k)}_{ij}$ can be expressed in terms of elements in $\Delta_{0,1,n+2}$. For example:
\be
{G}^{(1)}_{ij}=G_{i,n+2}G_{n+1,n+2}G_{j,n+1}-G_{i,n+1}G_{j,n+1}-G_{i,n+2}G_{j,n+2}+G_{i,j},
\label{C2}
\ee
where (\ref{C2}) follows from the skein relation (\ref{eq:skein}), see figure \ref{fig:sk}.

\begin{figure}
{\psset{unit=0.5}
\begin{pspicture}(-7,-5)(7,3)
\newcommand{\BASE}{%
\pcline[linecolor=green, linestyle=dashed, linewidth=1pt](-2.5,0)(2.5,0)
\pscircle*(-2,0){0.1}
\psarc[linewidth=0.5pt](-2,0){0.3}{-180}{0}
\rput(-2,-0.5){\makebox(0,0)[ct]{\tiny$i$}}
\pscircle*(-1,0){0.1}
\psarc[linewidth=0.5pt](-1,0){0.3}{-180}{0}
\rput(-1,-0.5){\makebox(0,0)[ct]{\tiny$j$}}
\pscircle*(1,0){0.1}
\rput(1,-0.5){\makebox(0,0)[ct]{\tiny $n{+}1$}}
\pscircle*(2,0){0.1}
\rput(2,-0.5){\makebox(0,0)[ct]{\tiny $n{+}2$}}
}
\rput(-10,0){\BASE}
\psarc[linewidth=0.5pt](-10.05,0){.65}{0}{180}
\psarc[linewidth=0.5pt](-10.05,0){1.25}{0}{180}
\psarc[linewidth=0.5pt](-9.95,0){1.75}{0}{180}
\psarc[linewidth=0.5pt](-9.95,0){2.35}{0}{180}
\psarc[linewidth=0.5pt](-9.1,0){0.3}{-180}{0}
\psarc[linewidth=0.5pt](-8.9,0){0.3}{-180}{0}
\psarc[linewidth=0.5pt](-8.1,0){0.3}{-180}{0}
\psarc[linewidth=0.5pt](-7.9,0){0.3}{-180}{0}
\psarc[linewidth=0.5pt](-8.5,0){.1}{0}{180}
\psarc[linewidth=0.5pt](-8.5,0){.7}{0}{180}
\rput(-6.75,0.5){\makebox(0,0){$-$}}
\rput(-3.5,0){\BASE}
\psarc[linewidth=0.5pt](-3.55,0){.65}{0}{180}
\psarc[linewidth=0.5pt](-3.55,0){1.25}{0}{180}
\psarc[linewidth=0.5pt](-3.95,0){1.25}{0}{180}
\psarc[linewidth=0.5pt](-3.95,0){1.85}{0}{180}
\psarc[linewidth=0.5pt](-2.4,0){0.3}{-180}{0}
\psarc[linewidth=0.5pt](-2.6,0){0.3}{-180}{0}
\rput(0,0.5){\makebox(0,0){$-$}}
\rput(3.5,0){\BASE}
\psarc[linewidth=0.5pt](3.55,0){1.75}{0}{180}
\psarc[linewidth=0.5pt](3.55,0){2.35}{0}{180}
\psarc[linewidth=0.5pt](3.95,0){1.15}{0}{180}
\psarc[linewidth=0.5pt](3.95,0){1.75}{0}{180}
\psarc[linewidth=0.5pt](5.4,0){0.3}{-180}{0}
\psarc[linewidth=0.5pt](5.6,0){0.3}{-180}{0}
\rput(6.75,0.5){\makebox(0,0){$+$}}
\rput(10,0){\BASE}
\psarc[linewidth=0.5pt](8.5,0){.2}{0}{180}
\psarc[linewidth=0.5pt](8.5,0){.8}{0}{180}
\rput(-10,-4){\makebox(0,0){$=$}}
\rput(-6,-4){\BASE}
\psarc[linewidth=0.5pt](-6.05,-4){.65}{0}{180}
\psarc[linewidth=0.5pt](-6.25,-4){1.05}{0}{180}
\psarc[linewidth=0.5pt](-5.95,-4){2.35}{0}{180}
\psarc[linewidth=0.5pt](-5.75,-4){1.95}{0}{180}
\psarc[linewidth=0.5pt](-5,-4){0.2}{-180}{0}
\psarc[linewidth=0.5pt](-5,-4){0.4}{-180}{0}
\psarc[linewidth=0.5pt](-4,-4){0.2}{-180}{0}
\psarc[linewidth=0.5pt](-4,-4){0.4}{-180}{0}
\psarc[linewidth=0.5pt](-4.5,-4){.1}{0}{180}
\psarc[linewidth=0.5pt](-4.5,-4){.3}{0}{180}
\rput(.5,-4){$={G}^{(1)}_{ij}$}
\end{pspicture}
}
\caption{}
\label{fig:sk}
\end{figure}

\subsection{Proof of Theorem \ref{th:main}}

The proof of this theorem is organized as follows: first we
construct the Poisson algebra of the elements $G^{(k)}_{i,j}$
considered as elements in a subgroup of $\Delta_{0,1,n+m}$, the
Fuchsian group in the $A_{n+m}$ case. Then we prove that the
abstract algebra ${\mathfrak D}_n$  is actually an infinite
dimensional Poisson algebra (see Appendix B).

We begin by extending the Korotkin--Samtleben bracket  to the
monodromy matrices of the system (\ref{eq:clash}):

\begin{prop}\label{lem-merging}
If $M_i$, $i=1,\dots,n+m$, satisfy the brackets (\ref{eq:KSP}), then
the set of new matrices ${\ttilde M}_i$, $i=1,\dots,n+1$ such that ${\ttilde M}_i=M_i$
for $i=1,\dots,n$ and ${\ttilde M}_{n+1}=M_h$ satisfy the same
brackets (\ref{eq:KSP}). We also have the following brackets for the powers of $M_h$:
\be
\label{MM-i-tilde}
\left\{ M_i\otim_,M^k_h\right\}
=\frac{1}{2}\left(  {\one{M}}_i \Omega {\two{M}}{}^k_h
+{\two{{M}}}{}^k_h\Omega{\one{{ M}}}_i-\Omega{\one{{ M}}}_i
{\two{{M}}}{}^k_h
-{\two{M}}{}^k_h{\one{{M}}}_i\Omega\right) ,\forall k\in{\mathbb Z},
\ee
\be
\label{MM-tilde}
\left\{{ M}_{h}\otim_,{M}^k_{h}\right\}
=\frac{1}{2}\left( {\two{{M}}}{}^k_{h}  \Omega {\one{{ M}}}_{h}
+{\one{{ M}}}_{h}{\two{{M}}}{}^k_h\Omega-\Omega{\one{{M}}}_h
{\two{{ M}}}{}^k_h
-{\one{{ M}}}_h\Omega{\two{{ M}}}{}^k_h\right) ,
\forall k\in{\mathbb Z}.
\ee
\end{prop}

\proof
First we prove the case when $m=2$, i.e. $M_h=M_{n+1} M_{n+2}$:
\begin{eqnarray}\nn
\left\{ M_i\otim_,M_h\right\}
\!\!\!\!\!
&=&
\!\!\!\!\!
\frac{1}{2}\left(\one{M_i}\Omega\two{M_{n+1} } +   \two{M_{n+1}} \Omega \one{M_i}-
\Omega\one{M_i} \two{M_{n+1} }  - \two{M_{n+1}} \one{M_i} \Omega
\right) \two{M_{n+2} }-\nn\\
\!\!\!\!\!
&+&
\!\!\!\!\!
\frac{1}{2} \two{M_{n+1} }\left(\one{M_i}\Omega\two{M_{n+2} } +   \two{M_{n+2} }\Omega \one{M_i}-
\Omega\one{M_i} \two{M_{n+2} }  - \two{M_{n+2}} \one{M_i} \Omega
\right)=\nn\\
\!\!\!\!\!
&=&
\!\!\!\!\!
\frac{1}{2} \left(\one{M_i}\Omega\two{M_{n+1}}\two{M_{n+2} }
+ \two{M_{n+1} }  \two{M_{n+2} }\Omega \one{M_i}-
\Omega\one{M_i}\two{M_{n+1} }
 \two{M_{n+2} }\right.  -\nn\\
 \!\!\!\!\!
 &-&
\!\!\!\!\!
\left. \two{M_{n+1} } \two{M_{n+2}} \one{M_i} \Omega\right).
\nn
\end{eqnarray}
The proof for any $m$ is a straightforward induction on $m$ and we
omit it. The proof of (\ref{MM-i-tilde}) is a simple consequence of
the fact that $M_1,\dots,M_n,M_h$ satisfy the brackets
(\ref{eq:KSP}). The proof of (\ref{MM-tilde}) is again the same sort
of computations, using the freedom of transferring
${\one{{M}}}_h\Omega=\Omega{\two{{M}}}_h$ and
${\two{{M}}}_h\Omega=\Omega{\one{{M}}}_{h}$ through the exchange
matrix $\Omega$.\endproof

\proof
To conclude the proof of Theorem  \ref{th:main} we first outline how to obtain (\ref{eq:newDn}). We compute
$$
\left\{G_{j,i}^{(m)},G_{p,l}^{(k)}
\right\}=\left\{\Tr\left(M_jM_h^{m}M_IM_h^{-m}\right),\Tr\left(M_pM_h^{k}M_lM_h^{-k}\right)
\right\},
$$
by applying the Leibnitz rule and  by using the extended
Korotkin--Samtleben bracket (\ref{MM-i-tilde}) and (\ref{MM-tilde}).
By a rather long but straightforward computation using the graphic
representation explained in  the proof of Lemma \ref{lm:techKS} we
arrive to a sum of traces that we break up by using the skein
relation.

The equivalence of (\ref{eq:newDn}) and (\ref{eq:Yangian}) is a standard exercise; the most important
part is to verify the Jacobi relations for these brackets. For this, the representation (\ref{eq:Yangian})
seems to be the most convenient one. The computation is purely technical and we present it
in Appendix B. \endproof

\subsection{$\mathfrak D_n$ as semi--classical limit of $Y'_q(\mathfrak o_n)$}\label{subse:yangian}

Here, we prove that $\mathfrak D_n$ is the semiclassical limit
of the twisted $q$--Yangian $Y'_q(\mathfrak{o}_n)$ for the
orthogonal Lie algebra $\mathfrak{o}_n$ introduced in \cite{MRS}.
The latter is the algebra generated by the matrix elements
$G_{i,j}^{(k)}$, $i,j=1,\dots,n$, $k\in\mathbb Z_{\geq 0}$ subject to the
defining relations:
\be\label{eq:molev}
R(\lambda,\mu)\one{\mathcal G}(\lambda)R(\lambda^{-1},\mu)^{T_1}\two{\mathcal G}(\mu)=
\two{\mathcal G}(\mu)R(\lambda^{-1},\mu)^{T_1}\one{\mathcal G}(\lambda)R(\lambda,\mu)
\ee
where the R-matrix is given by
\begin{eqnarray}\label{eq:qR}
R(\lambda,\mu)&=&(\lambda-\mu)\sum_{i\neq j} E_{ii}\otim E_{jj}+
(q^{-1}\lambda-q\mu)\sum_{i}E_{ii}\otim E_{ii} +\nn \\
&+&(q^{-1}-q)\lambda\sum_{i<j}E_{ij}\otim E_{ji}+(q^{-1}-q)\mu\sum_{i>j} E_{ij}\otim E_{ji}
\end{eqnarray}
and it is a solution of the Yang--Baxter equation and the apex $T_1$ indicates the transposition in space one.

The semiclassical limit is obtained by putting $q=-\exp(i \pi
\hbar)$ and taking the terms of order $\hbar$ in the Laurent
expansion as $\hbar$ tends to zero. The $R$ matrix is expanded as
$$
R(\lambda,\mu) = (\lambda-\mu) \ID\otimes \ID + i\pi \hbar\, r(\lambda,\mu),
$$
where $r$ is a classical $r$--matrix:
\be\label{eq:clr}
r(\lambda,\mu)= (\lambda+\mu)\sum_{ij} E_{ii}\otim E_{ii}+ 2 \lambda \sum_{i<j}E_{ij}\otim E_{ji} +
2 \mu \sum_{i>j} E_{ij}\otim E_{ji},
\ee
while the matrix $\mathcal G(\lambda)$ remains the same. The reflection equation (\ref{eq:molev}) in the semiclassical limit becomes:
\be\label{eq:molev-cl}
\left\{\one{\mathcal G}(\lambda)\otim_,
\two{\mathcal G}(\mu)\right\}=  \left[\frac{r(\lambda,\mu)}{\lambda-\mu},\one{\mathcal G}(\lambda)\two{\mathcal G}(\mu)
\right] + \one{\mathcal G}(\lambda) \frac{r(\lambda^{-1},\mu)^{T_1}}{\lambda^{-1}-\mu}\two{\mathcal G}(\mu)-
\two{\mathcal G}(\mu) \frac{r(\lambda^{-1},\mu)^{T_1}}{\lambda^{-1}-\mu}
\one{\mathcal G}(\lambda).
\ee
It is a straightforward computation to show that formula
(\ref{eq:molev-cl}) coincides with our formula
(\ref{eq:Yangian}).\footnote{Observe that in \cite{MRS} a different
normalization for $\mathcal G$ was used: the level $0$ was taken
lower triangular. As a consequence the $R$-matrix was globally
transposed.}

\subsection{Braid-group relations for $G_{i,j}^{(k)}$}\label{subse:braid}

The braid group (and the mapping class group) in the case of ${\mathfrak D}_n$
algebra is generated by $n$ generators. Besides the standard
generators $\beta_{i,i+1}$, $i=1,\dots,n-1$, each of which
interchanges the $i$th and the $i+1$th orbifold points, we have one
additional generator $\beta_{1,n}$ interchanging the first and the
last, $n$th orbifold points. The action of this element is not
trivial due to the presence of the new hole:

\begin{theorem}
Let ${\mathcal G}^{(k)}$ be a matrix of entries $G_{ij}^{(k)}$. The action of the braid-group elements
$\beta_{i,i+1}$ for $i=1,\dots,n-1$ is given by
\be\label{eq:DnBraid1}
\beta_{i,i+1}{\mathcal G}^{(k)}=\ttilde{\mathcal G}^{(k)}:
\left\{
\begin{array}{ll}
  {\ttilde G}^{(k)}_{i+1,j}=G^{(k)}_{i,j}, & j>i+1,\\
  {\ttilde G}^{(k)}_{j,i+1}=G^{(k)}_{j,i}, & j<i, \\
  {\ttilde G}^{(k)}_{i,j}=G^{(k)}_{i,j}G^{(0)}_{i,i+1}-G^{(k)}_{i+1,j}, & j>i+1, \\
  {\ttilde G}^{(k)}_{j,i}=G^{(k)}_{j,i}G^{(0)}_{i,i+1}-G^{(k)}_{j,i+1}, & j<i, \\
\multispan{2} \hbox{${\ttilde G}^{(k)}_{i,i}=G^{(k)}_{i,i}\bigl(G^{(0)}_{i,i+1}\bigr)^2
-G^{(k)}_{i,i+1}G^{(0)}_{i,i+1}$}\hfill{} \\
&-G^{(k)}_{i+1,i}G^{(0)}_{i,i+1}+G^{(k)}_{i+1,i+1},\\
  {\ttilde G}^{(k)}_{i,i+1}=G^{(k)}_{i,i}G^{(0)}_{i,i+1}-G^{(k)}_{i+1,i}, & \\
  {\ttilde G}^{(k)}_{i+1,i}=G^{(k)}_{i,i}G^{(0)}_{i,i+1}-G^{(k)}_{i,i+1}, & \\
  {\ttilde G}^{(k)}_{i+1,i+1}=G^{(k)}_{i,i}. &  \\
\end{array}%
\right.
\ee
The action of the last generator, $\beta_{n,1}$ is
\be\label{eq:DnBraid2}
\beta_{n,1}{\mathcal G}^{(k)}={\ttilde{\mathcal G}}^{(k)}:  \
\left\{
\begin{array}{ll}
  {\ttilde G}^{(k)}_{1,j}=G^{(k+1)}_{n,j}, & n>j>1,\\
  {\ttilde G}^{(k)}_{j,1}=G^{(k-1)}_{j,n}, & 1<j<n, \\
  {\ttilde G}^{(k)}_{n,j}=G^{(k)}_{n,j}G^{(1)}_{n,1}-G^{(k-1)}_{1,j}, & n>j>1, \\
  {\ttilde G}^{(k)}_{j,n}=G^{(k)}_{j,n}G^{(1)}_{n,1}-G^{(k+1)}_{j,1}, & 1<j<n, \\
\multispan{2}
\hbox{${\ttilde G}^{(k)}_{n,n}=G^{(k)}_{n,n}\bigl(G^{(1)}_{n,1}\bigr)^2-G^{(k+1)}_{n,1}G^{(1)}_{n,1}
  -G^{(k-1)}_{1,n}G^{(1)}_{n,1}+G^{(k)}_{1,1},$} \\
  {\ttilde G}^{(k)}_{n,1}=G^{(k-1)}_{n,n}G^{(1)}_{n,1}-G^{(k-2)}_{1,n}, & \\
  {\ttilde G}^{(k)}_{1,n}=G^{(k+1)}_{n,n}G^{(1)}_{n,1}-G^{(k+2)}_{n,1}, & \\
  {\ttilde G}^{(k)}_{1,1}=G^{(k)}_{n,n}, &  \\
\end{array}%
\right.
\ee
where we imply the relations (\ref{eq:Gijk1}) and (\ref{eq:Gijk2}), in particular, we have
$$
{\ttilde G}^{(1)}_{n,1}=G^{(0)}_{n,n}G^{(1)}_{n,1}-G^{(-1)}_{1,n}=G^{(1)}_{n,1}.
$$
\end{theorem}

\proof
We deduce the braid group relations for $G_{i,j}^{(k)}$ defined by
(\ref{eq:Gijk}) by the braid group action (\ref{eq:braidDM}) on
the monodromy matrices $M_1,\dots, M_n, M_h$. Setting:
$$
\ttilde G_{i,j}^{(k)}=\tr\left(\beta(M_i) \beta(M_h^k) \beta (M_j) \beta(M_h^{-k})
\right),
$$
one immediately obtains (\ref{eq:DnBraid1}). In the same way, defining
$$
\beta_{n,1}:=\beta_{n,h}\beta_{h,1}\beta_{n,h}^{-1}.
$$
we get (\ref{eq:DnBraid2}). \endproof

\begin{remark}
The braid group action (\ref{eq:DnBraid1}), (\ref{eq:DnBraid2}) does
not depend on the number $m$ of clashed poles, and it is therefore
well defined in the abstract case as well.
\end{remark}

Observe that the action of the braid-group elements
$\beta_{i,i+1}$ for $i=1,\dots,n-1$ is defined separately for each
of $n\times n$ matrices ${\mathcal G}^{(k)}$ (i.e. all relations
involve one and the same level $k$) and it has precisely the same
form for all of them,  while the action of the of the last generator
$\beta_{n,1}$ mixes different levels (labeled by the index
$(k)$).  The following proposition is straightforward to prove:

\begin{prop}\label{lem-braid}
The braid group transformations for ${\mathfrak D}_n$ algebra have the following matrix representation
in terms of the matrix ${\mathcal G}(\lambda)$ (\ref{G-def}):
\be\label{RA-infty}
\beta_{i,i+1}{\mathcal G}(\lambda)=B_{i,i+1}{\mathcal G}(\lambda)B^T_{i,i+1},\quad i=1,\dots,n-1
\ee
where the matrices $B_{i,i+1}$ depend only on $G_{ij}^{(0)}$ and
have the form (\ref{Bii+1}) (with  $G_{ij}$ replaced by
$G_{ij}^{(0)}$).
The action of $\beta_{n,1}$ is
\be
\beta_{n,1}{\mathcal G}(\lambda)={ B}_{n,1}(\lambda){\mathcal G}(\lambda)
\bigr({ B}_{n,1}(\lambda^{-1})\bigl)^T,
\label{R-n1-A-infty}
\ee
where
\be\label{eq:bn1}
B_{n,1}(\lambda)=\left(\begin{array}{ccccc}
0&0&\dots&0&\lambda\\
0&1&0&\dots&0\\
\vdots&0&\ddots&\ddots&\vdots\\
0&\vdots& \ddots& 1& 0\\
-{\lambda}^{-1}&0&\dots&0&G_{n,1}^{(1)}\\
\end{array}\right).
\ee
\end{prop}

\subsection{Quantum braid group relations}\label{subse:qbraid}

In this subsection, we assume ${G_{i,j}^{(k)}}^\hbar$ to be {\em Hermitian operators}, subject
to quantum exchange relations (\ref{eq:molev}). The action of the braid group then
follows from the one for the quantum $D_n$ algebra in \cite{Ch2}.

We define the quantum ${\mathcal G}^\hbar(\lambda)$ to be
\be
{\mathcal G}^\hbar_{i,j}(\lambda):= {\mathcal A}^\hbar_{i,j}+\sum_{k=1}^\infty {G_{i,j}^{(k)}}^\hbar\lambda^{-k},
\label{G-quan}
\ee
where ${\mathcal A}^\hbar$ is an upper-triangular matrix with the entries
${\mathcal A}^\hbar_{i,j}=\{{G^{(0)}}^\hbar_{i,j},i<j;\ q^{-1},i=j;\ 0,i>j\}$. Recall that $q^\dagger=q^{-1}$.

\begin{prop}\label{lem-braid-q}
The braid group transformations for the quantum ${\mathfrak D}_n$ algebra have the following matrix representation
in terms of the matrix ${\mathcal G}^\hbar(\lambda)$ (\ref{G-quan}):
\be
\label{RA}
\beta^\hbar_{i,i+1}{\mathcal G}^\hbar(\lambda)=B^\hbar_{i,i+1}{\mathcal G}^\hbar(\lambda)\bigl(B^\hbar_{i,i+1}\bigr)^\dagger,\quad i=1,\dots,n-1
\ee
where
\be
\label{Bii+1-q}
B^\hbar_{i,i+1}=\begin{array}{c}
            \vdots \\
            i \\
            i+1 \\
            \vdots \\
          \end{array}
          \left(
          \begin{array}{cccccccc}
            1 &  &  &  &  &  &  &  \\
             & \ddots &  &  &  &  &  &  \\
             &  & 1 &  &  &  &  &  \\
             &  &  &   q{G^{(0)}}^\hbar_{i,i+1} & -q^{2} &  & &  \\
            &  &  & 1 & 0 &  &  &  \\
             &  &  &  &  & 1 &  &  \\
             &  &  &  &  &  & \ddots &  \\
             &  &  &  &  &  &  & 1 \\
          \end{array}
          \right).
\ee
The action of $\beta^\hbar_{n,1}$ is
\be
\beta_{n,1}^\hbar{\mathcal G}^\hbar(\lambda)={B}^\hbar_{n,1}(\lambda){\mathcal G}^\hbar(\lambda)
\bigr({B}^\hbar_{n,1}(\lambda^{-1})\bigl)^\dagger,
\label{R-n1-A-q}
\ee
where
\be
\label{eq:bn1-q}
B^\hbar_{n,1}(\lambda)=\left(\begin{array}{ccccc}
0&0&\dots&0&\lambda\\
0&1&0&\dots&0\\
\vdots&0&\ddots&\ddots&\vdots\\
0&\vdots& \ddots&  1& 0\\
-q^2{\lambda}^{-1}&0&\dots&0&q{G_{n,1}^{(1)}}^\hbar\\
\end{array}\right).
\ee
\end{prop}

\section{Central elements of  the ${\mathfrak D}_n$ algebras and of its Poissonian reductions}\label{se:central}

Thanks to the matrix form of the full braid group action given
in Proposition \ref{lem-braid} above,  Molev and Ragoucy's result
that the central elements of the  ${\mathfrak D}_n$ algebra are given
by the coefficients of $\lambda$ in the polynomial
\be\label{eq:cent}
\det\left({\mathcal G}(\lambda)
\right)
\end{equation}
still holds true for the full braid group action.

We are now going to study two types of finite-dimensional
reductions,
the {\it level-$p$ reductions}\/ and the {\it
reductions to the $D_n$  case}\/ (see Sec.~\ref{sub:dn}),
the corresponding braid group actions and their
central elements. Before
describing these two types of reductions let us recall how to
produce the central elements in the $A_n$ case.

\subsection{Central elements of $A_n$}
The general central elements of the Poisson algebra are simultaneously the braid-group
invariants (which translates into mapping-class-group invariant terms in the geometrical setting).
This can be easily seen from that the relation (\ref{BAB}) holds for the transposed matrix
${\mathcal A}$ as well,
\be
\beta_{i,i+1}{\mathcal A}^T=B_{i,i+1}{\mathcal A}^TB^{T}_{i,i+1},
\label{BABT}
\ee
and so for any linear combination $\lambda{\mathcal A} +
\lambda^{-1}{\mathcal A}^T$, whose determinant is therefore
braid-group invariant object \cite{Bondal}. Because $\det {\mathcal A}=1$, the
generating function for Poisson central elements is
\be
\label{eq:c-e}
\det ({\mathcal A}^{-T}{\mathcal A}-\varphi)
=\det(\lambda{\mathcal A} + \lambda^{-1}{\mathcal A}^T)\lambda^{-n},\quad \varphi=-\lambda^2,
\ee
which coincides with the generating function for the invariants of the braid group.
Therefore when considering the Nelson--Regge algebra (\ref{eq:NR}) as an abstract algebra for
$\frac{n(n-1)}{2}$ elements, the total number of possibly independent central elements is
$\left[\frac n2\right]$.

In the geometric case, the Poisson commuting elements (the Casimir elements)
are traces of monodromies. Whereas $\tr M_i^k=\Tr(F^k)$ is a
constant, the monodromy at infinity is nontrivial and its trace is
equal to
$$
\tr M_\infty=(-1)^{n-1} \cosh(P),
$$
where $P$ is the perimeter of the hole defined in (\ref{eq:per}). We
can prove that in this case the only non trivial braid group
invariants  generated by (\ref{eq:c-e}) is precisely $P$
\cite{ChMa}.

\subsection{Level $p$-reductions}

We obtain the level $p$ reduction if we set
\be
\label{level-p}
{M}^p_{h}=\ID
\ee
for some integer $p$. In the $2\times 2$ monodromy case, $\tr M_h=2\cos(2\pi k/p)=e^{P_h/2}+e^{-P_h/2}$,
where $k$ is an integer and $P_h=4\pi i k/p$ is the complex-valued perimeter of the second hole.
The condition (\ref{level-p}) is Poissonian because substituting it into relations
(\ref{MM-i-tilde}) and (\ref{MM-tilde}) for $k=p$ we obtain identities: the left-hand sides of
these relations vanish.

In the geometric setting, this condition means that, instead of a
new hole, we introduce a new orbifold point of order $p$. On the
level of elements of the ${\mathfrak D}_n$ algebra, it means that
$G^{(k+p)}_{i,j}=G^{(k)}_{i,j}$ or, since $G^{(-k)}_{j,i}=G^{(k)}_{i,j}$, we obtain
\be
\label{level-p-G}
G^{(k)}_{i,j}=G^{(p-k)}_{j,i}\ \hbox{for}\ k=0,\dots,p-1.
\ee
Due to this reduction, the generating function simplifies to
$\mathcal G(\lambda) =\frac{1}{\lambda^p-1}  \mathcal G_p(\lambda)$ where
\be\label{Gp}
\mathcal G_p(\lambda):=\mathcal A^{(0)} + \frac{\mathcal G^{(1)}}{\lambda}+\dots+
 \frac{\mathcal G^{(p-1)}}{\lambda^{p-1}}+\frac{{\mathcal A^{(0)}}^T}{\lambda^p},
 \ee
so the algebra becomes finite and we reserve for it the notation ${\mathfrak D}^{(p)}_n$.
\subsubsection{Braid-group relations and central elements}

We now present the action of the braid group for the level $p$-reductions.

\begin{prop}\label{th-braid}
The braid group relations (\ref{RA-infty}) and (\ref{R-n1-A-infty}) with the
reduction condition (\ref{level-p-G}) imposed retain their matrix forms provided
we replace ${\mathcal G}(\lambda)$ by ${\mathcal G}_p(\lambda)$ given in (\ref{Gp}).
There are exactly  $\left[\dfrac{np}{2}\right]$ algebraically independent central elements in the
algebra ${\mathfrak D}^{(p)}_n$. They are generated by $\det {\mathcal G}_p(\lambda)$.
\end{prop}

The proof of this proposition is very technical and we set it in Appendix~C.

\begin{remark}
Observe that for $p=1$, $\det\left({\mathcal G}_1(\lambda)\right)$
becomes the generating function of the braid group invariants for
the $A_n$ algebra given in (\ref{eq:c-e}).
\end{remark}

We can now compute the dimension of the Poisson leaves corresponding to the  algebra
${\mathfrak D}^{(p)}_n$. From condition (\ref{level-p-G}) we have that
the number of generators $G^{(k)}_{i,j}$ of the algebra ${\mathfrak D}^{(p)}_n$
is $n^2p/2$ for even $p$ and $n(np-1)/2$ for odd $p$. Having
$\left[\frac{np}{2}\right]$ generally algebraically independent central elements,
we find that the highest dimensional symplectic leaves of the algebra ${\mathfrak D}^{(p)}_n$ have dimension
always even:
$$
\hbox{Poisson leaf dim}=
\left\{\begin{array}{l}
\frac{n^2p}2-\frac{np}2=\frac{n(n-1)p}2,\ \hbox{for $p$ even, any $n$,}\\
\\
\frac{n(n-1)p}2-\frac n2,\ \hbox{ for $p$ odd, $n$ even,}\\
\\
\frac{n(n-1)p}2-\frac {n-1}2,\ \hbox{ for $p$ odd, $n$ odd.}\\
\end{array}
\right.
$$
As a consequence, the geometric case corresponds to highly degenerated symplectic leaves, as their dimension is $2(n-2)$.
It is an interesting problem to provide the complete classification of the dimensions of
symplectic leaves of ${\mathfrak D}^{(p)}_n$ algebras in the spirit of such the classification
for the $A_n$ algebras constructed by Bondal \cite{Bondal}; we leave it for future studies.

\subsection{Reduction of ${\mathfrak D}_n$ to the $D_n$ algebra}

\subsubsection{Basic relations of the reduction ${\mathfrak D}_n\to D_n$}

We begin with that we naturally identify those elements of the algebra ${\mathfrak D}_n$
that correspond to geodesics without self-intersections with the corresponding elements of the
$D_n$ algebra (see Figure \ref{fi:Dn}):
$$
G^{(0)}_{i,j}\to {\hhat G}_{i,j},\quad 1\le i<j\le n
$$
and
$$
G^{(1)}_{i,j}\to {\hhat G}_{i,j},\quad 1\le j<i\le n.
$$

We now use the skein relations to present elements $G^{(1)}_{i,j}$ with $i\le j$:
\bea
&{}&G^{(1)}_{i,i}\to {\hhat G}_{i,i}^2+\Pi^2-2,\quad 1\le i\le n,
\label{red-ii}
\\
&{}&G^{(1)}_{i,j}\to 2{\hhat G}_{i,i}{\hhat G}_{j,j}-{\hhat G}_{j,i}
+(\Pi^2-2){\hhat G}_{i,j},\quad 1\le i<j\le n,
\label{red-ij}
\eea
or, in the graphical form,
$$
{\psset{unit=0.35}
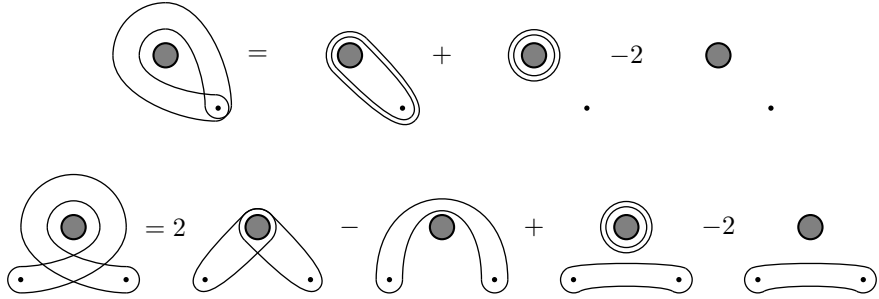
\begin{pspicture}(-7,-3)(7,2)
\newcommand{\MOLD}{%
\rput(2,-2){\pscircle*{0.1}}
\pscircle[fillstyle=solid, fillcolor=gray](0,0){0.5}
}
\newcommand{\GLOOPCLOSED}{%
\psarc[linewidth=0.5pt](2,-1.9){.5}{180}{360}
\psarc[linewidth=0.5pt](1.9,-2){.5}{-90}{90}
\psbezier[linewidth=0.5pt](1.9,-1.5)(1,-1.5)(-1,-1)(-1,0)
\psbezier[linewidth=0.5pt](1.9,-2.5)(0.5,-2.5)(-2,-1.5)(-2,0)
\psbezier[linewidth=0.5pt](1.5,-1.9)(1.5,-1)(1,1)(0,1)
\psbezier[linewidth=0.5pt](2.5,-1.9)(2.5,-0.5)(1.5,2)(0,2)
\psarc[linewidth=0.5pt](0,0){1}{90}{180}
\psarc[linewidth=0.5pt](0,0){2}{90}{180}
}
\newcommand{\GEAR}{%
\psarc[linewidth=0.5pt](0,0){.7}{90}{270}
\psbezier[linewidth=0.5pt](0,0.7)(2.5,0.7)(3.3,0.5)(3.3,0)
\psbezier[linewidth=0.5pt](0,-0.7)(2.5,-0.7)(3.3,-0.5)(3.3,0)
}
\newcommand{\GEARTHICK}{%
\psarc[linewidth=0.5pt](0,0){.9}{90}{270}
\psbezier[linewidth=0.5pt](0,0.9)(2.7,0.9)(3.5,0.6)(3.5,0)
\psbezier[linewidth=0.5pt](0,-0.9)(2.7,-0.9)(3.5,-0.6)(3.5,0)
}
\rput(-10.5,0){\MOLD}
\rput(-10.5,0){\GLOOPCLOSED}
\rput(-7,0){\makebox(0,0){$=$}}
\rput(-3.5,0){\MOLD}
\rput{315}(-3.5,0){\GEAR}
\rput{315}(-3.5,0){\GEARTHICK}
\rput(0,0){\makebox(0,0){$+$}}
\rput(3.5,0){\MOLD}
\pscircle[linewidth=0.5pt](3.5,0){0.8}
\pscircle[linewidth=0.5pt](3.5,0){1}
\rput(7,0){\makebox(0,0){$-2$}}
\rput(10.5,0){\MOLD}
\end{pspicture}
}
$$
$$
{\psset{unit=0.35}
\begin{pspicture}(-7,-3)(7,2)
\newcommand{\MOLD}{%
\rput(-2,-2){\pscircle*{0.1}}
\rput(2,-2){\pscircle*{0.1}}
\pscircle[fillstyle=solid, fillcolor=gray](0,0){0.5}
}
\newcommand{\GSHORT}{%
\psarc[linewidth=0.5pt](-2,-2){.5}{120}{300}
\psarc[linewidth=0.5pt](2,-2){.5}{-120}{60}
\psbezier[linewidth=0.5pt](-2.25,-1.567)(-1.817,-1.317)(1.817,-1.317)(2.25,-1.567)
\psbezier[linewidth=0.5pt](-1.75,-2.433)(-1.317,-2.183)(1.317,-2.183)(1.75,-2.433)
}
\newcommand{\GMED}{%
\psarc[linewidth=0.5pt](-2,-2){.5}{90}{270}
\psarc[linewidth=0.5pt](2,2){.5}{0}{180}
\psbezier[linewidth=0.5pt](-2,-1.5)(.3,-1.5)(1.5,-.3)(1.5,2)
\psbezier[linewidth=0.5pt](-2,-2.5)(1,-2.5)(2.5,-1)(2.5,2)
}
\newcommand{\GLONG}{%
\psarc[linewidth=0.5pt](-2,-2){.5}{90}{270}
\psarc[linewidth=0.5pt](-2,2){.5}{90}{270}
\psbezier[linewidth=0.5pt](-2,-1.5)(1.5,-1.5)(1.5,1.5)(-2,1.5)
\psbezier[linewidth=0.5pt](-2,-2.5)(2.1,-2.5)(2.1,2.5)(-2,2.5)
}
\newcommand{\GLOOP}{%
\psarc[linewidth=0.5pt](-2,-2){.5}{90}{270}
\psarc[linewidth=0.5pt](2,-2){.5}{-90}{90}
\psbezier[linewidth=0.5pt](-2,-1.5)(-1,-1.5)(1,-1)(1,0)
\psbezier[linewidth=0.5pt](-2,-2.5)(-0.5,-2.5)(2,-1.5)(2,0)
\psbezier[linewidth=0.5pt](2,-1.5)(1,-1.5)(-1,-1)(-1,0)
\psbezier[linewidth=0.5pt](2,-2.5)(0.5,-2.5)(-2,-1.5)(-2,0)
\psarc[linewidth=0.5pt](0,0){1}{0}{180}
\psarc[linewidth=0.5pt](0,0){2}{0}{180}
}
\newcommand{\GEAR}{%
\psarc[linewidth=0.5pt](0,0){.7}{90}{270}
\psbezier[linewidth=0.5pt](0,0.7)(2.5,0.7)(3.3,0.5)(3.3,0)
\psbezier[linewidth=0.5pt](0,-0.7)(2.5,-0.7)(3.3,-0.5)(3.3,0)
}
\newcommand{\GEARTHICK}{%
\psarc[linewidth=0.5pt](0,0){.9}{90}{270}
\psbezier[linewidth=0.5pt](0,0.9)(2.7,0.9)(3.5,0.6)(3.5,0)
\psbezier[linewidth=0.5pt](0,-0.9)(2.7,-0.9)(3.5,-0.6)(3.5,0)
}
\rput(-14,0){\MOLD}
\rput(-14,0){\GLOOP}
\rput(-10.5,0){\makebox(0,0){$=2$}}
\rput(-7,0){\MOLD}
\rput{225}(-7,0){\GEAR}
\rput{315}(-7,0){\GEAR}
\rput(-3.5,0){\makebox(0,0){$-$}}
\rput(0,0){\MOLD}
\rput{90}(0,0){\GLONG}
\rput(3.5,0){\makebox(0,0){$+$}}
\rput(7,0){\MOLD}
\rput(7,0){\GSHORT}
\pscircle[linewidth=0.5pt](7,0){0.8}
\pscircle[linewidth=0.5pt](7,0){1}
\rput(10.5,0){\makebox(0,0){$-2$}}
\rput(14,0){\MOLD}
\rput(14,0){\GSHORT}
\end{pspicture}
}
$$
(To obtain these relations we also used that when resolving the skein
relations, the empty loop is equal $-2$.) Note, first, the
appearance of the parameter
$$
\Pi:=e^{P_h/2}+e^{-P_h/2},
$$
which is the geodesic function for the second hole with the perimeter $P_h$,
and, second, that we can rewrite the relation (\ref{red-ij}) as\footnote{Note that
all the elements ${\hhat G}_{i,j}$ with $i\ne j$ can be expressed as polynomial
expressions of the elements $G_{a,b}$ of the $A_{n+m}$-algebra; this property is
however lacking for the diagonal elements ${\hhat G}_{i,i}$, which cannot be presented as
polynomial functions of the generators of the ${A}_{n+m}$-algebra.}
\be
\label{resolution}
G^{(1)}_{i,j}\to 2{\hhat G}_{i,i}{\hhat G}_{j,j}-{G}^{(1)}_{j,i}+(\Pi^2-2){G}^{(0)}_{i,j},
\quad 1\le i<j\le n,
\ee
or, recalling that ${\mathcal G}^{(k)}=\bigl({\mathcal G}^{(-k)}\bigr)^T$,
$$
G^{(1)}_{i,j}\to 2{\hhat G}_{i,i}{\hhat G}_{j,j}-{G}^{(-1)}_{i,j}+(\Pi^2-2){G}^{(0)}_{i,j},
\quad 1\le i<j\le n,\ \Pi\ne 0.
$$

It is especially useful to express this reduction in terms of the matrices $\widehat{\mathcal A}$,
${\widehat{\mathcal R}}$, and ${\widehat{\mathcal S}}$ defined in (\ref{hatA-matrix}), (\ref{R.cl}) and (\ref{S.cl}) respectively:
$$
{\mathcal G}^{(1)}\to {\widehat{\mathcal R}}+{\widehat{\mathcal S}}+(\Pi^2-1)\widehat{\mathcal A}-\widehat{\mathcal A}^T.
$$

We now continue expressing higher ${\mathcal G}^{(k)}$ fixing the
parameter $i$ and moving $j$ counterclockwise as shown below (in the
l.h.s. moving $j$ counterclockwise corresponds to constructing the
geodesic which winds around the hole twice):
$$
{\psset{unit=0.4}
\begin{pspicture}(-7,-3)(7,2)
\newcommand{\MOLD}{%
\rput(-2,-2){\pscircle*{0.1}}
\rput(2,-2){\pscircle*{0.1}}
\rput(-2.8,-2){\makebox(0,0){$i$}}
\rput(2.8,-2){\makebox(0,0){$j$}}
\psarc[linestyle=dashed, linewidth=1.5pt]{->}(0.6,-0.6){2}{-35}{305}
\pscircle[fillstyle=solid, fillcolor=gray](0,0){0.5}
}
\newcommand{\GSHORT}{%
\psarc[linewidth=0.5pt](-2,-2){.5}{120}{300}
\psarc[linewidth=0.5pt](2,-2){.5}{-120}{60}
\psbezier[linewidth=0.5pt](-2.25,-1.567)(-1.817,-1.317)(1.817,-1.317)(2.25,-1.567)
\psbezier[linewidth=0.5pt](-1.75,-2.433)(-1.317,-2.183)(1.317,-2.183)(1.75,-2.433)
}
\newcommand{\GMED}{%
\psarc[linewidth=0.5pt](-2,-2){.5}{90}{270}
\psarc[linewidth=0.5pt](2,2){.5}{0}{180}
\psbezier[linewidth=0.5pt](-2,-1.5)(.3,-1.5)(1.5,-.3)(1.5,2)
\psbezier[linewidth=0.5pt](-2,-2.5)(1,-2.5)(2.5,-1)(2.5,2)
}
\newcommand{\GLONG}{%
\psarc[linewidth=0.5pt](-2,-2){.5}{90}{270}
\psarc[linewidth=0.5pt](-2,2){.5}{90}{270}
\psbezier[linewidth=0.5pt](-2,-1.5)(1.5,-1.5)(1.5,1.5)(-2,1.5)
\psbezier[linewidth=0.5pt](-2,-2.5)(2.1,-2.5)(2.1,2.5)(-2,2.5)
}
\newcommand{\GLOOP}{%
\psarc[linewidth=0.5pt](-2,-2){.5}{90}{270}
\psarc[linewidth=0.5pt](2,-2){.5}{-90}{90}
\psbezier[linewidth=0.5pt](-2,-1.5)(-1,-1.5)(1,-1)(1,0)
\psbezier[linewidth=0.5pt](-2,-2.5)(-0.5,-2.5)(2,-1.5)(2,0)
\psbezier[linewidth=0.5pt](2,-1.5)(1,-1.5)(-1,-1)(-1,0)
\psbezier[linewidth=0.5pt](2,-2.5)(0.5,-2.5)(-2,-1.5)(-2,0)
\psarc[linewidth=0.5pt](0,0){1}{0}{180}
\psarc[linewidth=0.5pt](0,0){2}{0}{180}
}
\newcommand{\GEAR}{%
\psarc[linewidth=0.5pt](0,0){.7}{90}{270}
\psbezier[linewidth=0.5pt](0,0.7)(2.5,0.7)(3.3,0.5)(3.3,0)
\psbezier[linewidth=0.5pt](0,-0.7)(2.5,-0.7)(3.3,-0.5)(3.3,0)
}
\newcommand{\GEARTHICK}{%
\psarc[linewidth=0.5pt](0,0){.9}{90}{270}
\psbezier[linewidth=0.5pt](0,0.9)(2.7,0.9)(3.5,0.6)(3.5,0)
\psbezier[linewidth=0.5pt](0,-0.9)(2.7,-0.9)(3.5,-0.6)(3.5,0)
}
\rput(-13,0){\MOLD}
\rput(-13,0){\GLOOP}
\rput(-9,0){\makebox(0,0){$=2$}}
\rput(-5,0){\MOLD}
\rput{225}(-5,0){\GEAR}
\rput{315}(-5,0){\GEAR}
\rput(-1,0){\makebox(0,0){$-$}}
\rput(3,0){\MOLD}
\rput{90}(3,0){\GLONG}
\rput(8,0){\makebox(0,0){$+(\Pi^2-2)$}}
\rput(12,0){\MOLD}
\rput(12,0){\GSHORT}
\end{pspicture}
}
$$
$$
{\psset{unit=0.4}
\begin{pspicture}(-7,-3)(7,2)
\newcommand{\MOLD}{%
\rput(-2,-2){\pscircle*{0.1}}
\rput(2,-2){\pscircle*{0.1}}
\rput(-2.8,-2){\makebox(0,0){$i$}}
\rput(2.8,-2){\makebox(0,0){$j$}}
\pscircle[fillstyle=solid, fillcolor=gray](0,0){0.5}
}
\newcommand{\GSHORT}{%
\psarc[linewidth=0.5pt](-2,-2){.5}{120}{300}
\psarc[linewidth=0.5pt](2,-2){.5}{-120}{60}
\psbezier[linewidth=0.5pt](-2.25,-1.567)(-1.817,-1.317)(1.817,-1.317)(2.25,-1.567)
\psbezier[linewidth=0.5pt](-1.75,-2.433)(-1.317,-2.183)(1.317,-2.183)(1.75,-2.433)
}
\newcommand{\GMED}{%
\psarc[linewidth=0.5pt](-2,-2){.5}{90}{270}
\psarc[linewidth=0.5pt](2,2){.5}{0}{180}
\psbezier[linewidth=0.5pt](-2,-1.5)(.3,-1.5)(1.5,-.3)(1.5,2)
\psbezier[linewidth=0.5pt](-2,-2.5)(1,-2.5)(2.5,-1)(2.5,2)
}
\newcommand{\GLONG}{%
\psarc[linewidth=0.5pt](-2,-2){.5}{90}{270}
\psarc[linewidth=0.5pt](-2,2){.5}{90}{270}
\psbezier[linewidth=0.5pt](-2,-1.5)(1.5,-1.5)(1.5,1.5)(-2,1.5)
\psbezier[linewidth=0.5pt](-2,-2.5)(2.1,-2.5)(2.1,2.5)(-2,2.5)
}
\newcommand{\GLOOP}{%
\psarc[linewidth=0.5pt](-2,-2){.5}{90}{270}
\psarc[linewidth=0.5pt](2,-2){.5}{-90}{90}
\psbezier[linewidth=0.5pt](-2,-1.5)(-1,-1.5)(1,-1)(1,0)
\psbezier[linewidth=0.5pt](-2,-2.5)(-0.5,-2.5)(2,-1.5)(2,0)
\psbezier[linewidth=0.5pt](2,-1.5)(1,-1.5)(-1,-1)(-1,0)
\psbezier[linewidth=0.5pt](2,-2.5)(0.5,-2.5)(-2,-1.5)(-2,0)
\psarc[linewidth=0.5pt](0,0){1}{0}{180}
\psarc[linewidth=0.5pt](0,0){2}{0}{180}
}
\newcommand{\GEAR}{%
\psarc[linewidth=0.5pt](0,0){.7}{90}{270}
\psbezier[linewidth=0.5pt](0,0.7)(2.5,0.7)(3.3,0.5)(3.3,0)
\psbezier[linewidth=0.5pt](0,-0.7)(2.5,-0.7)(3.3,-0.5)(3.3,0)
}
\newcommand{\GEARTHICK}{%
\psarc[linewidth=0.5pt](0,0){.9}{90}{270}
\psbezier[linewidth=0.5pt](0,0.9)(2.7,0.9)(3.5,0.6)(3.5,0)
\psbezier[linewidth=0.5pt](0,-0.9)(2.7,-0.9)(3.5,-0.6)(3.5,0)
}
\rput(-9,0){\makebox(0,0){$=2$}}
\rput(-5,0){\MOLD}
\rput{225}(-5,0){\GEAR}
\rput{315}(-5,0){\GEAR}
\rput(-1,0){\makebox(0,0){$-$}}
\rput(3,0){\MOLD}
\rput(3,0){\GSHORT}
\rput(7.5,0){\makebox(0,0){$+(\Pi^2-2)$}}
\rput(12,0){\MOLD}
\rput(12,0){\GLOOP}
\end{pspicture}
}
$$

After completing this counterclockwise rotation of $j$, we obtain in the left-hand side the element $G^{(2)}_{i,j}$
whereas in the right-hand side the product ${\hhat G}_{i,i}{\hhat G}_{j,j}$
is cyclically symmetric, the term $-{G}^{(-1)}_{i,j}$
becomes $-{G}^{(0)}_{i,j}=-{\hhat G}_{i,j}$, and the last term $(\Pi^2-2){G}^{(0)}_{i,j}$
becomes $(\Pi^2-2){G}^{(1)}_{i,j}$, and we again apply the reduction (\ref{resolution}) and obtain
that the matrix ${\mathcal G}^{(2)}$ can be in turn presented as a linear combination of the
matrices ${\widehat{\mathcal R}}$, ${\widehat{\mathcal S}}$, ${\widehat{\mathcal A}}$, and $\widehat{\mathcal A}^T$. It is not difficult to
solve the obtained recurrent relations and obtain the following reduction law for ${\mathcal G}^{(k)}$:
\bea
&{}&{\mathcal G}^{(k)}\to \frac{e^{kP_h}-e^{-kP_h}}{e^{P_h}-e^{-P_h}}{\widehat{\mathcal R}}+
\frac{e^{kP_h}-2+e^{-kP_h}}{(e^{P_h}-1)(1-e^{-P_h})}{\widehat{\mathcal S}}
\nonumber
\\
&{}&\qquad +\left[\frac{e^{kP_h}}{1-e^{-P_h}}-\frac{e^{-kP_h}}{e^{P_h}-1}\right]{\widehat{\mathcal A}}
-\left[\frac{e^{(k-1)P_h}}{1-e^{-P_h}}-\frac{e^{-(k-1)P_h}}{e^{P_h}-1}\right]\widehat{\mathcal A}^T,\ k\ge1.
\label{recursion}
\eea
The corresponding law for ${\mathcal G}^{(-k)}$ can be obtained by transposing these relations. Now, the
main statement follows.

\begin{theorem}
\label{th-Dn}
The expressions (\ref{RA-infty}) and
(\ref{R-n1-A-infty}) are faithful representations for the braid-group action (\ref{Dn-braid-cl}),
(\ref{Dn-braid-cl-n1}) on elements of the $D_n$ algebra for {\bf any non--zero} $\Pi$ provided
the matrices ${\mathcal G}^{(k)}$ are expressed through the elements ${\hhat G}_{i,j}$ and the
parameter $\Pi\ne0$ using formulas (\ref{recursion}).
\end{theorem}

\proof We can verify directly that if we substitute the reduction formula
(\ref{recursion}) for every block ${\mathcal G}^{(k)}$
in the matrix representation (\ref{G-def}) and perform the braid-group
transformations in the matrix form (\ref{RA-infty}) and (\ref{R-n1-A-infty}), then,
in each matrix entry, the transformed quantities ${\ttilde {\hhat G}}_{i,j}$ will
satisfy relations (\ref{Dn-braid-cl}) and (\ref{Dn-braid-cl-n1}).\endproof

Note that one of the most important features of the braid group relations
for both the ${\mathfrak D}_n$ algebra and the $D_n$ algebra is that neither of them
depends on the hole perimeters $P$ and $P_h$. So, since these relations both describe the
same mapping class group transformations in the geometrical case, it is natural to
expect that they will also coincide if we apply the reduction procedure based on
the skein relation. In fact this is the case as shown by the following:

\begin{cor}
\label{cor-reduction}
A $p$-level reduction is {\em consistent with the $D_n$-reduction} provided $\bigl(e^{P_h}\bigr)^p=1$. In particular,
the braid-group representation for ${\mathcal G}_p(\lambda)$ (\ref{recursion}) generates then the
braid-group relations (\ref{Dn-braid-cl}) and (\ref{Dn-braid-cl-n1}) of the algebra $D_n$.
\end{cor}

\begin{remark}
The case $\Pi=0$ corresponding to the reduction of level 2 would correspond to the $D_n$ reduction
$G^{(1)}_{i,j}={\hhat G}_{j,i}=G^{(1)}_{j,i}$. In this case, however, the algebraic elements
${\hhat G}_{i,j}$ become dependent (${\hhat G}_{i,i}{\hhat G}_{j,j}={\hhat G}_{i,j}+{\hhat G}_{j,i}$)
and we are lacking the complete $D_n$ algebra.
\end{remark}

\subsubsection{Central elements of the $D_n$ algebra}\label{subse:DDn}

We now substitute reduction formulas (\ref{recursion}) into the
representation (\ref{G-def}) and perform the explicit summation over
powers of $\lambda^{-1}$. The results reads
\bea
{\mathcal G}(\lambda)&\to&\frac{\lambda}{(\lambda-1)(e^{-P_h}u-1)(e^{P_h}u-1)}\times
\nonumber
\\
&{}&\times \bigl[(\lambda-1){\widehat{\mathcal R}}+(\lambda+1){\widehat{\mathcal S}}+(\lambda^2-1){\widehat{\mathcal A}}
-(\lambda-\lambda^{-1})\widehat{\mathcal A}^T\bigr],
\label{sum-Dn}
\eea
so we come to the following proposition.
\begin{prop}
The $D_n$ algebra admits exactly $n$ algebraically independent central elements $c_1,\dots,c_n$. They are generated by
\bea
&&\det\bigl[(\lambda-1){\widehat{\mathcal R}}+(\lambda+1){\widehat{\mathcal S}}+(\lambda^2-1)\widehat{\mathcal A}
-(\lambda-\lambda^{-1})\widehat{\mathcal A}^T\bigr]=
\\
&&=(\lambda-1)^{n-1}\bigl[\lambda^{n+1}+\sum_{i=1}^n \lambda^i c_i+(-1)^{n+1}
\sum_{i=1}^n \lambda^{1-i}c_i+(-1)^{n+1}\lambda^{-n}\bigr].\nn
\label{central-D}
\eea
\end{prop}

\proof The fact that the central elements are generated by
$\det(\mathcal G(\lambda))$ follows from Theorem~\ref{th-Dn} and the
formula for the central elements of ${\mathfrak D}_n$. The fact that
this determinant takes the form given by the second row of
(\ref{central-D}) follows from the substitution (\ref{recursion})
and the fact that the matrix ${\widehat{\mathcal S}}$ has rank one,
so no more than one element of this matrix can enter the products
when expanding the determinant over products of entries, and all
other entries are proportional to $(\lambda-1)$. To prove the
algebraic independence of $c_1,\dots,c_n$, let us consider the
particular case where ${\hhat G}_{i,j}=0$ for $i\ne j$ and ${\hhat G}_{i,i}\ne0$
and ${\hhat G}^2_{i,i}\ne {\hhat G}^2_{j,j}$ for $i\ne j$. In this case, ${\mathcal G}(\lambda)$ becomes
$$
(\lambda+\lambda^{-1})(\lambda-1){\mathbb E}+\hbox{diag}_{{\hhat G}_{i,i}}
\left(
  \begin{array}{cccc}
    1+\lambda & 2\lambda & \dots & 2\lambda \\
    2 & 1+\lambda & \ddots & \vdots \\
    \vdots & \ddots & \ddots & 2\lambda \\
    2 & \dots & 2 & 1+\lambda \\
  \end{array}
\right)\hbox{diag}_{{\hhat G}_{i,i}}
$$
(here $\hbox{diag}_{{\hhat G}_{i,i}}$ is the diagonal matrix with the entries ${\hhat G}_{i,i}$),
and evaluating the determinant by the minors of the second matrix, we obtain that
\bea
\det{\mathcal G}(\lambda)&=& \sum_{k=0}^n (\lambda-1)^n(\lambda+\lambda^{-1})^{n-k}
\left[\frac{\lambda+1}{\lambda-1}\right]^{\frac{1-(-1)^n}{2}}\times
\nonumber
\\
&{}&\times \hbox{SYM\,}{}_k\{{\hhat G}^2_{1,1},\dots,{\hhat G}^2_{n,n}\},
\nonumber
\eea
where SYM$_k$ is the symmetrical function of order $k$ of $n$ pairwise distinct variables ${\hhat G}^2_{i,i}$, and
these functions are obviously algebraically independent for $k=1,\dots,n$.

It remains to prove that in $D_n$ there are {\em no more} than $n$ central elements.
For this let us consider the $D_n$ Poisson structure for ${\hhat G}_{i,j}$
treating the elements ${\hhat G}_{i,j}$ with $i,j=1,\dots,n$ as coordinates
of a linear space ${\mathbb C}^{n^2}$. The Poisson brackets then define locally a
structure of a bi-vector field, and if this structure has degeneracy at most $n$
in a vicinity of just one point, then there are no more than $n$ central elements
in the (global) Poisson algebra of $D_n$.

A convenient choice of such a point is again ${\hhat G}_{i,j}=0$ for $i\ne j$ and
${\hhat G}_{i,i}\ne0$ and ${\hhat G}^2_{i,i}\ne {\hhat G}^2_{j,j}$ for $i\ne j$.
Then, the Poisson brackets from \cite{Ch1},~\cite{Ch2} take the following form in the
{\em vicinity} of this point in the configuration space ${\mathbb C}^{n^2}$:
\bea
\{{\hhat G}_{i,j},{\hhat G}_{i,i}\}&=&2{\hhat G}_{j,j} +O({\hhat G}_{\alpha,\beta}),\quad \alpha\ne\beta,
\nn\\
\{{\hhat G}_{i,j},{\hhat G}_{j,j}\}&=&-2{\hhat G}_{i,i} +O({\hhat G}_{\alpha,\beta}),\quad \alpha\ne\beta,
\label{eq:vicinity}\\
\{{\hhat G}_{i,j},{\hhat G}_{j,i}\}&=&2{\hhat G}^2_{j,j}-2{\hhat G}^2_{i,i}\nn
\eea
and all other brackets are either zero or of order $O({\hhat G}_{\alpha,\beta})$ with $\alpha\ne\beta$
i.e., they are at least of the linear order in the variables that are small in the vicinity
of the given point.

The brackets (\ref{eq:vicinity}) for the variables ${\hhat G}_{i,j}$ with $i\ne j$ are obviously
non-degenerate (because these variables just come in $n(n+1)/2$
pairs $({\hhat G}_{i,j},{\hhat G}_{j,i})$ and
commute in the given approximation order $O(1)$ with such the variables from all other pairs). So,
in fact, the Poisson leaf has the dimension at least $n(n-1)=n^2-n$ in the vicinity of the given point, and
we therefore have no more than $n$ central elements of the $D_n$ algebra.
\endproof

\begin{remark}
In fact, if we disregard all terms of order $O({\hhat G}_{\alpha,\beta})$ with $\alpha\ne\beta$
in the brackets (\ref{eq:vicinity}), then the Poisson dimension of the obtained system is
exactly $n(n-1)$; this is a simple but nice exercise in linear algebra which we leave to the
reader. So, the highest Poisson leaf dimension of the $D_n$ algebra is $n^2-n=n(n-1)$.
\end{remark}

\subsubsection{Central elements for $D_2$ and $D_3$}
In the $D_2$ algebra,
we have the following two central elements:
\bea
C_1^{(2)}&=&{\widehat G}_{1,1}{\widehat G}_{2,2}-{\widehat G}_{1,2}-{\widehat G}_{2,1},
\nn\\
C_2^{(2)}&=&{\widehat G}_{1,2}{\widehat G}_{2,1}-{\widehat G}^2_{1,1}-{\widehat G}^2_{2,2}.
\nn
\eea

In the $D_3$ algebra case there are three central elements:
\be
C_1^{(3)}
\!\!
=
\!\!
{\widehat G}_{1,1}{\widehat G}_{2,2}{\widehat G}_{3,3}-{\widehat G}_{1,1}({\widehat G}_{3,2}+{\widehat G}_{2,3})
-{\widehat G}_{2,2}({\widehat G}_{1,3}+{\widehat G}_{3,1})
-{\widehat G}_{3,3}({\widehat G}_{2,1}+{\widehat G}_{1,2}),
\nn
\ee
\be
C_2^{(3)}=
{\widehat G}_{1,2}{\widehat G}_{2,3}{\widehat G}_{3,1}-{\widehat G}_{1,2}{\widehat G}_{2,1}
-{\widehat G}_{2,3}{\widehat G}_{3,2}-{\widehat G}_{3,1}{\widehat G}_{1,3}
+{\widehat G}^2_{1,1}+{\widehat G}^2_{2,2}+{\widehat G}^2_{3,3},
\nn
\ee
\bea
C_3^{(3)}
\!\!\!\!\!
&=&
\!\!\!\!\!
{\widehat G}_{1,3}{\widehat G}_{2,1}{\widehat G}_{3,2}-{\widehat G}_{1,2}{\widehat G}_{2,1}{\widehat G}^2_{3,3}
-{\widehat G}_{2,3}{\widehat G}_{3,2}{\widehat G}^2_{1,1}
-{\widehat G}_{3,1}{\widehat G}_{1,3}{\widehat G}^2_{2,2}
\nn\\
\!\!\!\!\!
&{}&
\!\!\!\!\!
+2{\widehat G}_{1,1}{\widehat G}_{2,2}({\widehat G}_{2,3}{\widehat G}_{3,1}-{\widehat G}_{2,1}-{\widehat G}_{1,2})
+2{\widehat G}_{2,2}{\widehat G}_{3,3}({\widehat G}_{3,1}{\widehat G}_{1,2}-{\widehat G}_{3,2}-{\widehat G}_{2,3})
\nn\\
\!\!\!\!\!
&{}&
\!\!\!\!\!
+2{\widehat G}_{3,3}{\widehat G}_{1,1}({\widehat G}_{3,1}{\widehat G}_{1,2}-{\widehat G}_{3,2}-{\widehat G}_{2,3})
+{\widehat G}^2_{2,1}
+{\widehat G}^2_{3,2}
+{\widehat G}^2_{1,3}
\nn\\
\!\!\!\!\!
&{}&
\!\!\!\!\!
-{\widehat G}_{1,2}{\widehat G}_{2,3}{\widehat G}_{1,3}
-{\widehat G}_{2,3}{\widehat G}_{3,1}{\widehat G}_{2,1}
-{\widehat G}_{3,1}{\widehat G}_{1,2}{\widehat G}_{3,2}
+{\widehat G}^2_{1,2}+{\widehat G}^2_{2,3}+{\widehat G}^2_{3,1}
\nn\\
\!\!\!\!\!
&{}&
\!\!\!\!\!
+({\widehat G}^2_{1,1}+1)({\widehat G}^2_{2,2}+1)
+({\widehat G}^2_{2,2}+1)({\widehat G}^2_{3,3}+1)
+({\widehat G}^2_{3,3}+1)({\widehat G}^2_{1,1}+1)
\nn
\eea

\section{Frobenius manifolds in the vicinity of a non semi-simple point}\label{se:FM}

In this section we interpret our $\mathfrak D_n$ algebra as the
Poisson algebra of the Stokes data of a Frobenius manifold in the
vicintiy of a non semi--simple point.

Frobenius manifolds where introduced by  Dubrovin  \cite{dubrovin}
as coordinate--free formulation of the famous
Witten--Dijkgraaf--Verlinde--Verlinde (WDVV) equations. Loosely
speaking Frobenius manifolds are $n$ dimensional complex manifolds
together with a smooth structure of Frobenius algebra on the tangent
space. This is a commutative associative algebra with unity and with
an invariant non--degenerate bilinear form (some other conditions
must be satisfied, such as the existence of a grading vector field;
a precise definition may be found in \cite{dubrovin}).

Semi--simple Frobenius manifolds of dimension $n$ can be realized as
the space of parameters ${\bf u}=(u_1,\dots,u_n)$ together with a
$n\times n$ skew-symmetric matrix function $V({\bf u})$ such that
the linear differential operator
$$
\Lambda(z):= \ddz -U -\frac{V({\bf u})}{z},
$$
$U$ being a diagonal matrix of entries $u_1,\dots,u_n$, has constant
monodromy data\cite{dubrovin}. Generically, the monodromy data of
$\Lambda(z)$ are encoded in the so--called {\it Stokes matrix}\/
$S$, an upper triangular matrix with $1$ on the diagonal (for a
general definition of the monodromy data see  \cite{MJ1,MJ2,MJU}).
Equivalently, semi--simple Frobenius manifolds are identified with
the space of monodromy preserving deformations of a $n$--dimensional
Fuchsian system:\footnote{We use here the calligraphic notation to
distinguish the $n\times n$ system and its monodromy matrices from
the $2\times2$ systems considered in the previous sections.}
\be
\label{eq:FB}
\frac{{\rm d}\Psi}{{\rm d}z}=\sum_{k=1}^n\frac{{\mathcal A}_k({\bf u})}{\lambda-u_k}\Psi,
\ee
where the matrices ${\mathcal A}_1({\bf u}),\dots,{\mathcal
A}_n({\bf u})$ are solutions of the Schlesinger equations
(\ref{schleq}) and have the form:
\be
\label{eq:BkV}
{\mathcal A}_k({\bf u})= -E_k\left(V({\bf u})+\frac{1}{2}\ID\right),\quad\hbox{where }
 (E_k)_{ij}=\delta_{ik}\delta_{kj}, \quad k=1,\dots,n.
\ee
Let $S$ be the Stokes matrix associated to the differential operator
$\Lambda(z):= \ddz -U -\frac{V}{z}$ If  ${\rm
rank}\left(G\right)=n$, where $G=S+S^T$, then the monodromy matrices
of this system (\ref{eq:FB}) have the form:
\be
\label{eq:monS}
{\mathcal M}_k=\ID-E_k \left(S+S^T\right)=\ID-E_k G,\quad k=1,\dots,n.
\ee
By using the Korotkin--Samtleben bracket and the relation
$$
\Tr\left({\mathcal M}_i {\mathcal M}_j\right)= n- 4 +S_{ij}^2,
$$
Ugaglia constructed the Poisson bracket among the entries $S_{ij}$ of the Stokes matrix \cite{Ugaglia}.
She obtained the same formula as (\ref{eq:NR}) with $S_{ij}$ in place of $G_{i,j}$ (up to a factor
$-\frac12$).

Our interpretation of the $\mathfrak D_n$ algebra as the Poisson
algebra of the Stokes data of a Frobenius manifold in the vicinity
of a non semi--simple point is based on the observation that non
semi--simple points correspond to the critical points of the
Schlesinger equations, i.e. to the clashing of two or more poles in
the Fuchsian system (\ref{eq:FB}).

We use the same notation as in Sec.~\ref{se:clash}. In particular we fix a number $\tilde n$, we set
$$
\tilde{\bf u}:=(u_1,\dots,u_{\tilde n-1}),\quad\hbox{and} \quad
u_j:=t v_j, \quad j=\tilde n,\dots,\tilde n+m-1=n
$$
and we define $A_i(\tilde{\bf u},t)$ for $i=1,\dots,\tilde n-1$ and
and $B_j(\tilde{\bf u},t)$ for $j=1,\dots,m=n-\tilde n+1$ as in
(\ref{eq:AtoB}), so that  the Schlesinger equations in the variable
$t$ assume the form (\ref{eq:sch-red}).

\begin{theorem}
Assume that $V$ is non resonant and that its eigenvalues
$\mu_1,\dots,\mu_n$ have real part
$\Re\mu_i\in\left(-\frac12,\frac12\right)$. If ${\rm
rank}\left(S+S^T\right)=n$, then there exist some matrix functions
${\mathcal A}_1^0(\tilde{\bf u}),\dots{\mathcal A}_{\tilde
n-1}^0(\tilde{\bf u})$ and ${\mathcal B}_1^0(\tilde{\bf
u})\dots,{\mathcal B}_{n-\tilde n+1}^0(\tilde{\bf u})$ such that the
Fuchsian system
\be
\label{eq:FB1}
\frac{{\rm d}\tilde\Phi}{{\rm d}z}=\sum_{k=1}^{\tilde n}\frac{\tilde {\mathcal A}_k}{\lambda-u_k}\tilde \Phi,
\ee
with
\be
\label{eq:BkVhat}
u_{\tilde n}=0,\quad
\tilde{\mathcal A}_i = {\mathcal A}_i^0,\hbox{ for } i=1,\dots,\tilde n-1,\quad\hbox{and }\,
\tilde{\mathcal A}_{\tilde n}=\sum_{j=1}^{n-\tilde n+1}{\mathcal B}^0_j.
\ee
has monodromy matrices ${\mathcal M}_1,\dots, {\mathcal M}_{\tilde n-1}, {\mathcal M}_h$ where
\be\label{eq:Mh}
{\mathcal M}_h
={\mathcal M}_{\tilde n}{\mathcal M}_{\tilde n+1}\dots{\mathcal M}_{n}.
\ee
Moreover, the estimates (\ref{eq:estimates1}), (\ref{eq:estimates2})
for the solutions to the Schlesinger equations (\ref{eq:sch-red})
hold true for $\sigma$ such that $\thi<\sigma<1$, where $\thi$ is
given by
$$
\thi=\max_{i,j=\tilde n,\dots,n} \left|\Re\left(\mu_i-\mu_j\right)\right|.
$$
\end{theorem}

\begin{remark}
Observe that the hypothesis that the eigenvalues $\mu_1,\dots,\mu_n$
have real part $\Re\mu_i\in\left(-\frac12,\frac12\right)$ is not
restrictive, in fact in the non--resonant case one can always
perform a sequence of elementary Schlesinger transformations to
shift the eigenvalues by integer and reduce to this case \cite{MJ2}.
\end{remark}

\proof  First let us prove that if the symmetric matrix $G:=S+S^T$   has rank $n$ then the group
$$
\tilde{\mathcal M}:=\langle{\mathcal M}_1,\dots, {\mathcal M}_{\tilde n-1}, {\mathcal M}_h\rangle
$$
is irreducible. This is a simple consequence of (\ref{eq:monS}) and
the definition of $ {\mathcal M}_h={\mathcal M}_{\tilde n}{\mathcal
M}_{\tilde n+1}\dots{\mathcal M}_{n}$. In fact assume by
contradiction that the group $\tilde{\mathcal M}$ admits an
invariant subspace and pick a vector ${\bf v}=(v_1,\dots,v_n)^T$ in
it. Then ${\bf v}$ is invariant w.r.t. the full monodromy group
${\mathcal M}:=\langle{\mathcal M}_1,\dots, {\mathcal M}_{\tilde n-1}, {\mathcal M}_{\tilde n}\dots,{\mathcal M}_n\rangle$.
In fact by definition, ${\mathcal M}_h$ is given  by
\begin{eqnarray}
{\mathcal M}_h &=& \ID -\sum_{i=\tilde n}^n E_{i} G + \sum_{i=\tilde n}^{n-1}\sum_{j=i+1}^n E_i G E_j G -
\sum_{i=\tilde n}^{n-2} \sum_{j=i+1}^{n-1}\sum_{l=j+1}^n E_i G E_j G E_l G+\nn\\
&&
\dots +(-1)^m E_{\tilde n} G E_{\tilde n+1}G\dots E_n G,
\end{eqnarray}
which is a matrix whose $j$-th row for  $j=\tilde n,\dots,n$ coincides with the $j$-th row of:
$$
E_j -E_j G + E_j G  \sum_{i=j+1}^n E_i G - E_j G \sum_{i=j+1}^{n-1} E_i G \sum_{l=i+1}^n E_l G+\dots E_j G E_{j+1} G\dots E_n G.
$$
Then ${\mathcal M}_h {\bf v}={\bf v}$ iff $M_j  {\bf v}={\bf v}$ for $j=\tilde
n,\dots,n$. This proves that ${\bf v}$ is invariant w.r.t. the full
monodromy group ${\mathcal M}$. But:
$$
{\mathcal M}_j {\bf v}={\bf v}\,\forall j=1,\dots,n\, \Leftrightarrow E_j G{\bf v}=0\,\forall j=1,\dots,n\, \Rightarrow G{\bf v}=0,
$$
which for ${\rm rank}(G)=n$ gives a contradiction as we wanted.

Existence of the matrix functions
${\mathcal A}_1^0(\tilde{\bf u}),\dots{\mathcal A}_{\tilde n-1}^0(\tilde{\bf u})$,
${\mathcal B}_1^0(\tilde{\bf u})\dots,{\mathcal B}_{n-\tilde n+1}^0(\tilde{\bf u})$
follows from the fact that the monodromy group is irreducible \cite{Bol2}. To prove that the new system (\ref{eq:FB1})
with conditions (\ref{eq:BkVhat})
has monodromy matrices ${\mathcal M}_1,\dots, {\mathcal M}_{\tilde
n-1}, {\mathcal M}_h$ and to obtain the estimates
(\ref{eq:estimates1}), (\ref{eq:estimates2}) we apply the clashing
theorem \ref{th:JMS}. The only assumption we have to verify is that
the eigenvalues $\lambda_1,\dots,\lambda_n$ of
$\tilde A_{\tilde n}$ satisfy the technical assumption (\ref{eq:techn}). Thanks to relation (\ref{eq:BkV}), we have that
 $$
{\rm eigenvalues}\left(\sum_{j=\tilde n}^n{\mathcal B}_j^0 \right)=-\frac12,\dots,-\frac12, -\mu_{\tilde n}-\frac{1}{2},\dots,-\mu_{ n}-\frac{1}{2},
$$
so that the technical condition is always satisfied when $V$ is non-resonant.
\endproof

This theorem states that in the vicinity of a non semi-simple point
the Frobenius manifold  is identified with the space of monodromy
preserving deformations of the $n$--dimensional Fuchsian system
(\ref{eq:FB1}) with the conditions (\ref{eq:BkVhat}).

\begin{theorem}
The Poisson algebra of the monodromy data of the system (\ref{eq:FB1}) is given by ${\mathfrak D}_n$.
\end{theorem}

\proof
The Poisson algebra of the monodromy data of the system
(\ref{eq:FB1}) is given by the Korotkin--Samtleben bracket
restricted to the adjoint invariant functions such as
$$
\Tr({\mathcal M}_i {\mathcal M}_j),\quad\hbox{and}\quad
\Tr({\mathcal M}_i {\mathcal M}_h^k {\mathcal M}_j{\mathcal M}_h^{-k}),
$$
where now the monodromy matrices are $n\times n$. Due to (\ref{eq:monS}), it is easy to prove that
$$
\Tr({\mathcal M}_i {\mathcal M}_h^k {\mathcal M}_j{\mathcal M}_h^{-k})=
\Tr\left({\mathcal M}_j -E_i G + E_i G {\mathcal M}_h^k E_j G {\mathcal M}_h^{-k}\right)= n-4 + \left(G {\mathcal M}_h^{k}\right)^2_{ij}
$$
where the last step is due to the identity $G {\mathcal M}_h={\mathcal M}_h^{-T} G$,
 which is a straightforward consequence of (\ref{eq:monS}) and  (\ref{eq:Mh}). Defining
$$
G_{i,j}^{(k)} =  \left(G {\mathcal M}_h^{k}\right)_{ij}
$$
we get that the Poisson brackets among the elements $G_{i,j}^{(k)}$ coincide with (\ref{eq:newDn})
up to a factor $-\frac12$. \endproof

\subsection{Level-$p$ reduction in the case of Frobenius manifolds}

Here we study which restrictions on the elements $G_{i,j}$, $i,j=1,\dots,n$, we must impose
to ensure the satisfaction of the level-$p$ reduction condition ${\mathcal M}_h^p=\ID$ (\ref{level-p}).
\begin{prop}
The product of monodromy matrices ${\mathcal M}_h:={\mathcal M}_{\tilde n}{\mathcal M}_{\tilde n+1}\dots {\mathcal M}_{n}$
has the following block-matrix structure
\be
\def\mystrut{\phantom{$\biggl|$}}
{\mathcal M}_h=\left[\begin{tabular}{c|c}
        $\ID_{(\tilde n-1)\times (\tilde n-1)}$ & \mystrut${\mathbb O}$
        \\
        \hline
        $B$ & \mystrut$-{\widetilde S}^{-1}{\widetilde S}^T$
        \\
        \end{tabular}
            \right],
\label{Mh}
\ee
where $B$ is an $(n-\tilde n+1)\times(\tilde n -1)$ matrix whose entries are
polynomials in $G_{i,j}$ with $i=1,\dots,n$ and $j=\tilde n,\dots, n$ and
${\widetilde S}$ is the $(n-\tilde n+1)\times(n-\tilde n +1)$ upper-triangular matrix
with unities on the diagonal and with its $(i,j)$ entry above the diagonal equal to
$G_{\tilde n+i-1,\tilde n+j-1}$.
\end{prop}

\proof
From the explicit form of ${\mathcal M}_i$ (\ref{eq:monS}) it follows that only the last $n-\tilde n+1$ lines of
${\mathcal M}_h$ differ from the unit matrix; only $G_{i,j}$ with {\em both} $i$ and $j$ greater or equal $\tilde n$ contribute
to the expression in the lower right square block and we let ${\widetilde {\mathcal M}}_r$ with $r=\tilde n,\dots n$
denote the lower-right $(n-\tilde n+1)\times(n-\tilde n +1)$-matrix blocks of
the corresponding monodromy matrices ${\mathcal M}_r$. The proposition assertion then follows
from the Dubrovin's identity ${\widetilde {\mathcal M}}_{\tilde n}{\widetilde {\mathcal M}}_{\tilde n+1}\dots
{\widetilde {\mathcal M}}_{n}=-{\widetilde S}^{-1}{\widetilde S}^T$.\footnote{Recall that this formula follows from the
chain of matrix equalities: $SE_1=E_1$, $E_iS^TE_j=0$ for $i<j$, $E_iSE_j=\delta_{i,j}E_j$ for $i\ge j$;
then $S{\mathcal M}_1=S-E_1(S+S^T)=(E_2+\cdots+E_n)S-E_1S^T$. Multiplying this expression by ${\mathcal M}_2$
from the right and using the above formulas, we obtain $(E_3+\cdots+E_n)S-(E_1+E_2)S^T$ as the result; we then
multiply it by ${\mathcal M}_3$ from the right and continue until we obtain that $S{\mathcal M}_1{\mathcal M}_2\dots {\mathcal M}_n=
-(E_1+E_2+\cdots+E_n)S^T=-S^T$.}
\endproof

We now introduce the notation $\widetilde {\mathcal M}_h:=-{\widetilde S}^{-1}{\widetilde S}^T$. We then have the following
lemma.
\begin{lm}\label{lm:S}
The condition (\ref{level-p}) is satisfied if $(\widetilde {\mathcal M}_h)^p=\ID$ and the symmetric form
${\widetilde S}+{\widetilde S}^T$ is nondegenerate.
\end{lm}

\proof
From the explicit form of ${\mathcal M}_h$ we have that the condition (\ref{level-p}) is equivalent to the
simultaneous satisfaction of the two conditions:
\begin{itemize}
\item[(i)] \ $\bigl(\ID+ {\widetilde {\mathcal M}_h}+(\widetilde {\mathcal M}_h)^2+\cdots+(\widetilde {\mathcal M}_h)^{p-1}\bigr)B=0$ and
\item[(ii)] \ $(\widetilde {\mathcal M}_h)^p=\ID$.
\end{itemize}
Multiplying the first condition by $(\widetilde {\mathcal M}_h-\ID)$ and using the second condition we obtain the
identity. If the matrix $(\widetilde {\mathcal M}_h-\ID)$ is nondegenerate this implies the satisfaction of the
first condition. But this nondegeneracy condition is exactly the condition of the nondegeneracy of the
symmetric form ${\widetilde S}+{\widetilde S}^T$ (upon the multiplication by ${\widetilde S}$ from the
right).
\endproof

\begin{example}
Let us consider the case of {\em arbitrary} $m=n-\tilde n+1\ge2$ and $p=m+1$. Then, a convenient choice
is $G_{i,j}\equiv1$ for $\tilde n \le i<j\le n$. Indeed, we then have that the characteristic equation
$\det({\widetilde M}_h-\eta\ID)=0$ is equivalent to
$$
\det\left[\begin{array}{cccc}
                                                      1+\eta & \eta & \cdots & \eta \\
                                                      1 & 1+\eta & \cdots & \eta \\
                                                      \vdots & \vdots & \ddots & \vdots \\
                                                      1 & 1 & \cdots & 1+\eta
                                                    \end{array}
\right]=1+\eta+\cdots+\eta^{m}=0,
$$
and we have exactly $m=p-1$ single roots that are $e^{2\pi i k/p}$, $k=1,\dots,p-1$; each root corresponds to a
nondegenerate eigenvalue with the
corresponding eigenvector of ${\widetilde M}_h$, and we therefore have that all the conditions of Lemma~\ref{lm:S}
are satisfied.
\end{example}

\subsection{Quantum cohomology of $\mathbb C\mathbb P^2$ and $\mathbb C\mathbb P^3$}

The fact that the Poisson algebra  (\ref{eq:NR})  coincides with the
Ugaglia bracket on the space of Stokes matrices of a semi--simple
Frobenius manifold poses the natural question of characterizing the
special class of semi--simple Frobenius manifolds coming from
Teichm\"uller theory. This is a highly non trivial problem that we
postpone to subsequent work \cite{ChMa}.  In this section we
concentrate on two particular cases: $A_3$ in the limit
$Z_1=Z_2=Z_3=0$ and $A_4$ in the limit
$Z_1=-Z_2=Z_3=-Z_4=\frac{\log(2)}{2}$, $Y=0$, let us dub them
$A_3^\ast$ and $A_4^\ast$ respectively. Through the identification
of the matrix $\mathcal A$ defined in (\ref{A-matrix}) with the
Stokes matrix $S$ associated to the Frobenius manifold structure, we
build a link between the $A_3^\ast$ and $A_4^\ast$ and the quantum
cohomology rings $H^\ast(\mathbb C\mathbb P^2)$ and $H^\ast(\mathbb
C\mathbb P^3)$ respectively. In fact:

\begin{theorem}
The matrices $\mathcal A$ in the cases $A_3^\ast$ and  $A_4^\ast$ have the form
$$
\left(\begin{array}{ccc}
1&3&3\\
0&1&3\\
0&0&1\\
\end{array}\right)
\quad \hbox{and}\quad
\left(\begin{array}{cccc}
1&4&6&4\\
0&1&4&6\\
0&0&1&4\\
0&0&0&1\\
\end{array}\right),
$$
which coincide with the Stokes matrix of the respective quantum cohomology
rings $H^\ast(\mathbb C\mathbb P^2)$ \cite{dubrovin} and
$H^\ast(\mathbb C\mathbb P^3)$ \cite{guzzetti}.
\end{theorem}

\proof The proof of this theorem is straightforward, it is simply
based on plugging in $G_{i,j}=\Tr(\gamma_i\gamma_j)$ the appropriate
values of the shear coordinates.\endproof

This link with the quantum cohomology of the projective space is
only valid in low dimension, i.e. for $n=3,4$. However it gives some
insight on the nature of the solutions of the Schlesinger equations
related to Teichm\"uller theory: we expect them to be
transcendental. In fact, on the one side we observe that generically
the monodromy group $\langle\gamma_1,\dots,\gamma_n\rangle$ is
irreducible and none of the matrices $\gamma_i$ is a multiple of the
identity, therefore there is no evidence that these solutions should
be special functions \cite{mazz}. On the other side, the Schlesinger
equations associated to $H^\ast(\mathbb C\mathbb P^2)$ are solved in
terms of Painlev\'e~VI transcendents. For the moment we can only
prove the following:

\begin{theorem}
The solutions to the $2\times 2$ Schlesinger equations (\ref{schleq}) with monodromy matrices
$\gamma_1,\dots,\gamma_n$ given by (\ref{eq:basisn}) are not algebraic in $u_1,\dots,u_n$.
\end{theorem}

\proof
This is a simple consequence of the fact that the analytic
continuation of the solutions to the Schlesinger equations is  given
by the action of the braid group on the monodromy matrices
(\ref{eq:braidDM}), which in terms of
$G_{i,j}=-\Tr(\gamma_i\gamma_J)$ is given by formulae
(\ref{eq:braidch}). In the geometric case, i.e. when the monodromy
matrices $\gamma_1,\dots,\gamma_n$ are given by (\ref{eq:basisn}),
it is easy to verify that $|G_{i,j}|>2$ and the braid group orbits
are therefore infinite as proved in \cite{DM}. Therefore the
corresponding solution to the Schlesinger equations cannot be
algebraic.
\endproof



\setcounter{section}{0}

\def\thetheorem{A.\arabic{theorem}}
\def\theprop{A.\arabic{prop}}
\def\thelemma{A.\arabic{lm}}
\def\thecor{A.\arabic{cor}}
\def\theexam{A.\arabic{exam}}
\def\theremark{A.\arabic{remark}}
\def\theequation{A.\arabic{equation}}

\appendix{Monodromy data}\label{se:mon-data}

A general description of monodromy data of linear systems of ODE can
be found in \cite{MJ1,MJ2,MJU}. Here we remind the notations and
definitions for $m\times m$ Fuchsian systems. We work in the basis where $A_\infty$ is diagonal.

Fix a real number $\varphi\in[0,2\pi[$ and consider the open subset ${\mathcal U}\in X_n$ such that
the rays $L_1,...,L_n$ defined by
\begin{equation}\label{eq:rays}
L_j:=\{u_j+i\rho e^{-i\varphi}\,|\, 0\leq\rho<\infty\}
\end{equation}
do not intersect. We assume that the points
$(u_1,\dots,u_n)\in{\mathcal U}$  are ordered in such a way that the
rays $L_1,\dots,L_n$ exit from infinity in clockwise order.

Let us fix a fundamental matrix solution near all singular points
$u_1$, \dots, $u_n$, $\infty$. To this end we  fix branch cuts on
the complex plane along the rays $L_1,\dots,L_n$ and choose the
branches of logarithms $\log(\lambda-u_1)$, \dots, $\log
(\lambda-u_n)$, $\log \lambda^{-1}$.  Assume $A_1,\dots,A_n$ to be
diagonalizable
$$
A_i = \Gamma_i^{-1}\Theta_i \Gamma_i,$$
where $\Theta_i$ is a diagonal matrix.

We can fix  fundamental matrices analytic on
\begin{equation}\label{domain}
\lambda\in {\mathbb C} \setminus \cup_{k=1}^n L_k,
\end{equation}
as follows:
\begin{equation}
\Phi_k({\lambda})=\left(\Gamma_k +{\mathcal O}({\lambda}-u_k)\right)
({\lambda}-u_k)^{\Theta_i}, \quad \lambda\to u_k, \quad
k=1, \dots, n,
\label{N6.1}
\end{equation}
and
\begin{equation}
\Phi({\lambda}):=
\Phi_\infty({\lambda})=\left(\ID+{\mathcal O}({1\over \lambda})\right)
{\lambda}^{-{\Theta_\infty}}{\lambda}^{-{R}^{(\infty)}},\quad\hbox{as}\quad
{\lambda}\rightarrow\infty,
\label{M3}
\end{equation}
where the two linear operators ${\Theta_\infty}$, $R^{(\infty)}$ are an {\it admissible pair} i.e.
the operator ${\Theta_\infty}$ is semisimple and the operator $R^{(\infty)}$ is
nilpotent and they satisfy the relation
\begin{equation}\label{lev1}
e^{2\pi i \Theta_\infty} R^{(\infty)} = R^{(\infty)}\, e^{2 \pi i \Theta_\infty}.
\end{equation}
Define the {\it connection matrices}  by
\begin{equation}
\Phi_\infty({\lambda})=\Phi_k({\lambda}){ C}_k,\label{N3}
\end{equation}
where $\Phi_\infty (\lambda)$ is to be analytically continued in a vicinity of the
pole $u_k$ along the positive side of the branch cut $L_k$.

The monodromy matrices ${M}_k$, $k=1,\dots,n,\infty$ are defined
with respect the basis $l_1,\dots,l_{n}$ of loops in the
fundamental group
$$
\pi_1\left(\mathbb C\backslash\{u_1,\dots u_{n}\},
\infty\right),
$$
chosen by imposing that the  small loops $l_1,\dots,l_n$ encircle
counter--clockwise the points $u_1,\dots,u_n$. Denote $l_j^*
\Phi_\infty(\lambda)$ the result of  analytic continuation of the
fundamental matrix $\Phi_\infty(\lambda)$ along the loop $l_j$. The
monodromy matrix ${M}_j$ is defined by
\begin{equation}
\label{momo}
l_j^* \Phi_\infty (\lambda) = \Phi_\infty (\lambda) {M}_j, ~~j=1, \dots, n.
\end{equation}
The monodromy matrices satisfy
\begin{equation}
{M}_\infty {M}_{1} \cdots {M}_n=\ID,\qquad
{M}_\infty=C_\infty^{-1}\exp\left(2\pi i{A_\infty}\right)
\exp\left(2\pi i{R^{(\infty)}}\right)C_\infty,
\label{N6}
\end{equation}
for some constant matrix $C_\infty$,
because of the choice of the ordering of the branch cuts $L_1$, \dots, $L_{n}$. Clearly one has
\begin{equation}
{M}_k={ C}_k^{-1} \exp\left(2\pi i \Theta_k\right)
{ C}_k,\qquad k=1,\dots,n.
\label{N4}
\end{equation}

The collection of the local monodromy data $\Theta_1,\dots,\Theta_n,\Theta_\infty$, $R^{(\infty)}$ together
with the central connection matrices $C_1,\dots,C_n,C_\infty$ uniquely fix the Fuchsian system with given poles.
They are  defined up to an equivalence that we now describe. The eigenvalues ${\Theta_\infty}$
of the matrix ${A_\infty}$
are defined up to permutations. Fixing the order of the eigenvalues, we define
the class of equivalence of the nilpotent part $R^{(\infty)}$ and of the connection matrices
$C_1,\dots,C_n,C_\infty$ by factoring out the transformations of the form
\begin{eqnarray}\label{class-mon}
&&
C_k \mapsto G_k ^{-1}C_k G_\infty, \quad k=1, \dots, n,
\quad C_\infty \mapsto G_\infty ^{-1}C_k G_\infty
\end{eqnarray}
where $G_k\in{\rm GL}(n,{\mathbb C})$ is such that
$$
[G_k,\Theta_k]=0
$$
and $G_\infty\in GL(n,{\mathbb C})$ is such that
\begin{equation}\label{cinf0}
\lambda^{-{\Theta_\infty}}  G \,  \lambda^{{\Theta_\infty}}
=G_0 +\frac{G_1}{\lambda} +\frac{G_2}{\lambda^2}+\dots ,
\end{equation}
for some constant matrices $G_0,G_1,G_2....$.

Observe that the monodromy matrices (\ref{N4}) will transform by a simultaneous conjugation
$$
M_k \mapsto G_\infty^{-1} M_k G_\infty, \quad k=1, 2, \dots, n, \infty.
$$
\begin{df} The class of equivalence (\ref{class-mon}) of the collection
\begin{equation}\label{m-data}
\Theta_1,\dots,\Theta_n,\Theta_\infty, {R^{(\infty)}},  C_1, \dots, C_n,C_\infty
\end{equation}
is called {\it monodromy data} of the Fuchsian system with respect
to a fixed ordering of the eigenvalues of the matrix $A_\infty$ and
a given choice of the branch cuts.
\end{df}

\begin{lm}
Two Fuchsian systems of the form (\ref{N1in}) with the same poles
$u_1,\dots,u_{n}$ and $\infty$ and the same matrix $A_\infty$
coincide, modulo conjugations by matrice $G$ such that
$[A_\infty,G]=0$, if and only if they have the same monodromy data
with respect to the same system of branch cuts $L_1,\dots,L_{n}$.
\label{lm2.8}
\end{lm}

\def\thetheorem{B.\arabic{theorem}}
\def\theprop{B.\arabic{prop}}
\def\thelemma{B.\arabic{lm}}
\def\thecor{B.\arabic{cor}}
\def\theexam{B.\arabic{exam}}
\def\theremark{B.\arabic{remark}}
\def\theequation{B.\arabic{equation}}

\appendix{Proof of Jacobi identities for the $\mathfrak D_n$ bracket}\label{se:Jacobi}

In this appendix, we prove the Jacobi identity:
\be
\bigl\{\left\{{\mathcal G}_{j,i}(\lambda),{\mathcal G}_{p,l}(\mu)\right\},{\mathcal G}_{q,r}(\nu)\bigr\}
\label{jiplqr}+\hbox{ cyclic permutations } =0
\ee
for the bracket (\ref{eq:Yangian}).
We proceed in three steps.

\noindent{\bf 1.} Without restricting the generality, we segregate all the terms containing the
terms ${\mathcal G}_{j,r}$ with all possible arguments. There are three cases.

{\bf 1a.} Terms in $\mathcal G_{j,r}\mathcal G_{q,l}\mathcal G_{p,i}$ with any choice of the arguments.
In this case, only the first two lines of formula
(\ref{eq:Yangian}) contribute. Then, with accounting for the cyclic
permutations, we have the sum of twelve terms
\begin{eqnarray}
&&
\left(\epsilon(j-p)+\frac{\lambda+\mu}{\lambda-\mu}\right)
\left(\epsilon(j-q)+\frac{\mu+\nu}{\mu-\nu}\right)
{\mathcal G}_{p,i}(\lambda) {\mathcal G}_{q,l}(\mu){\mathcal G}_{j,r}(\nu)
\nn\\
&+&
\left(\epsilon(j-p)+\frac{\lambda+\mu}{\lambda-\mu}\right)
\left(\epsilon(l-r)-\frac{\mu+\nu}{\mu-\nu}\right)
{\mathcal G}_{p,i}(\lambda) {\mathcal G}_{q,l}(\nu){\mathcal G}_{j,r}(\mu)
\nn\\
&+&
\left(\epsilon(i-l)-\frac{\lambda+\mu}{\lambda-\mu}\right)
\left(\epsilon(j-q)+\frac{\lambda+\nu}{\lambda-\nu}\right)
{\mathcal G}_{p,i}(\mu) {\mathcal G}_{q,l}(\lambda){\mathcal G}_{j,r}(\nu)
\nn\\
&+&
\left(\epsilon(i-l)-\frac{\lambda+\mu}{\lambda-\mu}\right)
\left(\epsilon(l-r)-\frac{\lambda+\nu}{\lambda-\nu}\right)
{\mathcal G}_{p,i}(\mu) {\mathcal G}_{q,l}(\nu){\mathcal G}_{j,r}(\lambda)
\nn\\
&+&
\left(\epsilon(p-q)+\frac{\mu+\nu}{\mu-\nu}\right)
\left(\epsilon(p-j)+\frac{\nu+\lambda}{\nu-\lambda}\right)
{\mathcal G}_{p,i}(\lambda) {\mathcal G}_{q,l}(\mu){\mathcal G}_{j,r}(\nu)
\nn\\
&+&
\left(\epsilon(p-q)+\frac{\mu+\nu}{\mu-\nu}\right)
\left(\epsilon(r-i)-\frac{\nu+\lambda}{\nu-\lambda}\right)
{\mathcal G}_{p,i}(\nu) {\mathcal G}_{q,l}(\mu){\mathcal G}_{j,r}(\lambda)
\nn\\
&+&
\left(\epsilon(l-r)-\frac{\mu+\nu}{\mu-\nu}\right)
\left(\epsilon(p-j)+\frac{\mu+\lambda}{\mu-\lambda}\right)
{\mathcal G}_{p,i}(\lambda) {\mathcal G}_{q,l}(\nu){\mathcal G}_{j,r}(\mu)
\nn\\
&+&
\left(\epsilon(l-r)-\frac{\mu+\nu}{\mu-\nu}\right)
\left(\epsilon(r-i)-\frac{\mu+\lambda}{\mu-\lambda}\right)
{\mathcal G}_{p,i}(\mu) {\mathcal G}_{q,l}(\nu){\mathcal G}_{j,r}(\lambda)
\nn\\
&+&
\left(\epsilon(q-j)+\frac{\nu+\lambda}{\nu-\lambda}\right)
\left(\epsilon(q-p)+\frac{\lambda+\mu}{\lambda-\mu}\right)
{\mathcal G}_{p,i}(\lambda) {\mathcal G}_{q,l}(\mu){\mathcal G}_{j,r}(\nu)
\nn\\
&+&
\left(\epsilon(q-j)+\frac{\nu+\lambda}{\nu-\lambda}\right)
\left(\epsilon(i-l)-\frac{\lambda+\mu}{\lambda-\mu}\right)
{\mathcal G}_{p,i}(\mu) {\mathcal G}_{q,l}(\lambda){\mathcal G}_{j,r}(\nu)
\nn\\
&+&
\left(\epsilon(r-i)-\frac{\nu+\lambda}{\nu-\lambda}\right)
\left(\epsilon(q-p)+\frac{\nu+\mu}{\nu-\mu}\right)
{\mathcal G}_{p,i}(\nu) {\mathcal G}_{q,l}(\mu){\mathcal G}_{j,r}(\lambda)
\nn\\
&+&
\left(\epsilon(r-i)-\frac{\nu+\lambda}{\nu-\lambda}\right)
\left(\epsilon(i-l)-\frac{\nu+\mu}{\nu-\mu}\right)
{\mathcal G}_{p,i}(\mu) {\mathcal G}_{q,l}(\nu){\mathcal G}_{j,r}(\lambda)
\nn
\end{eqnarray}
Note the cancelations between the lines with the numbers $2$ and $7$, $3$ and $10$, and $6$ and $11$.
The lines with the numbers $1$, $5$, and $9$ are proportional to the term
${\mathcal G}_{p,i}(\lambda) {\mathcal G}_{q,l}(\mu){\mathcal G}_{j,r}(\nu)$ with the
proportionality coefficient
\begin{eqnarray}
&&
\left(\epsilon(j-p)+\frac{\lambda+\mu}{\lambda-\mu}\right)
\left(\epsilon(j-q)+\frac{\mu+\nu}{\mu-\nu}\right)
\nn\\
&&+
\left(\epsilon(p-q)+\frac{\mu+\nu}{\mu-\nu}\right)
\left(\epsilon(p-j)+\frac{\nu+\lambda}{\nu-\lambda}\right)
\nn\\
&&+
\left(\epsilon(q-j)+\frac{\nu+\lambda}{\nu-\lambda}\right)
\left(\epsilon(q-p)+\frac{\lambda+\mu}{\lambda-\mu}\right)
\nn\\
&=&\bigl(\epsilon(j-p)\epsilon(j-q)+\epsilon(p-q)\epsilon(p-j)+\epsilon(q-j)\epsilon(q-p)\bigr)
\nn\\
&&+\left(\frac{\lambda+\mu}{\lambda-\mu}\cdot\frac{\mu+\nu}{\mu-\nu}
+\frac{\mu+\nu}{\mu-\nu}\cdot\frac{\nu+\lambda}{\nu-\lambda}
+\frac{\nu+\lambda}{\nu-\lambda}\cdot\frac{\lambda+\mu}{\lambda-\mu}\right).
\nn
\end{eqnarray}
The last line of this expression sums up to $-1$ whereas the combination of the $\epsilon$-factors
in the next to the last line is $1$ unless $j=p=q$ in which case it vanishes. Therefore this coefficient is
$$
-\delta_{jp}\delta_{pq}.
$$
Analogously, the lines with
the numbers $4$, $8$, and $12$ are proportional to
${\mathcal G}_{p,i}(\mu) {\mathcal G}_{q,l}(\nu){\mathcal G}_{j,r}(\lambda)$ with $\delta_{il}\delta_{lr}$.

\medskip

{\bf 1b.}  Terms in $\mathcal G_{j,r}\mathcal G_{q,p}\mathcal G_{i,l}$ with any choice of the arguments.

In this case, all four lines of (\ref{eq:Yangian}) contribute and we have the sum of
five terms
\begin{eqnarray}
&&
\left(\epsilon(i-p)+\frac{1+\lambda\mu}{1-\lambda\mu}\right)
\left(\epsilon(j-q)+\frac{\lambda+\nu}{\lambda-\nu}\right)
{\mathcal G}_{i,l}(\mu) {\mathcal G}_{q,p}(\lambda){\mathcal G}_{j,r}(\nu)
\nn\\
&+&
\left(\epsilon(i-p)+\frac{1+\lambda\mu}{1-\lambda\mu}\right)
\left(\epsilon(p-r)-\frac{\lambda+\nu}{\lambda-\nu}\right)
{\mathcal G}_{i,l}(\mu) {\mathcal G}_{q,p}(\nu){\mathcal G}_{j,r}(\lambda)
\nn\\
&+&
\left(\epsilon(p-r)-\frac{1+\mu\nu}{1-\mu\nu}\right)
\left(\epsilon(r-i)-\frac{1+\mu\lambda}{1-\mu\lambda}\right)
{\mathcal G}_{i,l}(\mu) {\mathcal G}_{q,p}(\nu){\mathcal G}_{j,r}(\lambda)
\nn\\
&+&
\left(\epsilon(q-j)+\frac{\nu+\lambda}{\nu-\lambda}\right)
\left(\epsilon(i-p)+\frac{1+\lambda\mu}{1-\lambda\mu}\right)
{\mathcal G}_{i,l}(\mu) {\mathcal G}_{q,p}(\lambda){\mathcal G}_{j,r}(\nu)
\nn\\
&+&
\left(\epsilon(r-i)-\frac{\nu+\lambda}{\nu-\lambda}\right)
\left(\epsilon(i-p)+\frac{1+\mu\nu}{1-\mu\nu}\right)
{\mathcal G}_{i,l}(\mu) {\mathcal G}_{q,p}(\nu){\mathcal G}_{j,r}(\lambda)
\nn
\end{eqnarray}
Here, again, the first line is canceled with the fourth line and the remaining lines
are all proportional to the term
${\mathcal G}_{i,l}(\mu) {\mathcal G}_{q,p}(\nu){\mathcal G}_{j,r}(\lambda)$
with the proportionality coefficient
\begin{eqnarray}
&&
\left(\epsilon(i-p)+\frac{1+\lambda\mu}{1-\lambda\mu}\right)
\left(\epsilon(p-r)-\frac{\lambda+\nu}{\lambda-\nu}\right)
\nn\\
&+&
\left(\epsilon(p-r)-\frac{1+\mu\nu}{1-\mu\nu}\right)
\left(\epsilon(r-i)-\frac{1+\mu\lambda}{1-\mu\lambda}\right)
\nn\\
&+&
\left(\epsilon(r-i)-\frac{\nu+\lambda}{\nu-\lambda}\right)
\left(\epsilon(i-p)+\frac{1+\mu\nu}{1-\mu\nu}\right)
\nn\\
&=&
\left(\epsilon(i-p)\epsilon(p-r)+\epsilon(p-r)\epsilon(r-i)+\epsilon(r-i)\epsilon(i-p)\right)
\nn\\
&+&
\left(-\frac{1+\lambda\mu}{1-\lambda\mu}\cdot\frac{\lambda+\nu}{\lambda-\nu}
+\frac{1+\mu\nu}{1-\mu\nu}\cdot\frac{1+\mu\lambda}{1-\mu\lambda}
-\frac{\nu+\lambda}{\nu-\lambda}\cdot\frac{1+\mu\nu}{1-\mu\nu}
\right)\nn=\\
&=&
\delta_{ir}\delta_{rp}.
\nn
\end{eqnarray}

\medskip

{\bf 1c.} Terms in $\mathcal G_{j,r}\mathcal G_{p,q}\mathcal G_{l,i}$ with any choice of the arguments.

In this case we again have the sum of five terms
\begin{eqnarray}
&&
\left(\epsilon(l-j)-\frac{1+\lambda\mu}{1-\lambda\mu}\right)
\left(\epsilon(q-j)+\frac{1+\mu\nu}{1-\mu\nu}\right)
{\mathcal G}_{l,i}(\lambda) {\mathcal G}_{p,q}(\mu){\mathcal G}_{j,r}(\nu)
\nn\\
&+&
\left(\epsilon(q-l)+\frac{1+\mu\nu}{1-\mu\nu}\right)
\left(\epsilon(j-l)+\frac{\nu+\lambda}{\nu-\lambda}\right)
{\mathcal G}_{l,i}(\lambda) {\mathcal G}_{p,q}(\mu){\mathcal G}_{j,r}(\nu)
\nn\\
&+&
\left(\epsilon(q-l)+\frac{1+\mu\nu}{1-\mu\nu}\right)
\left(\epsilon(i-r)-\frac{\nu+\lambda}{\nu-\lambda}\right)
{\mathcal G}_{l,i}(\nu) {\mathcal G}_{p,q}(\mu){\mathcal G}_{j,r}(\lambda)
\nn\\
&+&
\left(\epsilon(j-q)+\frac{\nu+\lambda}{\nu-\lambda}\right)
\left(\epsilon(l-q)-\frac{1+\lambda\mu}{1-\lambda\mu}\right)
{\mathcal G}_{l,i}(\lambda) {\mathcal G}_{p,q}(\mu){\mathcal G}_{j,r}(\nu)
\nn\\
&+&
\left(\epsilon(i-r)-\frac{\nu+\lambda}{\nu-\lambda}\right)
\left(\epsilon(l-q)-\frac{1+\mu\nu}{1-\mu\nu}\right)
{\mathcal G}_{l,i}(\nu) {\mathcal G}_{p,q}(\mu){\mathcal G}_{j,r}(\lambda).
\nn
\end{eqnarray}
Here the third line cancels with the fifth line, and the remaining
lines are all proportional to the term
${\mathcal G}_{l,i}(\lambda) {\mathcal G}_{p,q}(\mu){\mathcal G}_{j,r}(\nu)$
with the proportionality coefficient
\begin{eqnarray}
&&
\left(\epsilon(l-j)-\frac{1+\lambda\mu}{1-\lambda\mu}\right)
\left(\epsilon(q-j)+\frac{1+\mu\nu}{1-\mu\nu}\right)+\nn\\
&+&
\left(\epsilon(q-l)+\frac{1+\mu\nu}{1-\mu\nu}\right)
\left(\epsilon(j-l)+\frac{\nu+\lambda}{\nu-\lambda}\right)+\nn\\
&+&
\left(\epsilon(j-q)+\frac{\nu+\lambda}{\nu-\lambda}\right)
\left(\epsilon(l-q)-\frac{1+\lambda\mu}{1-\lambda\mu}\right)=\nn\\
&=&
\left(\epsilon(l-j)\epsilon(q-j)+\epsilon(q-l)\epsilon(j-l)+\epsilon(j-q)\epsilon(l-q)\right)
\nn\\
&+&
\left(-\frac{1+\lambda\mu}{1-\lambda\mu}\cdot\frac{1+\mu\nu}{1-\mu\nu}
+\frac{1+\mu\nu}{1-\mu\nu}\cdot\frac{\nu+\lambda}{\nu-\lambda}
-\frac{\nu+\lambda}{\nu-\lambda}\cdot\frac{1+\lambda\mu}{1-\lambda\mu}
\right)\nn=
-\delta_{jl}\delta_{lq}.
\nn
\end{eqnarray}

{\bf 2.} From the above cases 1a--1c, we see that the Jacobi identity is satisfied separately
for every distribution of indices unless at least three among the indices
$j$, $i$, $p$, $l$, $q$, and $r$ coincide. So, let us consider the bracket with three coinciding indices,
$\bigl\{\left\{{\mathcal G}_{s,i}(\lambda),{\mathcal G}_{s,l}(\mu)\right\},{\mathcal G}_{s,r}(\nu)\bigr\}+\hbox{ cyclic permutations.}$ \
Let us follow the term ${\mathcal G}_{s,i}(\nu){\mathcal G}_{s,l}(\lambda){\mathcal G}_{s,r}(\mu)$.
The coefficient by this term is $
\left(\epsilon(i-l)-\frac{\lambda+\mu}{\lambda-\mu}\right)
\left(\epsilon(i-r)-\frac{\mu+\nu}{\mu-\nu}\right)+\hbox{ cyclic permutations}
$
and it is again easy to see that this term vanishes unless $i=l=r$.  So finally we are left with the case:
$$
\bigl\{\left\{{\mathcal G}_{s,i}(\lambda),{\mathcal G}_{s,i}(\mu)\right\},{\mathcal G}_{s,i}(\nu)\bigr\}+
\bigl\{\left\{{\mathcal G}_{s,i}(\nu),{\mathcal G}_{s,i}(\lambda)\right\},{\mathcal G}_{s,i}(\mu)\bigr\}+
\bigl\{\left\{{\mathcal G}_{s,i}(\mu),{\mathcal G}_{s,i}(\nu)\right\},{\mathcal G}_{s,i}(\lambda)\bigr\}.
$$
In this case, we have
$$
\left\{{\mathcal G}_{s,i}(\lambda),{\mathcal G}_{s,i}(\mu)\right\}=\left(\epsilon(i-s)+\frac{1+\lambda\mu}{1-\lambda\mu}\right)
\left({\mathcal G}_{s,s}(\lambda){\mathcal G}_{i,i}(\mu)-{\mathcal G}_{s,s}(\mu){\mathcal G}_{i,i}(\lambda)\right)
$$
and the result of the double bracket before applying the cyclic symmetry reads
\bea
\allowdisplaybreaks
&&
\bigl\{\left\{{\mathcal G}_{s,i}(\lambda),{\mathcal G}_{s,i}(\mu)\right\},{\mathcal G}_{s,i}(\nu)\bigr\}
=\left(\epsilon(i-s)+\frac{1+\lambda\mu}{1-\lambda\mu}\right)\times
\nn\\
&&\times\biggl\{
\frac{\lambda+\nu}{\lambda-\nu}\bigl(
{\mathcal G}_{s,i}(\nu)[{\mathcal G}_{i,i}(\lambda){\mathcal G}_{s,s}(\mu)+{\mathcal G}_{i,i}(\mu){\mathcal G}_{s,s}(\lambda)]
-\bigr.
\nn\\
&&\qquad\qquad\qquad\qquad\qquad
\bigl.
-{\mathcal G}_{s,i}(\lambda)[{\mathcal G}_{i,i}(\nu){\mathcal G}_{s,s}(\mu)+{\mathcal G}_{i,i}(\mu){\mathcal G}_{s,s}(\nu)]
\bigr)+
\biggr.\nn\\
&&+\frac{\mu+\nu}{\mu-\nu}\bigl(
{\mathcal G}_{s,i}(\mu)[{\mathcal G}_{i,i}(\lambda){\mathcal G}_{s,s}(\nu)+{\mathcal G}_{i,i}(\nu){\mathcal G}_{s,s}(\lambda)]
-\bigr.
\nn\\
&&\qquad\qquad\qquad\qquad\qquad
\bigl.
-{\mathcal G}_{s,i}(\nu)[{\mathcal G}_{i,i}(\lambda){\mathcal G}_{s,s}(\mu)+{\mathcal G}_{i,i}(\mu){\mathcal G}_{s,s}(\lambda)]
\bigr)+
\nn\\
&&+\frac{1+\lambda\nu}{1-\lambda\nu}\bigl(
{\mathcal G}_{s,i}(\nu)[{\mathcal G}_{i,i}(\lambda){\mathcal G}_{s,s}(\mu)+{\mathcal G}_{i,i}(\mu){\mathcal G}_{s,s}(\lambda)]
-\bigr.
\nn\\
&&\qquad\qquad\qquad\qquad\qquad
\bigl.
-{\mathcal G}_{i,s}(\lambda)[{\mathcal G}_{i,i}(\nu){\mathcal G}_{s,s}(\mu)+{\mathcal G}_{i,i}(\mu){\mathcal G}_{s,s}(\nu)]
\bigr)+
\nn\\
&&+\frac{1+\mu\nu}{1-\mu\nu}\bigl(
-{\mathcal G}_{s,i}(\nu)[{\mathcal G}_{i,i}(\lambda){\mathcal G}_{s,s}(\mu)+{\mathcal G}_{i,i}(\mu){\mathcal G}_{s,s}(\lambda)]
+\bigr.
\nn\\
&&\qquad\qquad\qquad\qquad\qquad
\bigl.
+{\mathcal G}_{i,s}(\mu)[{\mathcal G}_{i,i}(\nu){\mathcal G}_{s,s}(\lambda)+{\mathcal G}_{i,i}(\lambda){\mathcal G}_{s,s}(\nu)]
\bigr)+
\nn\\
&&+\epsilon(s-i)\bigl(
[{\mathcal G}_{s,i}(\lambda)+{\mathcal G}_{i,s}(\lambda)]
[{\mathcal G}_{i,i}(\nu){\mathcal G}_{s,s}(\mu)+{\mathcal G}_{i,i}(\mu){\mathcal G}_{s,s}(\nu)]\bigr.-
\nn\\
&&-\biggl.\epsilon(s-i)
[{\mathcal G}_{s,i}(\mu)+{\mathcal G}_{i,s}(\mu)][{\mathcal G}_{i,i}(\nu){\mathcal G}_{s,s}(\lambda)
+{\mathcal G}_{i,i}(\lambda){\mathcal G}_{s,s}(\nu)]
\bigr)\biggr\}.
\nn
\eea
In this expression, the term proportional to the product of two $\epsilon$-functions gives zero under the
cyclic permutation, all the terms proportional to a single $\epsilon$-function are mutually canceled as well
as do all the terms proportional to the products of $\lambda$-factors, so the result is zero, as expected.

{\bf 3}. When five or six indices coincide, the Jacobi identity is satisfied identically because
$$
\{{\mathcal G}_{s,s}(\lambda),{\mathcal G}_{s,s}(\mu)\}\equiv 0.
$$

We have therefore proved the satisfaction of the Jacobi identities for all cases of indices distribution
in the formula (\ref{eq:Yangian}).

\def\thetheorem{C.\arabic{theorem}}
\def\theprop{C.\arabic{prop}}
\def\thelemma{C.\arabic{lm}}
\def\thecor{C.\arabic{cor}}
\def\theexam{C.\arabic{exam}}
\def\theremark{C.\arabic{remark}}
\def\theequation{C.\arabic{equation}}

\appendix{Proof of Proposition~\ref{th-braid}}\label{ap:algebraic-independence}

The proof of the first statement of the proposition is an obvious consequence of
the fact that $\mathcal G(\lambda) =\frac{1}{\lambda^p-1}  \mathcal
G_p(\lambda)$. The fact that the coefficients of
$\det\left({\mathcal G}_p(\lambda)\right)$ are central elements is
an obvious consequence of the braid group action. To prove that
there are exactly $\left[\dfrac{np}{2}\right]$ algebraically
independent central elements, let us write ${\mathcal G}_p(\lambda)$
in terms of the variable $u$ where $\lambda=u^2$. Then, up to
irrelevant multiplier $u^p$,
\be
{\mathcal G}_p(u):={\mathcal A}^{(0)}u^p+\sum_{k=1}^{p-1}{\mathcal G}^{(k)}u^{p-2k}+
{{\mathcal A}^{(0)}}^Tu^{-p},
\label{Gp-u}
\ee
and the symmetry ${\mathcal G}_p(u)=\bigl({\mathcal G}_p(u^{-1})\bigr)^T$ becomes obvious, so
$\det {\mathcal G}_p(u)=\det {\mathcal G}_p(u^{-1})$. We then must prove that elements with
nonnegative powers of $u$ in the expansion of this determinant (except the highest term that is
just $1\cdot u^{pn}$) are algebraically independent. For this, let us consider the form
$$
d \det {\mathcal G}_p(u)=\det {\mathcal G}_p(u)\tr \bigl({\mathcal G}^{-1}_p(u)d {\mathcal G}_p(u)\bigr)
$$
in the vector space of the differentials $dG^{(k)}_{i,j}$
(with the constraint (\ref{level-p-G}) imposed) at the special point at which all $G^{(k)}_{i,j}\equiv 1$.

The matrix ${\overline{\mathcal G}}_p(u)$ for all
$G_{i,j}^{(k)}$ equal to unity has the entries
$u^p+u^{p-2}+\cdots+u^{2-p}+u^{-p}$ on the diagonal, the entries
$u^p+u^{p-2}+\cdots+u^{2-p}$ above the diagonal, and the entries
$u^{p-2}+\cdots+u^{2-p}+u^{-p}$ below the diagonal. In fact, it is
not difficult to find $\det{\overline {\mathcal
G}}_p(u)\cdot{\overline {\mathcal G}}_p^{-1}(u)$. This is the matrix
with all the diagonal terms equal to
$u^{p(n-1)}+u^{p(n-1)-2}+\cdots+u^{-p(n-1)+2}+u^{-p(n-1)}$, with the
$\{i,j\}$ entry above the diagonal equal to
$-u^{p(n-2|j-i|)}(u^p+u^{p-2}+\cdots+u^{2-p})$, and with the
$\{i,j\}$ entry below the diagonal equal to
$-u^{-p(n-2|i-j|)}(u^{p-2}+\cdots+u^{2-p}+u^{-p})$.
We introduce the
standard scalar multiplication on the linear space of differentials
(taking into account the symmetry (\ref{level-p-G})):
$$
\left\langle
dG_{i,j}^{(k)}\Bigm|dG_{l,m}^{(s)}\right\rangle=\delta_{i,l}\delta_{j,m}\delta_{k,s}+\delta_{i,m}\delta_{j,l}\delta_{k,p-s},\
\hbox{for}\ i\ne j\ \hbox{or}\ s\ne p/2.
$$
We then segregate the coefficients standing by nonnegative powers of
$u$ and $v$ in the bilinear form:
$$
\Bigl\langle\det {\overline {\mathcal G}}_p(u)\tr
\bigl({\overline{\mathcal G}}^{-1}_p(u)d {\mathcal
G}_p(u)\bigr)\Bigm| \det {\overline {\mathcal G}}_p(v)\tr
\bigl({\overline{\mathcal G}}^{-1}_p(v)d {\mathcal
G}_p(v)\bigr)\Bigr\rangle.
$$
(We perform the calculations for even $n$ and odd $p$, other cases can be treated analogously.)
The bilinear form of variations of central elements (its
$\{i,j\}$ entries are the coefficients of $n\cdot u^{pn-2i}v^{pn-2j}$,
$i,j=1,\dots,pn/2$ in the above expression) is the sum of the following four $(np/2)\times
(np/2)$ matrices (the first two of them come from the brackets between differentials of nondiagonal
entries and the last two arise from the brackets between diagonal term differentials)
$$
{\psset{unit=0.6}
\begin{pspicture}(-7,-5)(7,5.5)
\newcommand{\FRAME}{%
\psframe[linewidth=0.5pt, fillstyle=solid,
fillcolor=white](-3,-2.7)(3,2.7)
}
\newcommand{\EMPTYFRAME}{%
\psframe[linecolor=white, fillstyle=solid,
fillcolor=white](1.8,-1.5)(7.8,3.5)
\pcline[linewidth=0.5pt](1.8,-1.5)(7.8,-1.5)
\pcline[linewidth=0.5pt](7.8,3.5)(7.8,-1.5) } \rput(2.5,-2){\FRAME}
\rput(-3.2,-2.2){\EMPTYFRAME} \rput(0.0,0.0){\FRAME}
\rput(-1.4,.9){\FRAME} \rput(-2.7,2.1){\FRAME}
\rput(0,0){\makebox(0,0)[cc]{$\small\left(\begin{array}{ccccccccc}
                               1 & 1 & 1 & \dots & 1 & 0& 0& \dots &0\\
                                1 & 2 & 2 & \dots & 2 & 1&0& \ddots &\vdots \\
                                1 & 2 & 3 & \dots & 3 & 2& 1&\ddots& 0 \\
                                \vdots & \vdots & \vdots & \ddots & \vdots & 3& 2& \ddots& 0 \\
                                1 & 2 & 3 & \dots & p & \vdots &3 &\ddots &1 \\
                                0 & 1 & 2 & 3 & \dots & p & \vdots & \ddots &2 \\
                                0 & 0 & 1 & 2 & 3 & \dots & p &\ddots & 3 \\
                                \vdots & \ddots & \ddots & \ddots & \ddots & \ddots &\ddots &\ddots &\vdots \\
                                0 & \dots & 0 & 0 & 1 & 2 &3 &\dots &p \\
                              \end{array}\right),$}}
\end{pspicture}
}
$$
$$
{\psset{unit=0.6}
\begin{pspicture}(-7,-4)(7,4.5)
\newcommand{\FRAME}{%
\psframe[linewidth=0.5pt, fillstyle=solid,
fillcolor=white](-3,-2.7)(3,2.7)
}
\newcommand{\EMPTYFRAME}{%
\psframe[linecolor=white, fillstyle=solid,
fillcolor=white](1.8,-1.5)(7.8,3.5)
}
\rput(2,-1.7){\FRAME}
\rput(-3.5,-1.9){\EMPTYFRAME}
\def\udots{\kern-.2em
  \lower1ex\hbox{$\cdot$}\cdot\kern0em\lower-1ex\hbox{$\cdot$}
  }
\rput(0,0){\makebox(0,0)[cc]{$\small\left(\begin{array}{cccccccc}
                               0 & \dots & \dots & \dots & \dots & \dots& \dots &0\\
                                \vdots & & &   & && \udots &\vdots \\
                                \vdots & & & & & \udots & & 0 \\
                                \vdots &  &  &  & \udots  & &0  &1 \\
                                \vdots &  &   & \udots &  & \udots & 1 &2 \\
                                \vdots &  & \udots &   & \udots & \udots &\udots & 3 \\
                                \vdots & \udots & &0 & 1 & 2& \udots & \vdots \\
                                0 & \dots & 0  & 1 & 2 &3 &\dots &p \\
                              \end{array}\right)$}}
\end{pspicture}
}
$$
$$
{\psset{unit=0.6}
\begin{pspicture}(-5,-3.6)(11,4)
\newcommand{\FRAME}{%
\psframe[linewidth=0.5pt, fillstyle=solid,
fillcolor=white](-2.2,-1.5)(2.2,1.5)
}
\newcommand{\EMPTYFRAME}{%
\psframe[linecolor=white, fillstyle=solid,
fillcolor=white](1.8,-1.5)(7.8,3.5)
}
\rput(-2.3,1.1){\FRAME}
\def\udots{\kern-.2em
  \lower1ex\hbox{$\cdot$}\cdot\kern0em\lower-1ex\hbox{$\cdot$}
  }
\rput(0,0){\makebox(0,0)[cc]{$\small\left(\begin{array}{ccccccc}
                               1 & 1 & \dots & 1 & 1 & \dots & 1 \\
                               1 & 2 & \dots & 2 & 2 & \dots & 2 \\
                               \vdots &\vdots  &\ddots  & \vdots & \vdots &  &\vdots \\
                               1 & 2 & \dots & p{-}1 & p{-}1 & \dots & p{-}1 \\
                               1 & 2 & \dots & p{-}1 & p{-}1 & \dots & p{-}1 \\
                               \vdots &\vdots  &  & \vdots & \vdots & \ddots  &\vdots \\
                               1 & 2 & \dots & p{-}1 & p{-}1 & \dots & p{-}1 \\
                              \end{array}\right)$}}
\end{pspicture}
}
{\psset{unit=0.6}
\begin{pspicture}(7,-3.6)(0,0)
\newcommand{\FRAME}{%
\psframe[linewidth=0.5pt, fillstyle=solid,
fillcolor=white](-2.5,-1.7)(2.5,1.7)
}
\newcommand{\EMPTYFRAME}{%
\psframe[linecolor=white, fillstyle=solid,
fillcolor=white](1.8,-1.5)(7.8,3.5)
}
\rput(-2.4,1.15){\FRAME}
\def\udots{\kern-.2em
  \lower1ex\hbox{$\cdot$}\cdot\kern0em\lower-1ex\hbox{$\cdot$}
  }
\rput(0,0){\makebox(0,0)[cc]{$\small\left(\begin{array}{ccccccc}
                               0 & \dots & 0 & 1 & 1 & \dots & 1 \\
                               \vdots & \udots & \udots & 2 & 2 & \dots & 2 \\
                               0 & 1  &\udots  & \vdots & \vdots &  &\vdots \\
                               1 & 2 & \dots & p{-}1 & p{-}1 & \dots & p{-}1 \\
                               1 & 2 & \dots & p{-}1 & p{-}1 & \dots & p{-}1 \\
                               \vdots &\vdots  &  & \vdots & \vdots & \ddots  &\vdots \\
                               1 & 2 & \dots & p{-}1 & p{-}1 & \dots & p{-}1 \\
                              \end{array}\right)$}}
\end{pspicture}
}
$$
We now perform the following row and column operations: first, we subtract the first row from $p-1$ subsequent rows,
then subtract the second row from $p-1$ subsequent rows, etc. Then, we perform the same operation with columns: subtract
the first column from $p-1$ subsequent columns, etc. The first matrix then becomes just the $(np/2)\times (np/2)$ unit matrix,
the second matrix will contain just one nonzero $p\times p$ block in the lower right corner: in this block we obtain the
matrix with unities on the antidiagonal and zeros elsewhere; the third and fourth matrices become $(np/2)\times (np/2)$-matrices
composed from $n^2/4$ equal $p\times p$ blocks: these blocks are
$$
\def\udots{\kern-.2em
  \lower1ex\hbox{$\cdot$}\cdot\kern0em\lower-1ex\hbox{$\cdot$}
  }
\small\left(\begin{array}{ccccc}
                               1 & 0 & \dots & 0 & 0 \\
                               0 & 1 & \ddots & 0 & 0 \\
                               0 & 0  &\ddots  & \ddots & \vdots \\
                               \vdots & \vdots & \ddots & 1 & 0 \\
                               0 & 0 & \dots & 0 & 0 \\
                              \end{array}\right),
\ \hbox{and}\
\left(\begin{array}{ccccc}
                               0  & \dots & 0 & 1 & 0 \\
                               \vdots & \udots & \udots & 0 & 0 \\
                               0 & 1  &\udots  & \udots & \vdots \\
                               1 & 0  &\dots  & 0 & 0 \\
                               0 & 0 & \dots & 0 & 0 \\
                              \end{array}\right)
$$
for the respective third and fourth matrices. We then subtract, first, the upper $p\times (np/2)$ row-block from all the lower blocks
and then subtract the last $(np/2)\times p$ column-block from all the preceding ones. In the last two matrices only the upper-right
$p\times p$ block remains nonzero whereas in the sum of the first two matrices we have nonzero $p\times p$ blocks on the diagonal and
on the bottom and left sides. In the obtained matrix, we perform the following recurrent procedure: expand over $(p-1)$st row (only
$E_{p-1,p-1}$ entry is nonzero), then subtract the very last row from the $p$th row from the bottom and expand over the obtained
$p$th line from the bottom (only the very first entry, $E_{1,pn/2-p+1}$, is nonzero). Then subtract the first line from the $(p-1)$th
row and expand, first, over the very last column (only the entry $E_{np/2,np/2}$ is nonzero) and then over $pn/2-p+1$th column
(only the first entry, $E_{pn/2-p+1,1}$, is nonzero). Eventually, expand over the first and the last rows in all the intermediate blocks
(only entries on the diagonal are nonzero). After this chain of operations, we come to the block-diagonal matrix, which has the same
form as the initial one, but has all blocks of size $(p-2)\times (p-2)$ instead of $p\times p$. Continuing this procedure (in the case
of odd $p$, we come on the last stage to the block matrix with blocks of size $1\times1$. The upper right block of such a matrix is zero,
the lower-right block is $2$, all the diagonal blocks are $1$, and the matrix is lower-triangular, therefore its determinant is $2$.
The proposition is proved.


\end{document}